\documentclass[11pt, reqno]{amsart}
\usepackage{graphicx} % Required for inserting images
\usepackage{amsmath, amssymb, amsthm, amsfonts, bbm, mathtools, mathrsfs, thmtools, mathdots}
\usepackage{enumerate}
\usepackage{diagbox}
\usepackage{fullpage}

\usepackage{geometry}
\usepackage{comment}
\usepackage{verbatim}
\includecomment{note}
\specialcomment{note}{\begingroup\color{red}}{\endgroup}
\includecomment{noteb}
\specialcomment{noteb}{\begingroup\color{blue}}{\endgroup}
\excludecomment{noteb}
\includecomment{notec}
\specialcomment{notec}{\begingroup\color{blue}}{\endgroup}

\usepackage[colorlinks, pagebackref, linkcolor=red!50!black]{hyperref}
\usepackage[lite,abbrev,msc-links,alphabetic]{amsrefs}
\usepackage{tikz}
\usetikzlibrary{positioning}

\usepackage{tikz-cd}

\numberwithin{equation}{subsection}

\theoremstyle{plain}
\newtheorem{proposition}{Proposition}[subsection]

\newtheorem{corollary}[proposition]{Corollary}
\newtheorem{lem}[proposition]{Lemma}
\newtheorem{theorem}[proposition]{Theorem}
\newtheorem{example}[proposition]{Example}

\theoremstyle{definition}
\newtheorem{definition}[proposition]{Definition}

\newtheorem{remark}[proposition]{Remark}

\newcommand{\bs}{\backslash}
\newcommand{\wt}{\widetilde}

\newcommand{\BA}{\mathbb A}
\newcommand{\BC}{\mathbb C}

\newcommand{\BR}{\mathbb R}

\newcommand{\BQ}{\mathbb Q}

\newcommand{\BL}{\mathbb L}

\newcommand{\CS}{\mathcal{S}}

\renewcommand{\AA}{\mathbb A}

\newcommand{\CC}{\mathbb C}
\newcommand{\cH}{\mathcal H}

\newcommand{\gl}{\mathfrak{gl}}
\newcommand{\fh}{\mathfrak h}
\newcommand{\fm}{\mathfrak m}
\newcommand{\fp}{\mathfrak p}

\newcommand{\fs}{\mathfrak s}
\newcommand{\fu}{\mathfrak u}
\newcommand{\fx}{\mathfrak x}
\newcommand{\fy}{\mathfrak y}

\newcommand{\LL}{\mathbb L}
\renewcommand{\O}{\mathcal O}
\newcommand{\ol}{\overline}

\renewcommand{\P}{\mathcal P}
\renewcommand{\S}{\textnormal{S}}
\newcommand{\cS}{\mathcal S}
\newcommand{\W}{\mathcal W}

\DeclareMathOperator{\As}{As}

\DeclareMathOperator{\End}{End}
\DeclareMathOperator{\Gal}{Gal}
\DeclareMathOperator{\GL}{GL}

\DeclareMathOperator{\Hom}{Hom}

\DeclareMathOperator{\Mat}{Mat}

\DeclareMathOperator{\Orb}{Orb}
\DeclareMathOperator{\ord}{ord}
\DeclareMathOperator{\Res}{Res}
\DeclareMathOperator{\SHerm}{SHerm}
\DeclareMathOperator{\Stab}{Stab}
\DeclareMathOperator{\Supp}{Supp}
\DeclareMathOperator{\Tr}{Tr}
\DeclareMathOperator{\U}{U}

\DeclareMathOperator{\Vol}{Vol}

\newcommand{\FJ}{\textnormal{FJ}}
\newcommand{\good}{\textnormal{good}}
\newcommand{\inert}{\textnormal{in}}

\renewcommand{\Re}{\textnormal{Re}}
\newcommand{\simrightarrow}{\xrightarrow{\raisebox{-0.7ex}[0ex][0ex]{$\sim$}}}
\newcommand{\transp}{\prescript{\top}{}{}}

\newcommand{\rss}{\textnormal{rss}}

\newcounter{keepeqno}

\newenvironment{num}
 {\setcounter{keepeqno}{\value{equation}}%
  \begin{list}{(\theequation)}{\usecounter{equation}}%
  \setcounter{equation}{\value{keepeqno}}}
 {\end{list}}

\title[Twisted GGP conjecture]{Central values of Asai L-functions and twisted Gan--Gross--Prasad conjecture}

\author[Weixiao Lu]{Weixiao Lu}
\address{Department of Mathematics, Massachusetts Institute of Technology}
\email{weixiao.LU@univ-amu.fr}

\author[Danielle Wang]{Danielle Wang}
\address{Department of Mathematics, Duke University}
\email{danielle.wang@duke.edu}

\author[Zhiyu Zhang]{Zhiyu Zhang}
\address{Department of Mathematics, Stanford University}
\email{zyuzhang@stanford.edu}

\date{\today} 
\begin{document}
\maketitle

\begin{abstract}
We study certain new relative trace formulas on (non-reductive) period integrals involving Weil representations, in the context of the relative Langlands program. We study normal representatives using Galois theory, and establish geometric decompositions of relative trace formulas using normal representatives for good test functions. By comparing global representatives, local distributions and orbital integrals, we prove the twisted Gan--Gross--Prasad (GGP) conjecture on Asai L-functions, in any dimension under some local assumptions, allowing ramifications of number fields.
\end{abstract}

\setcounter{tocdepth}{1}
\tableofcontents

\section{Introduction}

The Waldspurger formula \cite{waldspurger1985valeurs} gives a striking relation between twisted versions of Hecke periods and central L-values of modular forms. The choice of quadratic twists is crucial for many applications. See \cite{duke1988hyperbolic, duke1990representation} for its applications to sum of three squares and Heegner points. In this paper, we establish a relation between twisted versions of Rankin--Selberg period integrals and L-functions for twisted tensor products of automorphic forms, predicted by the twisted Gan--Gross--Prasad conjecture \cite{TwistedGGP}. Similarly, we allow the choice of a general quadratic twist over the base number field.

The study of period integrals and L-functions is a central topic in number theory and representation theory, nowadays in the framework of the relative Langlands program, see e.g. \cite{BZSV}. A refined fundamental example is the Gan--Gross--Prasad (GGP) conjecture \cite{GGP-conjectures} for classical groups, with Ichino--Ikeda type refinements \cite{GGP-conjecture-refinements}. There has been recent significant progress:
\begin{enumerate}
    \item proof of GGP conjecture on Bessel periods \cite{GGP-stable,GGP-endoscopic} and Fourier--Jacobi periods \cite{GGP-FourierJacobi} for unitary groups, based on the relative trace formula approach of Jacquet and Rallis by comparisons with Rankin-Selberg periods.
    \item formulations and results on arithmetic analogs and applications to Beilinson--Bloch--Kato conjecture on Rankin-Selberg ($\GL_n \times \GL_{n+1}$) motives  \cite{LTXZZ, liu2024iwasawa, disegni2024gan}. 
\end{enumerate}

See \cite{Zhang2018ICM,gross2022road,BP2022ICM,AGGP-survey-2024} for nice summaries. The relative trace formula is a powerful geometric tool for studying period integrals over number fields. A key point is that the non-vanishing of period integrals is often equivalent to the non-vanishing of a spectral distribution of test functions. On the one hand, this distribution can often be packaged together (over nice automorphic representations) into a relative trace formula, with geometric meaning and a decomposition into sums of products of local orbital integrals of test functions. Then we may establish a geometric comparison of two relative trace formulas, by solving local harmonic analysis questions on comparison of orbital integrals, e.g. fundamental lemma and smooth transfer, and global questions on automorphic spectrum. On the other hand, this distribution often admits a spectral decomposition into products of local relative characters of test functions, and Langlands functoriality predicts a relation between automorphic representations on different groups. Thus by proving a comparison of geometric decompositions, we deduce the GGP conjecture under some local assumptions. And comparisons of local relative characters and the local GGP conjecture lead to refined equalities between period integrals and L-functions.

However, we do not have such a refined theory of relative trace formulas in general. This paper concerns about Asai L-functions \cite{Asai1977,Flicker}, which occur naturally in the study of
\begin{enumerate}
    \item twisted symmetric square L-functions of modular forms.
    \item automorphic base changes for unitary groups and general linear groups.
    \item classical Yoshida liftings to $\mathrm{GSp}_4$.
    \item Hilbert modular surfaces and unitary Shimura varieties.
\end{enumerate}

Let $F$ be a number field, $V$ be a skew-Hermitian space over a quadratic extension $E/F$, and $K/F$ be another quadratic extension. This paper proves the twisted GGP conjecture \cite{TwistedGGP} under local assumptions, relating twisted Asai L-function (for $K/F$) and (non-polarizable) period integrals for the embedding $H=\U(V) \to G=\Res_{K/F} \U(V_K)$ and the Weil representation of $H$ on Lagrangians of $V$. We summarize our main results:

\begin{enumerate}
    \item new relative trace formulas and their geometric decompositions into orbital integrals after choosing representatives. Roughly we have (see Section \ref{section: decomp and orbits} for details) 
\[
\mathrm{RTF}_{(G, H,\rho)}(f, \phi_1, \phi_2)=\sum_{a \in [H \backslash (R \times V)](F)_{\mathrm{rs}}} \mathrm{Orb}(a, f, \phi_1, \phi_2).
\]
Here $R$ is a nice representative of $G/H$, and we transform a product of two theta series into one theta series via partial Fourier transforms \cite{Li1992,Liu2014relative,Xue-GGP}.
    \item a proof of the twisted Gan--Gross--Prasad conjecture \cite{TwistedGGP} under certain local assumptions, using such decompositions and comparisons of orbital integrals. 
\end{enumerate}

\subsection{Relative trace formulas in polarized reductive cases}

A good way to study the non-vanishing of period integrals is via relative trace formulas (RTF). Let us recall the classical set up. For simplicity, we will ignore any convergence and rationality issue. Let $G$ be a reductive group over $F$. Let $\BA=\BA_F$ be the ring of adeles. Denote by $[G]=G(F) \backslash G(\BA)$ the adelic quotient with suitable Haar measure.

In the polarized reductive case, we consider a smooth affine variety $X=G/H$ where $H \leq G$ a reductive subgroup over $F$. Let $\pi$ be a cuspidal automorphic representation of $G(\mathbb A)$. We define the period integral $P_X(\varphi) = \int_{[G]} \varphi(h) dh$ for $\varphi \in \pi$. Consider the relative trace formula
\[
\mathrm{RTF}_{X}(f):= \int_{[H] \times [H]} K_f(h_1, h_2)  dh _1 dh_2, \quad f \in \CS(G(\BA)).
\]\
Here $K_f(x, y)=\sum_{\gamma \in G(F)} f(x^{-1}\gamma y)$ is the kernel function on $[G] \times [G]$ for the right translation action $R(f)$ on the spectrum $L^2([G])$. When the center of $G$ is anistropic, the kernel function for $R(f)$ on $\pi \subseteq L^2_{\mathrm{cusp}}([G])$ is 
\[
K_{f,\pi}(x,y)= \sum_{\phi \in \mathrm{OB}(\pi)} (R(f)\phi(x)) \overline{\phi(y)}
\]
where $\mathrm{OB}(\pi)$ is an orthonormal basis of $\pi$. Then
\[
\mathrm{RTF}_{X,\pi}(f):= \int_{[H] \times [H]} K_{f,\pi}(h_1, h_2)  dh_1 dh_2 = \sum_{\phi \in \mathrm{OB}(\pi)} P_X(R(f)\phi) \overline{P_X(\phi)}
\]
is directly related to the non-vanishing of $P_X$ on $\pi$, and appear as a direct summand of $\mathrm{RTF}_{X}(f)$. 

The key for the relative trace formula approach is a geometric decomposition of the relative trace formula into orbital integrals (for a good test function $f$ with regular support)
\[
\mathrm{RTF}_{X}(f)=\sum_{\gamma \in [H(F)\backslash G(F) / H(F)]_{rs}} \Vol([(H \times H)_\gamma]) \Orb(\gamma, f)
\]
where we consider orbits and orbital integrals for the natural action of $H \times H$ on $G$. And $\Vol([(H \times H)_\gamma])$ is the volume of adelic quotient of the stabilizer $(H \times H)_\gamma$ over $F$. When $f=\otimes_v f_v$ is a pure tensor and the stabilizer of $\gamma$ is trivial, we have a decomposition:
\[
\Orb(\gamma, f)= \prod_v \Orb_v(\gamma, f_v)
\]
where $v$ runs over all places of $F$, $\Orb_v(\gamma, f_v)=\int_{(H \times H)(F_v)} f_v(h_1^{-1} \gamma h_2) dh_1 dh_2$ is the local orbital integral at $v$. By local comparison of orbital integrals (via fundamental lemmas and smooth transfers) for $(G, X)$ and $(G', X')$, we can compare two relative trace formulas and transfer period integrals from $G$ to $G'$, hence deduce numerical relations of $L$-functions and period integrals for $(G, X)$ from $(G', X')$. A good example is the original proof of the Gan--Gross--Prasad conjecture for unitary groups by W. Zhang \cite{WeiIchinoIkeda14,WeiGGP14} under certain local assumptions. 

However, a general study of relative trace formulas for non-reductive period integrals (e.g., those involving Weil representations) is unclear. For the Fourier--Jacobi GGP conjecture, \cite{GGP-FourierJacobi} uses RTFs to compare with Rankin-Selberg L-functions, where the orbital integrals could be naturally defined. However, the essential non-splitting feature makes twisted GGP conjecture \cite{TwistedGGP} and Asai L-functions more mysterious and difficult, which is a major motivation of this paper. For comparison, we discover and utilize more complicated geometric decompositions.

\subsection{Results on the twisted GGP conjecture}

Let us explain our results in more detail. Let $E/F$ and $K/F$ be quadratic
extensions of number fields. Let $L = E \otimes_F K$. Let $(V, \langle -, - \rangle)$ be an $n$-dimensional
$E/F$ skew-Hermitian space. Choose a pure imaginary element $j \in E$, then we could transform between skew-Hermitian spaces and Hermitian spaces as in \cite[Remark 13.1]{GGP-FourierJacobi}.

Consider embeddings of reductive groups over $F$:
\[
H=\U(V) \to G=\Res_{K/F} \U(V_K).
\]

We view $\Res_{E/F} V^\vee$ as a symplectic space over $F$ with the symplectic form 
$\Tr_{E/F}(\langle -, - \rangle)$. Let $\Res_{E/F} V^\vee=\LL + \LL^\vee$ be a polarization of Lagrangians, and $\psi$ be a non-trivial additive character of $\AA / \mathbb F$. 
Let $\omega=\omega_{\psi, \mu}$ be the Weil representation
of the unitary group $\U(V)(\AA)$, which is realized on
$\cS(\LL(\AA))$.
 For $\phi \in \cS(\LL(\AA))$, define the theta series (which is left $H(F)$-invariant)
\[
    \Theta(h,\phi) = 
    \sum_{x \in \LL(F)} (\omega(h)\phi)(x).
\]

Let $\pi$ be a cuspidal automorphic representation of $\U(V_K)(\AA_K)$, consider the period integral
\begin{equation}
  \P(f,\phi) = \int_{[\U(V)]}
         f(h)\ol{\Theta(h,\phi)}\, dh  
\end{equation}
for $f \in \pi$, $\phi \in \cS(\LL(\AA))$. Here $[\U(V)]=\U(V)(F) \backslash \U(V)(\AA)$.\

We say a cuspidal automorphic representation $\Pi$ of $\GL_n(\AA_L)$ is \emph{Hermitian} if
\begin{itemize}
    \item $\Pi$ is conjugate self dual (i.e. $\Pi^\vee \cong \Pi^c$, where $\Pi^c$ is defined by $\Pi^c(g)=\Pi(g^c)$).
    \item The Asai $L$-function $L(s,\Pi,\operatorname{As}_{L/K}^{(-1)^{n+1}})$ has a pole at $s=1$.
\end{itemize}

In general, suppose $V$ and $V'$ are two skew-Hermitian spaces of dimension $n$ over $E$. For almost all places $v$
of $F$, we have $\U(V)(K_v) \simeq \U(V')(K_v)$. We say that a representation
$\pi$ of $\U(V)(\AA_K)$ and a representation
$\pi'$ of $\U(V')(\AA_K)$ are \emph{nearly equivalent}
if $\pi_v \simeq \pi_v'$ for all but finitely many
places $v$. If the weak base change of $\pi$ and $\pi'$ to $\GL_n(\BA_L)$ agree, 
$\pi$ and $\pi'$ are nearly equivalent. 

We prove the global twisted GGP conjecture \cite{TwistedGGP} under some local assumptions, which says that the non-vanishing of $L(\frac12, \pi, \As_{L/E} \otimes \mu^{-1})$ is equivalent to the non-vanishing of above period integral $\mathcal{P}$ on some $\pi'$ that is nearly equivalent to $\pi$. 

For any extension $L_1/L_2$ of number fields, we denote by $\mathrm{Spl}(L_1/L_2)$, $\mathrm{Inert}(L_1/L_2)$, $\mathrm{Ram}(L_1/L_2)$ the set of finite places of $L_2$ which is completely split, inert, ramified in $L_1$ respectively.

Let $\Pi$ be a Hermitian cuspidal tempered automorphic representation of $\GL_n(\AA_L)$. Let $\mathrm{Ram}_F(\Pi)$ be the set of finite place $v$ of $F$ where $\Pi$ is ramified at some place of $L$ above $v$. 

\begin{restatable}{theorem}{mainthm}
\label{thm:mainthm} 
Assume that
\begin{enumerate}
    \item \label{item:fieldsunr}
   $\mathrm{Ram}(E/F) \subseteq \mathrm{Spl}(K/F)$, $\mathrm{Ram}(K/F) \subseteq \mathrm{Spl}(E/F)$.
    \item \label{item:charsunr}
    For $v \in \mathrm{Inert}(E/F) \cap \mathrm{Inert}(K/F)$, $\mu_v$, $\psi_{v}$ are unramified 
    and $j \in \O_{E,v}^\times$.
    \item Any archimedean and $2$-adic place of $F$ is split in $L$.
    \item \label{item:piunr}
   $\mathrm{Ram}_F(\Pi) \cap \mathrm{Inert}(E/F) \cap \mathrm{Inert}(K/F) = \emptyset$.
    \item There exist $v_1 \in \mathrm{Spl}(E/F), v_2 \in \mathrm{Spl}(L/F)$ such that $\Pi$ is supercuspidal at all places above $v_1$ and $v_2$.
\end{enumerate}
Then the following are equivalent as predicted by \cite{TwistedGGP}:
\begin{enumerate}
    \item The Asai $L$-value $L(\frac 12,\Pi,\operatorname{As}_{L/E} \times \mu^{-1}) \ne 0$.
    
    \item There exists a skew-Hermitian space $V$ of
    dimension $n$ over $E$, and a cuspidal automorphic 
    representation $\pi$ of $\U(V)(\AA_K)$ such that $\operatorname{BC}(\pi) = \Pi$, and the period
    integral  $\P(f,\phi) = \int_{[\U(V)]}
         f(h)\ol{\Theta(h,\phi)}\, dh\ne 0$
for some $f \in \pi$, $\phi \in \cS(\LL(\AA))$.
\end{enumerate}
\end{restatable}

If $E/F$ is fixed with split $2$-adic places and archimedean places, under a suitable choice of $\psi$ and $\mu$, we can find infinitely many examples of such $K$ satisfying $(1)(2)(3)$. In this way, we may view the twisted GGP conjecture as a family of relations between period integrals and L-functions. Such relations \cite{TwistedGGP} would be useful to study analytic and arithmetic properties of Asai L-functions, including twisted triple product L-functions for Hilbert modular forms, see e.g. \cite{liu2016hirzebruch,grossi2024asai}. 

This global conjecture for Asai L-functions in \cite{TwistedGGP} is not known previously, even for $n=2$. We plan to study the refined conjecture \cite{TwistedGGP} further in the future.

One essential difficulty in the proof is to find correct relative trace formulas with nice geometric properties. Our new geometric decomposition (based on normal representatives) for good test functions relate these relative trace formulas to local orbital integrals involving partial Fourier transforms, which are hard to work with in general. Under local assumptions, we manage to relate these orbital integrals to more standard orbital integrals and use previous results as inputs. Another new difficulty is that the definition of smooth transfer only works for good test functions. Moreover, we need to establish new results on local distributions to obtain relative trace identities. Moreover, we use the local twisted GGP conjecture \cite{TwistedGGP}, which is known by \cite{le2025local} in our case.

\subsection{The base change case}

In the base change case, Asai L-functions are related to symmetric square and exterior square L-functions. Let us briefly explain. Let $\Pi_0$ be a cuspidal tempered automorphic representation of $\GL_n(\mathbb A_E)$ with an associated $n$-dimensional $\ell$-adic Galois representation $\rho_\ell: \Gal_{E} \to \GL(V_\rho)$, with a cuspidal base change $\Pi=\mathrm{BC}_{L/E}(\Pi_0)$ to $\GL_n(\mathbb A_L)$. Then the Asai L-function for $\Pi$ is the standard L-function for the Galois representation $\operatorname{As}_{L/E}(\rho|_{L})$. We have an isomorphism of Galois representations 
\[
\operatorname{As}_{L/E}(\rho|_{L}) \cong \mathrm{Sym}^2 \rho \oplus (\wedge^2 \rho \otimes \eta_{L/E}),
\]
where $\eta_{L/E}$ is the quadratic Hecke character of $E$ associated to $L/E$ by global class field theory. In particular, if we fix $E/F$ and let $K$ vary, our main theorem on the twisted GGP conjecture may give new information on central L-values of symmetric square and exterior square L-functions.

\subsection{Proof strategy of the main theorem}

We summarize our proof as follows.

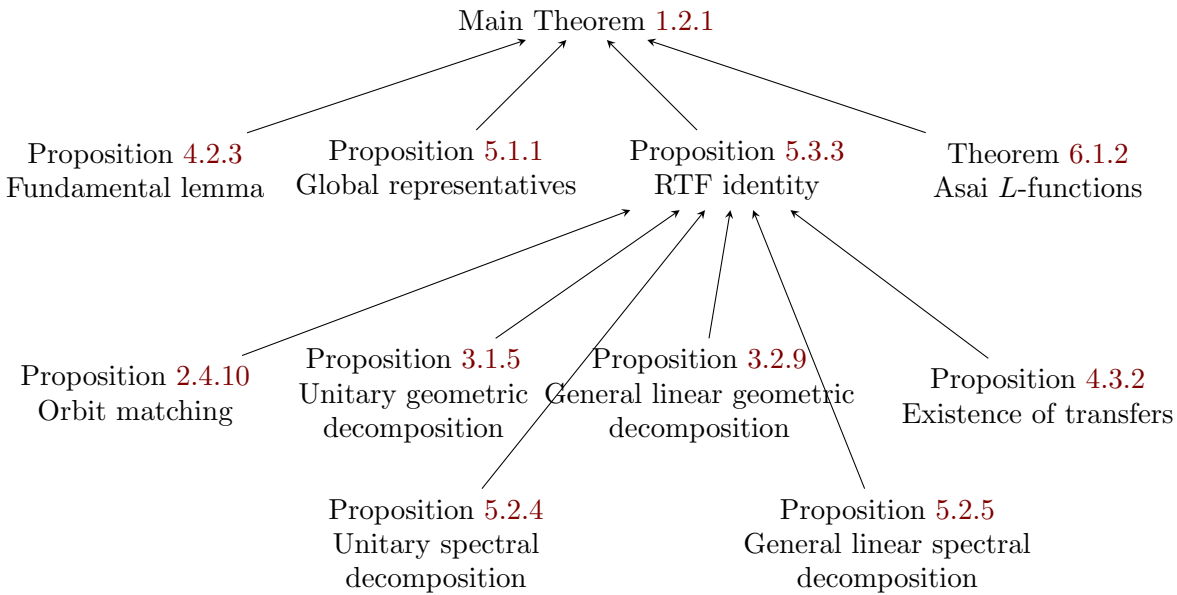
\begin{center}
\begin{tikzpicture}

\node[align=center] (mainthm) at (6, 6) {Main Theorem \ref{thm:mainthm}};
\node[align=center] (fl) at (0, 4) {Proposition \ref{prop:FL} \\ Fundamental lemma};
\node[align=center] (globalreps) at (4, 4) {Proposition \ref{prop:globalkottwitz} \\ Global
representatives};
\node[align=center] (rtfid) at (8, 4) {Proposition \ref{prop:identity} \\ RTF identity};
\node[align=center] (asai) at (12, 4) {Theorem \ref{thm:asai} \\ Asai $L$-functions};
\node[align=center] (uspectral) at (4, -1) {Proposition \ref{prop:uspectral} \\ Unitary spectral \\
decomposition};
\node[align=center] (glspectral) at (10, -1) {Proposition \ref{prop:glspectral} \\ General linear spectral \\ decomposition};
\node[align=center](matching) at (0, 1) {Proposition \ref{lem:matching} \\ Orbit matching};
\node[align=center](udecomp) at (3.7, 1) {Proposition \ref{prop:udecomp} \\ Unitary geometric \\
decomposition};
\node[align=center](gldecomp) at (7.5, 1) {Proposition \ref{prop:gldecomp} \\ General linear geometric \\
decomposition};
\node[align=center](splittransfer) at (12, 1) {Proposition \ref{prop:splittransfer} \\ Existence of transfers};

\draw[-stealth] (matching) -- (rtfid);
\draw[-stealth] (udecomp) -- (rtfid);
\draw[-stealth] (gldecomp) -- (rtfid);
\draw[-stealth] (splittransfer) -- (rtfid);
\draw[-stealth] (uspectral) -- (rtfid);
\draw[-stealth] (glspectral) -- (rtfid);
\draw[-stealth] (fl) -- (mainthm);
\draw[-stealth] (globalreps) -- (mainthm);
\draw[-stealth] (rtfid) -- (mainthm);
\draw[-stealth] (asai) -- (mainthm);
\end{tikzpicture}
\end{center}

\subsubsection{Geometric decomposition of relative trace formulas}

We design new relative trace formulas and compare them via comparisons of orbital integrals, generalizing the approach \cite{Liu2014relative,Xue-GGP,Xue-Ichino,TGGP-Wang} using \cite{JPSS1983}. The non-vanishing of Asai L-functions is naturally related to the non-vanishing of certain Rankin–Selberg type period integral \cite{Flicker}.

The geometric decomposition into orbital integrals is more involved. We use the following tools on the unitary side:
\begin{enumerate}
    \item A partial Fourier transform connecting $\CS(\BL(\mathbb A)) \otimes \CS(\BL(\mathbb A))$ to $\CS(V^\vee(\mathbb A))$.
    \item A reduction formula of products of two theta series as theta series of Weil translations.
    \item A choice of normal representatives of $H(F) \times H(F)$ orbits in $G(F)$. Here we use \cite[Lemma 8.8]{Kottwitz1980}.
\end{enumerate} 
The general linear side is similar, where we consider 
\[
H_1 = \Res_{E/F} \GL_n, \quad G' = \Res_{L/F} \GL_n, \quad H_2 = \Res_{K/F} \GL_n.
\]
And the resulting orbital integrals involve a transfer factor.

 Due to the lack of a natural action of $H \times H$ (resp. $H_1 \times H_2$) on $V^\vee$  (resp. $F_n \times F^{-,n}$), we have to work with normal representatives of $H \times H$-orbits in $G$ (resp. $H_1 \times H_2$-orbits in $G'$) to deal with the involved Weil representation for $H$ (resp. $H_1$), which is essentially different to the the untwisted case \cite{Xue-GGP}. Here we develop a new theory of normal representatives, using a third subfield $M$ in $L=E \otimes_F K$.

\subsubsection{Comparison of representatives, orbital integrals and relative trace formulas}

We compare the geometric sides of two relative trace formulas. The comparison is via a base change from unitary groups to general linear groups (for the quadratic extension $L/K$) using Flicker-Rallis periods. We relate the base change of the Weil representation of $H(\mathbb A)$ on $\CS(\BL(\mathbb A))$ to the Weil representation of $H_1(\mathbb A)$ on $\CS(\mathbb A_E^n)$. Another subtle point is that we need to choose compatible normal representatives with sufficiently nice local behaviors, which is done in Section \ref{section: representatives}. 

Then we reduce the comparison to local comparisons of orbital integrals, namely, fundamental lemmas and the existence of smooth transfers at a place $v$ of $F$. 
\begin{enumerate}
\item If $v$ is archimedean, by assumption $v$ is split in $K$ and we apply Xue's     (partial) smooth transfer theorem.
\item If $v$ is split in $K$, we do reductions to prove fundamental lemmas and smooth transfer via the Fourier--Jacobi case \cite{Xue-GGP}. 
\item If $v$ is inert in $K$, by assumption $v$ is either split or inert in $E$. If $v$ is split in $E$, both sides are the same, so fundamental lemmas and smooth transfer are known. If $v$ is inert in $E$, then we do reductions and use the fundamental lemmas in \cite{TGGP-Wang}.
\item If $v$ is ramified in $K$, then by assumption $v$ is split in $F$, so fundamental lemmas and smooth transfer are reduced to known cases.
\end{enumerate}

Then by choosing good test functions and using the local twisted GGP conjecture \cite{le2025local}, we finish the proof of the twisted GGP conjecture under these local assumptions.

\subsection{Further applications and relative Langlands program}

We mention some potential applications of our results and methods.

\begin{enumerate}
    \item It is important to prove the existence of smooth transfers for all test functions at any place, which would in particular remove local assumptions (1)(3) in Theorem \ref{thm:mainthm}.
    \item Period integrals involving Weil representations provide a large class of new examples, and our geometric decomposition method shall be useful for other examples, as in \cite{BZSV, MWZ24} for split groups. Following \cite{MWZ24}, we may formulate variants for unitary groups and more GGP type conjectures on period integrals and L-functions. It would be interesting to study similar geometric decompositions of relative trace formulas. Besides Whittaker and non-reductive cases, period integrals in \cite{BZSV, MWZ24} are constructed from a tuple $(G, H, \rho, 0)$: here $H \leq G$ is a reductive subgroup and $\rho: H \to \mathrm{Sp}(V)$ is a symplectic representation which induces a Weil representation $\omega$ of $H$ on $\CS(\BL(\BA))$ where $\BL$ is a Langragian of $V$. The Weil representation $\omega$ is independent of the choice of Lagrangian.
For $\phi \in\CS(\BL(\BA))$, we have the $\rho$-theta series on $[H]$
\[
\Theta_{\BL}(h, \phi)=\sum_{x \in \BL(F)}  \omega(h)\phi(x).
\]
For a cuspidal automorphic form $\varphi$ on $G(\BA)$, consider its period integral (we use complex conjugation following \cite[Section 2.2]{TwistedGGP})
\[
P_{(G, H, \rho)}(\varphi, \phi)=\int_{[H]} \varphi(h) \overline{\Theta_{\BL}(h, \phi)} dh.
\]
The relative trace formula on $f \in \CS(G(\BA)), \phi_1, \phi_2 \in \CS(\BL(\BA))$ is 
\[
\mathrm{RTF}_{(G, H,\rho)}(f, \phi_1, \phi_2):= \int_{[H] \times [H]} \sum_{\gamma \in G(F)} f(h_1^{-1}\gamma h_2) \overline{\Theta_\BL(h_1, \phi_1)} \Theta_\BL(h_2, \phi_2) dh_1 dh_2,
\]
which is related to period integrals via the spectral decomposition:
\[
\mathrm{RTF}_{(G, H,\rho)}(f, \phi_1, \phi_2)=\sum_{\pi} \sum_{\varphi \in OB(\pi)} P_{(G,H,\rho)}(R(f)\varphi, \phi_1) \overline{P_{(G,H,\rho)}(\varphi, \overline{\phi}_2)}.
\]
However, the geometric decomposition into orbital integrals is non-trivial. Similarly, consider a pair $(G', H_1, \rho_1)$ and a reductive subgroup $H_2 \leq G'$ and form a relative trace formula
\[
\mathrm{RTF}_{(G', H_1, \rho_1, H_2)}(f', \phi_1'):= \int_{[H_1] \times [H_2]} \sum_{\gamma \in G(F)} f'(h_1^{-1}\gamma h_2) \Theta_{\BL'}(h_1, \phi_1') dh_1 dh_2.
\]
Again, the geometric decomposition into orbital integrals is non-trivial. 
    \item 
Moreover, our theorems have applications to arithmetic analogs of twisted GGP conjectures and the study of Asai motives, see e.g. \cite{TAFL} where the twisted AFL is proved. In the case $n=2$, the arithmetic conjecture is related to twisted triple product formula of Ichino and its arithmetic analogs. 
\end{enumerate}

Finally, we give a summary of our paper. In Section \ref{section: representatives}, we develop the theory of normal representatives. In Section \ref{section: decomp and orbits}, we give relative trace formulas and decompositions into orbital integrals under the choice of representatives. In Section \ref{section: matching orbits and integrals}, we establish matching of local orbital integrals. In Section \ref{section: comparison RTF}, we match relative trace formulas using results in previous sections. In Section \ref{section: proof of the main theorem}, we finish the proof of the twisted GGP conjecture under our local assumptions.

\subsection{Acknowledgement}
We thank Wee Teck Gan, Xue Hang, Yifeng Liu and Wei Zhang for helpful suggestions.
\begin{comment}

\end{comment}

\subsection{Notations and Conventions}

\begin{itemize}

\item We often use the right action of a reductive group $G$ on a smooth affine variety $X$ over a field. And $x \in X$ is called semisimple if $G.x$ is closed, and regular if the stabilizer $G_x$ has minimal dimension among all orbits. 

\item We often use $\zeta$ to mean an element in the double coset,  $\gamma$ to mean the corresponding element in the group under corresponding anti-symmetric maps, and $\delta$ to mean further simplifications of $\gamma$. On the unitary side, we use the notation
\[
[\zeta, z] \in R^{V} \times V^\vee.
\]
On the general linear side, we use the notation
\[
[\gamma, x, y] \in S \times F_n \times F^{-,n}.
\]

\item Let $E/F$ be a quadratic extension of number fields. Let $\BA=\BA_F$ and $\BA_E$ be their rings of adeles respectively. Denote by $\sigma_{E/F}$ the generator of $\Gal(E/F)$. 

\item We use $j \in E$ (resp. $j_M \in M$) to be a totally imaginary element in $E/F$ (resp. $M/F$).

\item By a skew-hermitian space $V$ over $E$, we mean a 
$E$-vector space $V$ equipped with an skew-Hermitian form $\left<-,-\right>: V \times V \to E$, 
i.e. it is $\sigma_{E/F}$-linear in the first variable, 
linear in the second variable, and 
$\left<x,y\right>=-\sigma_{E/F}(\left<y,x\right>)$. 

We choose a non-zero totally imaginary element $j$ of $E$. Then $(V, j\left<-,-\right>)$ is a Hermitian space.  

\item Let $[\SHerm_n^\times(F)]$ denote the set
of isomorphism classes of $n$-dimensional
nondegenerate skew-Hermitian spaces over $E$.

\item Let $\eta_{E/F}$ be the quadratic character of
$\AA^\times/F^\times$ associated to $E/F$ by
global class field theory, and $\mu$ be a character
of $\AA_E^\times/E^\times$ such that 
$\mu|_{\AA^\times} = \eta_{E/F}$. Similarly, we denote by $\eta_{L/K}$ the quadratic character of $\AA^\times_K$ associated to $L/K$.

\item Let $\psi: \BA_F / F \to \BC^\times$ be a non-trivial additive character, which gives a character $\psi_E(x)=\psi(\Tr_{E/F}(x)): \BA_E / E \to \BC^\times$.

\item Let $\Res_{E/F} V^\vee =\LL + \LL^\vee$ be a polarization of Lagrangians. Let $\omega=\omega_{\psi, \mu}$ be the Weil representation
of the unitary group $\U(V)(\AA)$, which is realized on
$\cS(\LL(\AA))$. We have $\overline{\omega}=\omega_{\psi^{-1}, \mu^{-1}}$. Here we follow the notations of \cite{Xue-GGP}, which are slightly different from \cite{GGP-FourierJacobi}.

\item Let $K/F$ be another quadratic extension of number fields. Let $\sigma_{K/F}$ be the generator of $\Gal(K/F)$. 
Set 
$$
H=\U(V) \to G=\Res_{K/F}\U(V_K),
$$
\[
H_1 = \Res_{E/F} \GL_n, \quad G' = \Res_{L/F} \GL_n, \quad H_2 = \Res_{K/F} \GL_n.
\]

\item Let $M$ be the subfield of $L$ fixed by $\sigma_{L/K} \circ \sigma_{L/E}$.
We fix an element $j_M \in M$ satisfying 
$\sigma_{M/F}(j_M) = - j_M$.

\item We choose $\eta'$ a character of $\BA^\times_L$ such that $\eta'|_{\BA^\times_K}=\eta_{L/K}$. 

\item For any connected reductive group $G$ over a number field $F$, let $A_G^\infty$ be the connected component of the $\BR$-point of the split center of the $\BQ$-group $\mathrm{Res}_{F/\BQ} G$.

A property is that
\[
    G(\BA) = G(\BA)^1 \times A_{G}^\infty.
\]

\item We use 
$f' \in C_c^\infty(G'(\AA))$ and $\phi' \in
\cS(\AA_{E,n})$ on the general linear side. We use $f \in C_c^\infty(G(\AA))$ and
$\phi_1 \otimes \phi_2 \in \cS(\LL(\mathbb A)) \otimes \cS(\LL(\mathbb A))$ on the unitary side.

\end{itemize}

\subsection{Measures}
Using our fixed character $\psi: \BA_F/F \to \BC^{\times}$, we can naturally provide Haar measure $dg$ on $G(\BA_F)$ for any linear algebraic group $G$ over $F$, see \cite[Section 2.3]{GGP-endoscopic}.

The Tamagawa measure $dg$ on $G(\BA_F)$ is of the form $dg = (\Delta_G^*)^{-1} \prod dg_v$, where $dg_v$ is a Haar measure on $G(F_v)$ such that the hyperspecial subgroup has Haar measure 1 for almost all $v$ and $\Delta_G^*$ is a constant. When $G = \GL_n$, then $\Delta_G^* = \zeta_F^*(1) \zeta_F(2) \cdots \zeta_F(n)$, and
\[
    dg_v(\GL_n(\O_{F_v}))
    = 1
\]
for $\psi_v$ unramified.
We also have
\[
    \Delta_{\U(V)}^* = L(1,\eta_{E/F}) L(2,\eta_{E/F}^2) \cdots L(n,\eta^n_{E/F})
\]
and
\[
    dg_v(\U(V)(\O_{F_v}))
    = 1
\]
for when $\psi_v$ and $E_v/F_v$ is unramified.

\section{Normal representatives}
\label{section: representatives}

In this section, we deal with the existence of normal representatives on regular semisimple orbits (Definition \ref{normal reps: unitary} and \ref{normal reps: GL side}). The main result is Proposition \ref{prop:normalrep}. In Section \ref{subsec:orbits_and_matching}, we discuss the matching of regular semisimple representatives.

On the unitary side, we have a symmetric pair $H \leq G$ with involution $\theta$. We know that
 
\begin{num} 
    \item \label{eq:centralizer_in_diagonal} Given $\zeta \in G$, if the centralizer $C_G(\theta(\zeta^{-1})\zeta)$ of $\theta(\zeta^{-1}) \zeta$ in $G$ is a torus and $\zeta \in C_G(\theta(\zeta^{-1})\zeta)$, then $H \times H$-stabilizer of $\zeta$ is in the diagonal subgroup $H \leq H \times H$.
\end{num}

Indeed, under the assumption, we have $C(\theta(\zeta)^{-1}\zeta) \subseteq C(\zeta)$. Note that $g^{-1}\zeta h =\zeta$ for $g, h \in H$ implies that $g^{-1}\theta(\zeta) h =\theta(\zeta)$, hence $h^{-1} \zeta^{-1} \theta(\zeta) h= \zeta^{-1} \theta(\zeta)$, hence $h \in C(\theta(\zeta)^{-1}\zeta)$.

On the general linear side, for two subgroups $H_1, H_2$ of $G'$, we need to choose $\xi \in G'/H_2$ such that its stabilizer in $H_1$ is contained in certain subgroup $H_{10} \leq H_1$ so there is a natural action of $H_{10}$ on Weil representation for $F_n \times F^{-,n}$ through partial Fourier transforms.

\subsection{General formulation}

Consider the action of a group $T$ on a set $S$, with a subgroup $T_0 \leq T$. Our interest is on the locus of $x \in S$ such that $\mathrm{Stab}_{T}(x) \leq T_0$, which is stable under the action of $T_0$. We call such $x$ a $T_0$-good element in $S$. The following lemma is clear.

\begin{lem}
Consider the action of a group $T$ on a set $\widetilde{S}$. Let $S \subseteq \widetilde{S}$ be a subset, and $H=\mathrm{Stab}(S) \leq T$. Assume that
\begin{enumerate}
    \item For any $\widetilde{y} \in \widetilde{S}$, there exists $y \in S$ and $g \in T$ such that $g\widetilde{y}=y$.
    \item If $gy_1=y_2$ for some $g \in T$ and $y_1, y_2 \in S$, then there exists $h \in H$ such that $hy_1=y_2$.
\end{enumerate}
Then the natural inclusion $H \backslash S \to G \backslash \widetilde{S}$ is a bijection.
\end{lem}

Let $G$ be a linear algebraic group over a field $F$. 
Let $\theta_1$ and $\theta_2$ be commuting
involutions of $G$,
and let $H_1$ and $H_2$ be the subgroups of $G$ fixed
by $\theta_1$ and $\theta_2$, respectively.
Consider the right action of $H_1 \times H_2$
on $G$ given by 
\[
    \xi.(h_1, h_2) = h_1^{-1} \xi h_2.
\]

We have a $H_1$-equivariant embedding 
\[
G/H_2 \hookrightarrow G,  \, \, \xi \mapsto \gamma=\xi \theta_2(\xi)^{-1}
\]
whose image lies in $G^{\theta_2=(-)^{-1}}:=\{ \gamma \in G : \gamma \theta_2(\gamma)=1\}$. Here $H_1$ acts on $\gamma \in G^{\theta_2=(-)^{-1}}$ by
\[
    \gamma.h_1 = h_1^{-1}\gamma \theta_2(h_1).
\]

\begin{definition}
An element $\gamma \in G^{\theta_2=(-)^{-1}}$ is 
\emph{$\theta_1\theta_2$-normal} if 
$\gamma$ commutes with $\theta_1\theta_2(\gamma) = \theta_1(\gamma)^{-1}$.  
\end{definition}

Let $\delta=\gamma \theta_1\theta_2(\gamma) = \gamma \theta_1(\gamma)^{-1} \in G$. Then 
\[
\theta_1(\delta) \delta =id.
\] 

\begin{lem}\label{lemma: normal implies good}
\begin{enumerate}
    \item An element $\gamma$ is normal iff $\theta_1 \theta_2(\delta)=\delta$.
    \item If $\gamma$ is $\theta_1\theta_2$-normal and
the centralizer $C_G(\delta)$ is a
torus in $G$, then $\Stab_{H_1}(\gamma) \subseteq H_1^{\theta_2=\mathrm{id}}$.
\end{enumerate}
In particular, if $\delta$ is regular semisimple and $C_G(\delta)$ is a torus and $\gamma$ is $\theta_1\theta_2$-normal, then $\delta$ is good.
\end{lem}
\begin{proof}
The first part follows from the definition. Suppose $h \in H_1$ satisfies $h^{-1}\gamma \theta_2(h)
= \gamma$. Then 
\[
    \theta_2(h)^{-1}\theta_1\theta_2(\gamma)
    h = \theta_1\theta_2(\gamma),
\]
so $h \in C_G(\gamma\theta_1\theta_2(\gamma))$.
By assumption 
$\gamma \in C_G(\gamma\theta_1\theta_2(\gamma))$
as well, so $\gamma$ and $h$ commute.
This combined with $h^{-1}\gamma \theta_2(h) 
= \gamma$ implies that $h = \theta_2(h)$ as desired.
\end{proof}

We may hope that the $H_1$-orbit of a given element $\gamma \in G^{\theta_2=(-)^{-1}}$ contains a $\theta_1\theta_2$-normal element. This is not true in general. In the next two subsections, we will prove related results in our cases by Lie algebra methods.

\subsection{Normal representatives on the unitary side}

Recall that $V$ is a $n$-dimensional skew-hermitian space over a quadratic extension $E/F$, and
\[
    H=\U(V) \leq G=\U(V_K)   
\]
is the fixed subgroup under the Galois involution $\sigma_{L/E}$ on $\U(V_K)$.

\begin{definition}
For $\zeta \in G(F)$, let
\[
    T_\zeta \coloneqq \{(h_1, h_2) \in H \times H
    : h_1^{-1}\zeta h_2 = \zeta\}
\]
denote the stabilizer of $\zeta$ in $H \times H$.
We say that $\zeta$ is
\emph{good} if $T_\zeta$ is contained in the
diagonal $H \le H \times H$.
\end{definition}

\begin{example}
If $K/F$ is split, i.e., $G=H \times H$, then any $\zeta=(1,\zeta_2)$ is good.
\end{example}

Let $G_{\good}$ be the locus of good element of $G$.

\begin{lem}
$G_{\good} \subseteq G$ is a non-empty closed subset, and stable under $H$-conjugacy action.   
\end{lem}
\begin{proof}
The action of $H \times H$ on $G$ induces a map
$$
f: \{ (h_1,h_2, \zeta) \in H \times H \times G : h_1^{-1}\zeta h_2=\zeta \} \to G
$$
which is flat hence open. Then $G - G_{\good}$ is the image of the open subset $H \times H \times G - \Delta H \times G$, hence open. Therefore, $G_{\good}$ is closed and non-empty ($1 \in G_{\good}$).

If $\zeta$ is good, then $h^{-1}\zeta h$ is also good for any $h \in H$ by definition.
\end{proof}

\begin{definition}
\label{def:normalgeneral}
We say that an element $\xi \in G(F)$ is
\emph{normal} if $\xi$ commutes with $\xi\sigma_{L/E}(\xi)^{-1}$.
This is equivalent to $\xi\sigma_{L/E}(\xi)^{-1} = \sigma_{L/E}(\xi)^{-1}\xi$.
\end{definition}

Note that for any $\xi \in G(F)$ and $h_1, h_2 \in H_1$,
we have
\begin{equation}
\label{eq:norm}
    (h_1^{-1}\xi h_2)\sigma_{L/E}(h_1^{-1}\xi h_2)^{-1}
    = h_1^{-1}(\xi \sigma_{L/E}(\xi)^{-1}) h_1.
\end{equation}

From Lemma \ref{lemma: normal implies good}, we have

\begin{lem}
\label{lem:normalgeneral}
Suppose $\xi \in G(F)$ is normal and the centralizer
of $\xi\sigma_{L/E}(\xi)^{-1}$ is a (commutative) torus.
Then for any $h_1, h_2 \in H_1$ such
that $h_1^{-1}\xi h_2 = \xi$, we have $h_1 = h_2$. 
\end{lem}

\begin{definition} \label{normal reps: unitary}
Choose a set $R$ of normal representatives of
the regular semisimple $H(F) \times H(F)$-orbits
of $G(F)$ that contain a normal element. Let 
\[
    Y =R^V\times V^\vee.
\]
For $\zeta \in R^V$, consider the action of $T_\zeta \subseteq H$ on $z \in V^\vee$ by right multiplication.  We say that $z$ is regular semisimple if it is under the action of $T_\zeta$.  Let $(R^V\times V^\vee)_\rss$ be the set of elements of $Y$ such that $z$ is regular semisimple. 

Form the equivalence relation 
$[\zeta, z] \sim [\zeta, zt]$ for $t \in T_\zeta$ on $(R^V\times V^\vee)_{\rss}$. Let 
\begin{equation}
[Y_\rss]=(R^V\times V^\vee)_{\rss} / \sim
\end{equation}
be the set of regular semisimple equivalence classes.
\end{definition}

\begin{remark}
Here we choose $R$ a good representative of $F$-orbits, 
which is usually smaller than a good representative of $F_v$-orbits for local place $v$ of $F$. Note that $H$ acts on $G_{\good} \times V \subseteq G \times V$, and we could consider orbital integrals for this restricted action. But we do not know how to describe the locus $G_{good} \leq G$ in general. 
\end{remark}

\subsection{Normal representatives on the base change side} 
Recall 
\[
H_1 = \Res_{E/F} \GL_{n,E}, \quad G' = \Res_{L/F} \GL_{n,L}, \quad H_2 = \Res_{K/F} \GL_{n,K}.
\]
Moreover, let $M$ be the third quadratic extension of $F$ inside $L=E \otimes_F K$. Let $H_{10}=\GL_{n,F} \leq H_1$. Note that 
\[
    G'(F)/H_2(F) = \GL_n(L)/\GL_n(K)  \cong \S_{n,L/K} \coloneqq \{\gamma \in  \GL_{n}(L) : \gamma \sigma_{L/K}(\gamma) = 1\} 
\] 
via the anti-symmetrization map 
\[
    \xi \mapsto \xi\sigma_{L/K}(\xi)^{-1}.
\]
Consider the right action of $g \in G'(F)$ on $\gamma \in \S_{n,L/K}$ by
\begin{equation}
    \gamma.g = g^{-1}\gamma \sigma_{L/K}(g).    
\end{equation}

\begin{definition}
For $\xi \in G'(F)$, let
\[
    T_\xi \coloneqq \{(h_1, h_2) \in H_1 \times H_2
    : h_1^{-1}\xi h_2 = \xi\}
\]
denote the stabilizer of $\xi$ in $H_1 \times H_2$.
We say that $\xi \in G'(F)$ is
\emph{good} if for any $(h_1, h_2) \in T_\xi$, 
we have $h_1 \in H_{10}$. 
\end{definition}

\begin{definition}
\label{def:glnormal}
For $\gamma \in S_{n, L/K}$, let
\[
    T_\gamma = \{h_1 \in H_1 : h_1\gamma
    \sigma_{E/F}(h_1) = \gamma\}
\]
denote the stabilizer of $\gamma$ in $H_1$.
\end{definition}

\begin{lem}
\label{lem:goodsn}
An element $\xi \in G'(F)$ is good if and only if
$T_\gamma \subseteq H_{10}$, where
$\gamma = \xi\sigma_{L/K}(\xi)^{-1}$.
\end{lem}
\begin{proof}
This follows from the fact that for any $h_1 \in H_1(F)$, 
there exist $h_2 \in H_2(F)$ such that
$(h_1, h_2) \in T'_\xi$ if and only if $h_1 \in T_\gamma$.
Indeed, if $(h_1, h_2) \in T_\xi'$, then 
\[
    \gamma = \xi \sigma_{L/K}(\xi)^{-1}
    = (h_1^{-1} \xi h_2)(h_2^{-1} \sigma_{L/K}(\xi)^{-1} \sigma_{E/F}(h_1))
    = h_1^{-1} \gamma \sigma_{E/F}(h_1),
\]
so $h_1 \in T_\gamma$. On the other hand, if $h_1 \in T_\gamma$,
then $(h_1, h_2) \in T_\xi'$ for $h_2 = \xi^{-1}h_1\xi$, which
is an element of $H_2(F)$ since
\[
    \sigma_{L/K}(h_2) = \sigma_{L/K}(\xi)^{-1} \sigma_{E/F}(h_1)
    \sigma_{L/K}(\xi) = \xi^{-1} h_1 \xi = h_2,
\] 
where the middle equality follows from $h_1 \in T_\gamma$.
\end{proof}

\begin{lem}
\label{lem:normalequivdefs}
Let $\gamma \in \S_{n, L/K}$.
The following are equivalent.
\begin{enumerate}[(1)]
\item $\gamma$ is normal in the sense of Definition
\ref{def:normalgeneral}.
\item $\gamma\sigma_{L/E}(\gamma)^{-1} \in \GL_n(M)$.
\item $\gamma\sigma_{L/E}(\gamma)^{-1} \in \S_{n, L/K}$.
\end{enumerate}
\end{lem}
\begin{proof}
This follows from the definition.
\end{proof}

\begin{lem}
    Suppose $\xi \in G'(F)$ has the property that 
    $\gamma = \xi \sigma_{L/K}(\xi)^{-1}$ is normal and
    regular semisimple. Then $\xi$ is good.
\end{lem}
\begin{proof}
By Lemma \ref{lem:goodsn}, $\xi$ is good if and only if
$T_\gamma$ is contained in $H_{10}(F)$. 
Since $\gamma$ is regular semisimple, the centralizer
of $\gamma\sigma_{L/E}(\gamma)^{-1}$ is a torus.
Suppose $h_1 \in T_\gamma$, so $h_1 \in H_1$ and  
$h_1^{-1}\gamma \sigma_{E/F}(h_1) = \gamma$. 
By Lemma \ref{lem:normalgeneral}, we have
$h_1 = \sigma_{E/F}(h_1)$, so $h_1 \in H_{10}(F)$.
\end{proof}

To prove the results on the existence of good
representatives and the matching of orbits, we recall
the Cayley transform. 

We consider the Lie algebras
\[
    \fs_{n, L/K} \coloneqq \{x \in \gl_n(L) :
    x + \sigma_{L/K}(x) = 0\}
\]
and
\[
    \fs_{n, L/E} \coloneqq \{x \in \gl_n(L) : x + \sigma_{L/E}(x) = 0\}
\]
of $\S_{n, L/K}$ and $\S_{n, L/E}$, respectively. Recall $\langle -, -\rangle$ is the skew-Hermitian
form on $V_K$. For $g \in \End(V_K)$, let $g^\ast \in \End(V_K)$ be
such that $\langle gv, w\rangle =
\langle v, g^\ast w\rangle$. Let
\[
    \fh(V) \coloneqq \{g \in \End(V) : g^\ast = g\},
\]
and
\[
    \fu(V_K) \coloneqq \{x \in \End(V_K) : x + x^\ast = 0\}
\]
be the Lie algebra of $\U(V_K)$.

\begin{definition}
For any $\alpha \in L^\times$, let 
\[
D_\alpha = \{x \in \gl_n(L) : \det(\alpha - x) = 0 \}. 
\]
The Cayley transform is the map
\begin{align*}
    c_\alpha \colon \gl_n(L) - D_\alpha &\simrightarrow \gl_n(L) - D_1 \\
    \delta &\mapsto - (\alpha + \delta )(\alpha - \delta)^{-1},
\end{align*}
which has inverse $c_\alpha^{-1}(x) = -\alpha(1 + x)(1 - x)^{-1}$.
\end{definition}

The Cayley transform is $G'$-equivariant, and
if $\alpha \sigma_{L/E}(\alpha) = 1$, it induces an isomorphism
\[
    c_\alpha \colon \S_{n, L/E} - D_\alpha \simrightarrow \fs_{n, L/E} - D_1.
\]
If $\alpha\sigma_{L/K}(\alpha) = 1$, it induces
isomorphisms
\[
    c_\alpha \colon \S_{n, L/K} - D_\alpha \simrightarrow
    \fs_{n, L/K} - D_1,
\]
\[
    c_\alpha \colon \U(V_K) - D_\alpha \simrightarrow
    \fu(V_K) - D_1.
\]

\begin{definition}
Let $\gamma \in S_{n,L/K}$, then the ``norm'' of $\gamma$ is 
$\delta = \gamma\sigma_{L/E}(\gamma)^{-1}$.
\end{definition}

\begin{proposition}
    \label{prop:normalrep}
Every $H_1$-orbit of
$\S_{n,L/K}$ contains a normal element (\ref{def:normalgeneral})
\end{proposition}
\begin{proof}
We want to show that there exists
$g \in H_1$ such that 
$g^{-1}\delta g \in \GL_n(M)$, since
for such $g$, the element $\gamma.g$ would be normal,
since
\[
    (\gamma.g)\sigma_{L/E}(\gamma.g)^{-1}
    = g^{-1}\delta g
\]
by \eqref{eq:norm}.

It is clear that $\delta$ is an element of $\S_{n, L/E}$.
Let $\alpha \in M$ be such that
$\alpha \sigma_{M/F}(\alpha) = 1$ and $\alpha$ is not
an eigenvalue of $\delta$. Let $A = j_M c_\alpha(\delta)$.
We have $c_\alpha(\delta) \in \fs_{n, L/E}$, so $A \in \Mat_n(E)$.

Note that 
\[
    \sigma_{L/M}(\delta) = 
    \sigma_{L/E}\circ \sigma_{L/K}(\gamma)\sigma_{L/K}(\gamma)^{-1}
    = \sigma_{L/E}(\gamma)^{-1} \gamma
    = \gamma^{-1} \delta \gamma.
\]
Since $\sigma_{L/M}(\alpha) = \alpha$, we also have
\begin{align*}
    c_\alpha(\sigma_{L/M}(\delta))
    &= - (\alpha + \sigma_{L/M}(\delta))(\alpha - \sigma_{L/M}(\delta))^{-1}\\
    &= \sigma_{L/M}(-(\alpha + \delta)(\alpha - \delta)^{-1}) = \sigma_{L/M}(c_\alpha(\delta)).
\end{align*}
Combining the previous two displayed equations, we obtain
\[
    \sigma_{L/M}(c_\alpha(\delta)) = \gamma c_\alpha(\delta) \gamma^{-1}.
\]
Since $A = j_M c_\alpha(\delta)$ and $j_M \in M$, 
this implies
\[
    \sigma_{E/F}(A) = \sigma_{L/M}(A) = \gamma^{-1} A \gamma.
\]

This implies that the elementary divisors of
$A$ have coefficients in $F$, so there exists $g \in H_1$ such
that $g^{-1} A g \in \gl_n(F)$. Now we can transform it back using Cayley maps.

This means $g^{-1} c_\alpha(\delta) g \in \gl_n(M)$, so 
\[
    g^{-1}\delta g = c_\alpha^{-1}(g^{-1} c_\alpha(\delta) g) \in \GL_n(M). \qedhere
\]
\end{proof}

\begin{lem}[cf. {\cite[Lemma~1.1(ii)]{Arthur-Clozel}}]
\label{lem:injnorm}
If $\gamma$ and $\gamma'$ are elements of $\S_{n, L/K}$ such that
$\gamma\sigma_{L/E}(\gamma)^{-1}$ and $\gamma'\sigma_{L/E}(\gamma')^{-1}$ are conjugate
by an element of $H_1$, 
then $\gamma$ and $\gamma'$ are in the same $H_1$-orbit
of $\S_{n,L/K}$.
\end{lem}
\begin{proof}

By Proposition \ref{prop:normalrep},
we can assume that $\gamma$ and $\gamma'$ are normal
(see Definition \ref{def:glnormal}).
By \eqref{eq:norm}, 
we can assume that $\gamma\sigma_{L/E}(\gamma)^{-1}
= \gamma'\sigma_{L/E}(\gamma')^{-1} = \delta$.

Let
\[
    \mathfrak t_\gamma(F) = \{g \in \gl_n(E) :
    \gamma \sigma_{E/F}(g) = g\gamma\}.
\]
This is an $F$-algebra. Let 
\[
    \mathfrak t_\gamma(E) = \{g \in \gl_n(E) :
    \delta g = x \delta \}.
\]

We have an isomorphism
\[
    x \otimes e \mapsto xe\colon  \mathfrak t_\gamma(F)
\otimes_F E \simeq \mathfrak t_\gamma(E).
\]
Indeed, this map is injective because every element
of $\mathfrak t_\gamma(F) \otimes_F E$ is of the
form $g_1 \otimes 1 + g_2 \otimes j_E$, and if
$g_1 + j_Eg_2 = 0$ for 
$g_1, g_2 \in \mathfrak t_\gamma(F)$, then
\[
    \gamma\sigma_{E/F}(g_1)
    = g_1\gamma 
    = -j_Eg_2 \gamma
    = -j_E\gamma\sigma_{E/F}(g_2)
    = - \gamma\sigma_{E/F}(g_1),
\]
so $g_1 = g_2 = 0$.
It is surjective because 
if $g \in \mathfrak t_\gamma(E)$, then
$\gamma \sigma_{E/F}(g) \gamma^{-1} \in \mathrm{Mat}_n(E)$
since
\[
    \sigma_{L/K} \circ \sigma_{L/E}(\gamma \sigma_{E/F}(g) \gamma^{-1})
    = \sigma_{L/E}(\gamma)^{-1} g
    \sigma_{L/E}(\gamma) = \gamma^{-1} g \gamma
    = \sigma_{L/K}(\gamma \sigma_{E/F}(g) \gamma^{-1}),
\]
where we have used the fact that 
$\gamma\sigma_{L/K}(\gamma) = 1$,
and the middle equality holds since $\delta g = g\delta$.
Then
\[
   \sigma_{E/F}(\gamma \sigma_{E/F}(g) \gamma^{-1})
   = \sigma_{L/K}(\gamma \sigma_{E/F}(g) \gamma^{-1})
   = \gamma^{-1} g \gamma,
\]
so
\[
    \gamma \sigma_{E/F}(g + \gamma 
    \sigma_{E/F}(g) \gamma^{-1})
    = \gamma(\sigma_{E/F}(g) + \gamma^{-1} g \gamma)
    = \gamma \sigma_{E/F}(g) + g \gamma
    = (g + \gamma \sigma_{E/F} \gamma^{-1})\gamma,
\]
which means
\[
    g + \gamma \sigma_{E/F}(g)\gamma^{-1} \in \mathfrak t_\gamma(F). 
\]
Similarly, we find
\[
    j_E(g - \gamma \sigma_{E/F}(g) \gamma^{-1}) \in
    \mathfrak t_\gamma(F),
\]
so the map $g \otimes e \mapsto ge$ is surjective since
\[
    g = \frac{1}{2}(g + \gamma \sigma_{E/F}(g) \gamma^{-1})
    + \frac{1}{2j_E} \cdot j_E(g - \gamma \sigma_{E/F}(g)\gamma^{-1}).
\]

Under the isomorphism $\mathfrak t_\gamma(F)
\otimes_F E \simeq \mathfrak t_\gamma(E)$,
the corresponding action of $\Gal(E/F)$ on
$\mathfrak t_\gamma(E)$ is given by
\[
    \sigma_{E/F}.g = \gamma \sigma_{E/F}(g) \gamma^{-1}.
\]
Then by Hilbert's Theorem 90 
\cite[Chapter~X,~Section~1,~Exercise~2]{Serre-local}, 
we have 
\begin{equation}
\label{eq:Hone}
    H^1(\Gal(E/F), \mathfrak t_\gamma(E)^\times) = 0.
\end{equation}
This is equivalent to the statement that if
$g \in H_1$ satisfies
\[
    g^{-1}\delta g = \delta, \quad
    g(\gamma \sigma_{E/F}(g) \gamma^{-1}) = 1,
\]
then 
\[
    g = h(\gamma \sigma_{E/F}(h) \gamma^{-1})^{-1}
\]
for some $h \in H_1$ satisfying $h^{-1}\delta h = \delta$.

Now let $g = \gamma'\gamma^{-1}$. Note that
$g \in H_1$. This is because since $\gamma$ and $\gamma'$ are normal, we have
\[
    \sigma_{L/E}(\gamma)^{-1}\gamma
    = \gamma\sigma_{L/E}(\gamma)^{-1} 
    = \gamma'\sigma_{L/E}(\gamma')^{-1}
    = \sigma_{L/E}(\gamma')^{-1}\gamma',
\]
so $\sigma_{L/E}(\gamma'\gamma^{-1})
= \gamma'\gamma^{-1}$. 
We also have 
\[
    g^{-1}\delta g 
    = \gamma\gamma'^{-1}
    (\gamma'\sigma_{L/E}(\gamma')^{-1}) \gamma'\gamma^{-1}
    = \gamma (\sigma_{L/E}(\gamma')^{-1}\gamma')
    \gamma^{-1}
    = \delta.
\]
In addition, since $\gamma\sigma_{L/K}(\gamma) = \gamma'\sigma_{L/K}(\gamma') = 1$,
we have $\sigma_{E/F}(g) = \sigma_{L/K}(g)
= \gamma'^{-1}\gamma$, so
\[
    g(\gamma \sigma_{E/F}(g) \gamma^{-1})
    = \gamma'\gamma^{-1} (\gamma (\gamma'^{-1}\gamma)\gamma^{-1}) = 1.
\]
Then by \eqref{eq:Hone}, there exists 
$h \in H_1$ such that
$\gamma'\gamma^{-1} = h(\gamma \sigma_{E/F}(h) \gamma^{-1})^{-1}$, which is equivalent to
$\gamma' = h \gamma \sigma_{E/F}(h)^{-1}$.
\qedhere
\end{proof}

\begin{definition}[cf. Definition \ref{normal reps: unitary}]  \label{normal reps: GL side}
Let $S$ be a set of normal representatives
of the regular semisimple $H_1$-orbits
of $\S_{n,L/K}$. Let 
\[
    X = S \times F_n \times F^{-,n}.
\]
For $\gamma \in S$, consider the action of $T_\gamma \subseteq H_{10}$ on
$[x, y] \in F_n \times F^{-,n}$ given by
$[x, y].t = [xt, t^{-1} y]$.
We say $[x, y]$ is regular semisimple if it is under this
action of $T_\gamma$.

Let $(S \times F_n \times F^{-,n})_\rss$ be the set of elements of $X$
such that $[x, y]$ is regular semisimple. Form the equivalence relation
$[\gamma, x, y] \sim [\gamma, xt, t^{-1}y]$ on $(S \times F_n \times F^{-,n})_\rss$ 
for $t \in T_\gamma$.
Let 
\[
    [X_\rss] = (S \times F_n \times F^{-,n})_\rss/\sim
\]
be the set of regular semisimple equivalence classes.
\end{definition}

\subsection{Orbits and orbit matching}
\label{subsec:orbits_and_matching}

Consider the right action of $\GL_n(L)$ on 
\[
    \Mat_n(L) \times L_n \times L^n
\] 
given by
\[
    [\xi, x, y].h = [h^{-1}\xi h, xh, h^{-1} y].
\]
Recall the invariants
\begin{equation}
\label{eq:invar}
    a_i([\xi, x, y]) = \Tr \wedge^i \xi, \quad
    b_i[\xi, x, y]) = x(\xi^{i-1} y)
\end{equation}
for $1 \le i \le n$.
We say that an element of $\Mat_n(L) \times L_n \times L^n$ is 
\emph{strongly regular semisimple} if and only if
\begin{enumerate}
    \item $\xi$ is regular semisimple as an element
    of $\Mat_n(L)$;
    \item the vectors $\{x, x\xi, \dots, x\xi^{n-1}\}$ 
    span the $L$-vector space $L_n$;
    \item the vectors $\{y, \xi y, \dots, \xi^{n-1} y\}$
    span the $L$-vector space $L^n$.
\end{enumerate}

\begin{lem}[{\cite[Lemma~5.6]{Liu2014relative},
\cite[Proposition~6.2~and~Theorem~6.1]{rallis2007multiplicity}}] 
Strongly regular semisimple elements of $\Mat_n(L) \times
L_n \times L^n$
have trivial stabilizers, and two strongly 
regular semisimple
elements are in the same $\GL_n(L)$-orbit if and only
if they have the same invariants.
\end{lem}

We also consider the space
\[
    \S_{n, L/E} \times E_n \times E^n
\]
which has right action of $\GL_n(E)$ induced 
by the $\GL_n(L)$-action on $\mathrm{Mat}_n(L) \times L_n \times L^n$.
We say that an element of this space is strongly
regular semisimple
if it is strongly regular semisimple as an element of 
$\Mat_n(L) \times L_n \times L^n$.

\begin{lem}
We have embeddings
\[
   [\gamma, x, y] \mapsto [\gamma \sigma_{L/E}(\gamma)^{-1}, x, y] \colon X \to 
   \S_{n, L/E} \times E_n \times E^n
\]
and
\[
    [\zeta, z] \mapsto [\zeta \sigma_{L/E}(\zeta)^{-1}, z, z^\ast] \colon
    Y \to \S_{n, L/E} \times E_n \times E^n
\]
such that an element of $X$ or $Y$ is regular
semisimple if and only if its image is strongly regular
semisimple.
\end{lem}

\begin{definition}
We define the invariants of an element
$[\gamma, x, y] \in X$ or $[\zeta, z] \in Y$
to be equal to the invariants $a_i$ and $b_i$,
as defined in \eqref{eq:invar}, of
its image under the above embeddings.
\end{definition}

\begin{lem}
\label{lem:sv}
Two strongly regular semisimple elements of $[\delta, x, y]$ and
$[\delta', x', y']$ of $\S_{n, L/E} \times E_n
\times E^n$ are in the same $\GL_n(L)$-orbit
if and only if they are in the same $\GL_n(E)$-orbit.
\end{lem}
\begin{proof}
Suppose there exists $h \in \GL_n(L)$ such that
\[
    [\delta, x, y].h = [h^{-1}\delta h, xh, h^{-1} y]
    = [\delta', x', y'].
\]
Since $\delta \sigma_{L/E}(\delta) 
= \delta'\sigma_{L/E}(\delta') = 1$, we have
\begin{align*}
    [\delta, x, y].\sigma_{L/E}(h) 
    &= [\sigma_{L/E}(h)^{-1} \delta \sigma_{L/E}(h), x\sigma_{L/E}(h),
    \sigma_{L/E}(h)^{-1}y]\\
    &= \sigma_{L/E}([h^{-1}\delta^{-1} h, xh, h^{-1}y])\\
    &= \sigma_{L/E}([\delta'^{-1}, x', y']) \\
    &= [\delta', x', y'].
\end{align*}
By the triviality of stabilizers of strongly
regular semisimple elements, we have 
$\sigma_{L/E}(h) = h$, so $h \in \GL_n(E)$.
\end{proof}

\begin{lem}
\label{lem:embed}
\leavevmode
\begin{enumerate}[(1)]
\item
Two elements $[\gamma, x, y]$ and $[\gamma', x', y']$
of $X_\rss$ are equivalent if and only if they
have the same invariants.

\item 
Two elements $[\zeta, z]$ and $[\zeta', z']$ of
$Y_\rss^V$ are equivalent if and only if
they have the same invariants.
\end{enumerate}
\end{lem}
\begin{proof}
\leavevmode
\begin{enumerate}[(1)]
\item
Let $\delta = \gamma\sigma_{L/E}(\gamma)^{-1}$ and
$\delta' = \gamma'\sigma_{L/E}(\gamma')^{-1}$.
Note that $[\gamma, x, y]$ and 
$[\gamma', x', y']$ have the same invariants
if and only if $[\delta, x, y]$ and
$[\delta', x', y']$ are in the same $\GL_n(L)$-orbit, 
and thus by Lemma \ref{lem:sv} in the same 
$\GL_n(E)$-orbit.

The only if direction is straightforward, by the
definition of the equivalence relation of $X$.
For the other direction, suppose 
\[
   [\delta, x, y].h = [h^{-1}\delta h, xh, h^{-1}y] =
   [\delta', x', y']
\]
for $h \in \GL_n(E)$.
By Lemma \ref{lem:injnorm}, we have $\gamma = \gamma'$.

Similarly to the proof of Lemma \ref{lem:sv}, 
since $\delta\sigma_{L/K}(\delta) 
= \delta'\sigma_{L/K}(\delta') = 1$, we have
\begin{align*}
    [\delta, x, y].\sigma_{L/K}(h) 
    &= [\sigma_{L/K}(h)^{-1} \delta \sigma_{L/K}(h), x\sigma_{L/K}(h),
    \sigma_{L/K}(h)^{-1}y]\\
    &= \sigma_{L/K}([h^{-1}\delta^{-1} h, xh, -h^{-1}y])\\
    &= \sigma_{L/K}([\delta'^{-1}, x', -y']) \\
    &= [\delta', x', y'].
\end{align*}
By triviality of stabilizers, we have $\sigma_{L/K}(h) = h$,
so $h \in \GL_n(F)$. 
Since $\gamma$ is normal and $\delta$ is regular
semisimple, this implies $h \in T_\gamma$, so
$[\gamma, x, y]$ and $[\gamma', x', y']$ are equivalent
since $[x, y].h = [x', y']$.

\item
Let $\varsigma = \zeta\sigma_{L/E}(\zeta)^{-1}$
and $\varsigma' = \zeta'\sigma_{L/E}(\zeta')^{-1}$.
As in the previous part, $[\zeta, z]$ and
$[\zeta', z']$ have the same invariants if and only
if $[\varsigma, z, z^\ast]$ and
$[\varsigma', z', z'^\ast]$ are in the same $\GL_n(E)$-orbit.

Suppose that 
\[
    [\varsigma, z, z^\ast].g = [\varsigma', z', z'^\ast]
\]
for $g \in \GL_n(E)$. Since $\varsigma\varsigma^\ast
= \varsigma'\varsigma'^\ast$, we have
\begin{align*}
    [\varsigma, z, z^\ast].g^{\ast, -1} 
    &= [g^\ast \varsigma g^{\ast,-1}, zg^{\ast, -1}, g^\ast z^\ast] \\
    &= [g^{-1}\varsigma'^{-1} g, g^{-1} z^\ast, zg]^\ast\\ 
    &= [\varsigma^{-1}, z'^\ast, z']^\ast\\
    &= [\varsigma, z', z'^\ast],
\end{align*}
so $g \in \U(V)$.
Let $h = \zeta^{-1} g \zeta'$. 
We check that  
\[
    \sigma_{L/E}(h) = \sigma_{L/E}(\zeta)^{-1} g \sigma_{L/E}(\zeta')
    = \zeta^{-1} g \zeta' = h,
\]
where the middle equality holds since 
$g^{-1}\varsigma g = \varsigma'$.
Since $g$, $\zeta$, $\zeta' \in \U(V_K)$, we also have
$h \in \U(V)$.
Thus $\zeta = \zeta'$ since $\zeta' = g^{-1}\zeta h$.

Since $\zeta$ is normal and $\varsigma$ is regular
semisimple, we have $g = h \in T_\zeta$, so $[\zeta, z]$
and $[\zeta', z']$ are equivalent since $z' h = z$. 
\qedhere
\end{enumerate}
\end{proof}

\begin{definition}
\label{def:lieorbits}
Let $\fx = \gl_n \times F_n \times F^{-,n}$,
and let $\fy(V) = \mathfrak h(V) \times V^\vee$.
The groups $\GL_n$ and $\U(V)$ act on $\fx$ and
$\fy(V)$, respectively, via
\[
    [A, x, y].h = [h^{-1} A h, xh, h^{-1}y],\quad
    [A', z].h = [h^{-1}A' h, zh].
\]
Let $\fx_\rss(F)/\GL_n(F)$ and $\fy(V)_\rss(F)/\U(V)(F)$
denote the regular semisimple orbits of $\fx(F)$
and $\fy(V)(F)$ under these actions.
\end{definition}

\begin{lem}[{\cite[Lemma~3.1]{WeiGGP14}}]
\label{lem:liematch}
There is a bijection
\[
    \fx_\rss(F)/\GL_n(F) \simeq \coprod_{V}
    \fy(V)_\rss(F)/\U(V)(F)
\]
such that $[A, x, y]$ and $[A', z]$ match
if and only if $[A, x, y]$ and $[A', z, z^\ast]$
are in the same $\GL(V)$-orbit.
\end{lem}

\begin{lem}
\label{lem:normmatch}
\leavevmode
\begin{enumerate}[(1)]
\item Suppose $\gamma \in \S_{n,L/K}$ is regular
semisimple and normal,
and let $\delta = \gamma \sigma_{L/E}(\gamma)^{-1}$.
Suppose $h^{-1}\delta h = \varsigma$ for $h \in \GL(V)$
and $\varsigma \in \U(V)(K)$.
Then $\zeta \coloneqq h^{-1}\gamma h$ is a normal element
of $\U(V)(K)$ which satisfies $\varsigma = \zeta\sigma_{L/E}(\zeta)^{-1}$. 

\item Suppose $\zeta \in \U(V)(K)$ is regular semisimple
and normal, and
let $\varsigma = \zeta\sigma_{L/E}(\zeta)^{-1}$.
Suppose $h^{-1} \varsigma h = \delta$ for $h \in \GL(V)$
and $\delta \in \S_{n,L/K}$. 
Then $\gamma \coloneqq h^{-1}\zeta h$ is a normal element
of $\S_{n,L/K}$ which satisfies 
$\delta = \gamma\sigma_{L/E}(\gamma)^{-1}$.
\end{enumerate}
\end{lem}
\begin{proof}
\leavevmode
\begin{enumerate}[(1)]
\item First note that $\gamma$ normal implies
$\sigma_{L/M}(\delta) = \delta$. Also, since
$h \in \GL(V)$, it is clear that $\varsigma \in \S_{n, L/E}$.
This combined with $\varsigma \in \U(V)(K)$ implies
$\sigma_{L/E}(\varsigma^\ast) = \varsigma$.

The fact that $\varsigma = \zeta\sigma_{L/E}(\zeta)^{-1}$
and commutes with $\zeta$ is obvious, since
$\sigma_{L/E}(\zeta) = h^{-1}\sigma_{L/E}(\gamma) h$,
so we just
need to check that $\zeta\zeta^\ast = 1$. 
To this end, note that since $\gamma$ is normal,
\[
    \gamma = c_{n-1}\delta^{n-1} + \cdots + c_0
\]
for some $c_0, \dots, c_{n-1} \in L$.
This implies
\[
    \sigma_{L/M}(\gamma) = \sigma_{L/M}(c_{n-1})
    \delta^{n-1} + \cdots + \sigma_{L/M}(c_0),
\]
and
\[
    \zeta = c_{n-1}\varsigma^{n-1} + \cdots + c_0,
\]
since $\zeta = h^{-1}\gamma h$ and $\varsigma = h^{-1}\delta h$.
Note that $\sigma_{L/E}(\varsigma^\ast) = \varsigma$ then implies
\begin{align*}
    \sigma_{L/E}(\zeta^\ast) = \sigma_{L/M}(c_{n-1}) \varsigma^{n-1}
    + \cdots + \sigma_{L/M}(c_0) = h^{-1}\sigma_{L/M}(\gamma) h.
\end{align*}
This implies $\zeta^\ast = h^{-1}\sigma_{L/K}(\gamma) h$,
so $\zeta\zeta^\ast = 1$ since
$\gamma\sigma_{L/K}(\gamma) = 1$.

\item As before, we have $\sigma_{L/E}(\varsigma^\ast) = \varsigma$,
and $\delta \in \S_{n, L/E}$, so $\delta_{L/M}(\delta) 
= \delta$. Since $\zeta$ commutes with $\varsigma$,
we have
\[
    \zeta = c_{n-1} \varsigma^{n-1} + \cdots + c_0
\]
for some $c_0, \dots, c_{n-1} \in L$.
Then
\[
    \sigma_{L/E}(\zeta^\ast)
    = \sigma_{L/M}(c_{n-1}) \varsigma^{n-1} + \cdots + 
    \sigma_{L/M}(c_0),
\]
\[
    \gamma = c_{n-1} \delta^{n-1} + \cdots + c_0,
\]
and
\[
    \sigma_{L/M}(\gamma) =
    \sigma_{L/M}(c_{n-1})\delta^{n-1} + \cdots +
    \sigma_{L/M}(c_0)
    = h^{-1}\sigma_{L/E}(\zeta^\ast)h,
\]
which implies 
$\sigma_{L/K}(\gamma) = h^{-1}\zeta^\ast h$,
so $\gamma\sigma_{L/K}(\gamma) = 1$. As before,
it is clear that $\gamma$ is normal
and satisfies $\delta = \gamma\sigma_{L/E}(\gamma)^{-1}$.
\qedhere
\end{enumerate}
\end{proof}

\begin{lem}
\label{lem:rsscayley}
Let $[\gamma, x, y] \in X_\rss$ and
$[\zeta, z] \in Y_\rss^V$. Suppose $\alpha \in M$
satisfies $\alpha \sigma_{M/F}(\alpha) = 1$
and is not an eigenvalue of $\delta = \gamma\sigma_{L/E}(\gamma)^{-1}$
or $\varsigma = \zeta\sigma_{L/E}(\zeta)^{-1}$. 
Then we have
$[j_M c_\alpha(\delta), x, y] \in \mathfrak x_\rss(F)$
and $[j_M c_\alpha(\varsigma), z] \in \fy(V)_\rss(F)$.
\end{lem}

\begin{proposition}[Matching of orbits]
    \label{lem:matching}
There is a natural bijection 
\[
  [X_\rss] \simeq \coprod_{V} [Y_\rss^V]
\]
such that $[\gamma, x, y]$ and $[\zeta, z]^V$ match 
if and only if they have the same invariants.
\end{proposition}
\begin{proof}
As usual, we let 
$\delta = \gamma\sigma_{L/E}(\gamma)^{-1}$, and
$\varsigma = \zeta\sigma_{L/E}(\zeta)^{-1}$.
Note that $[\gamma, x, y]$ and $[\zeta, z]$
have the same invariants if and only if
$[\delta, x, y]$ and $[\varsigma, z, z^\ast]$ 
are in the same $\GL(V)$-orbit, which is true
if and only if for any $\alpha$
as in Lemma \ref{lem:rsscayley},
$[j_M c_\alpha(\delta), x, y] \in \mathfrak x_\rss(F)$ and 
$[j_M c_\alpha(\varsigma), z] \in \mathfrak y(V)_\rss(F)$ 
match in the sense of Lemma \ref{lem:liematch}.
Thus the proposition follows from Lemmas \ref{lem:embed},
\ref{lem:liematch}, 
\ref{lem:normmatch}, and \ref{lem:rsscayley}.

Indeed, suppose $[\gamma, x, y] \in X_\rss$.
Let $\alpha \in M$ be as in Lemma \ref{lem:rsscayley}.
By Lemma \ref{lem:liematch}, there exists
$[A', z'] \in \fy(V)_\rss(F)$, for a unique $V$, such
that 
\[
    [j_M c_\alpha(\delta), x, y].h =
    [A', z', z'^\ast]
\]
for some $h \in \GL(V)$.
Let $\varsigma = c_\alpha^{-1}(j_M^{-1} A')$.
Then $\varsigma = h^{-1}\delta h$, so by Lemma
\ref{lem:normmatch}, we have $\varsigma = 
\zeta\sigma_{L/E}(\zeta)^{-1}$ for some $\zeta'
\in \U(V)(K)$. Suppose 
$\zeta' = h_1^{-1}\zeta h_2$ for $\zeta \in R^V$
and $h_1, h_2 \in \U(V)(F)$.
Let $z = z' h_1^{-1}$.
Then $[\zeta, z] \in Y_\rss^V$, and $[\zeta, z]$
and $[\gamma, x, y]$ 
have the same invariants.
Furthermore, $[\zeta, z] \in Y_\rss^V$ is unique by Lemma
\ref{lem:embed}.

The same argument shows that for any $[\zeta, z]
\in Y_\rss^V$, there exists a unique $[\gamma, x, y]
\in X_\rss$ such that $[\gamma, x, y]$ and
$[\zeta, z]$ have the same invariants.
\end{proof}

Recall that not every $\U(V)(F) \times \U(V)(F)$-orbit
of $\U(V)(K)$ contains a normal element.
Using Proposition \ref{lem:matching}, in Lemma
\ref{lem:unitarynormalexistence} below, we give
an equivalent definition of the regular semisimple
orbits which do contain a normal element.
Furthermore, in Lemma \ref{lem:localglobal}, we 
show that the local--global principle holds
for checking whether a regular semisimple orbit contains a normal element.

\begin{lem}
\label{lem:unitarynormalexistence}
Let $\zeta \in \U(V)(K)$ be an element which
is regular semisimple with respect to the 
$\U(V)(F) \times \U(V)(F)$ action. The orbit
$\U(V)(F) \zeta \U(V)(F)$ contains a normal element if and
only if
\[
    \zeta\sigma_{L/E}(\zeta)^{-1}
    = \gamma\sigma_{L/E}(\gamma)^{-1}
\]
for some $\gamma \in \S_{n, L/K}$.
\end{lem}
\begin{proof}
Suppose $h_1^{-1} \zeta h_2$ is normal for some
$h_1, h_2 \in \U(V)(F)$. By
Proposition \ref{lem:matching}, there exists
$\gamma' \in \S_{n, L/K}$ such that
\[
    h_1^{-1}\zeta\sigma_{L/E}(\zeta) h_1^{-1}
    = h^{-1}\gamma'\sigma_{L/E}(\gamma')^{-1} h
\]
for some $h \in \GL_n(E)$.
This is equivalent to $\zeta\sigma_{L/E}(\zeta)^{-1}
= \gamma\sigma_{L/E}(\gamma)^{-1}$ for
$\gamma = h_1 h^{-1}\gamma' \sigma_{E/F}(h h_1^{-1})$.

On the other hand, suppose that we are given
$\zeta\sigma_{L/E}(\zeta)^{-1} =
\gamma\sigma_{L/E}(\gamma)^{-1}$ for $\gamma
\in \S_{n, L/K}$. 
By Proposition \ref{prop:normalrep}, there exists
$h \in \GL_n(E)$ such that $h^{-1}\gamma
\sigma_{E/F}(h)$ is normal. 
Then by Lemma \ref{lem:normmatch} (1), there exists
normal $\zeta' \in \U(V)(K)$ with
$\zeta'\sigma_{L/E}(\zeta')^{-1}
= \zeta\sigma_{L/E}(\zeta)^{-1}$, which implies
that $\zeta'$ is in the orbit $\U(V)(F)\zeta \U(V)(F)$.
\end{proof}

\begin{lem}
\label{lem:normtimesF}
Let $M/F$ and $K/F$ be quadratic extensions of number
fields. Let $L = M \otimes_F K$.
Let $N_{L/M} \colon L \to M$ and
$N_{L_v/M_v} \colon L_v \to M_v$ denote the norm maps. 
Let $x \in M^\times$ be such that
\[
    x \in N_{L_v/M_v}(L_v^\times) \cdot F_v^\times
\]
for all places $v$ of $F$. Then
$x \in N_{L/M}(L^\times)\cdot F^\times$.
\end{lem}
\begin{proof}
The assumption on $x$ tells us that there exists
$(f_v)_v \in \AA^\times$ such that
$x \cdot (f_v)_v \in N_{L/M}(\AA_L^\times)$.
Since $x \in M$, we have
\[
    \eta_{K/F}((f_v)_v) = \eta_{L/M}(x) = 1,
\]
where $\eta_{K/F} \colon \AA^\times/F^\times
\to \Gal(K/F)$ and
$\eta_{L/M} \colon \AA_M^\times/M^\times
\to \Gal(L/M)$
are the maps coming from global class field theory.
Since the kernel of $\eta_{K/F}$ is 
$N_{K/F}(\AA_K^\times)$,
there exists $f \in F^\times$ such that
\[
    f^{-1} \cdot (f_v)_v \in N_{K/F}(\AA_K^\times)
    \subseteq N_{L/M}(\AA_L^\times).
\]
This means that $xf \in N_{L/M}(\AA_L^\times)$,
and since $xf \in M$, by the Hasse norm
theorem, this implies that $xf \in N_{L/M}(L^\times)$.
\end{proof}

\begin{lem}
\label{lem:localglobal}
Suppose $\delta \in \S_{n, M/F} \coloneqq
\S_{n,L/K} \cap \GL_n(M)$ is regular semisimple.
Then $\delta = \gamma\sigma_{L/E}(\gamma)^{-1}$
for some $\gamma \in \S_{n, L/K}$ if and only
if $\delta = \gamma_v\sigma_{L_v/E_v}(\gamma_v)^{-1}$
for some $\gamma_v \in \S_{n, L_v/K_v}$, for every
place $v$ of $F$.
\end{lem}
\begin{proof}
Let $L'$ be the centralizer of $\delta$ in
$\Mat_n(L)$, and let $M'$ be the centralizer of
$\delta$ in $\Mat_n(M)$. 
Since $\delta\sigma_{M/F}(\delta) = 1$, we have
that $\sigma_{M/F}$ acts on $M'$ and 
$\sigma_{L/K}$ and $\sigma_{L/E}$ act on $L'$. 
Let $F'$ be the
$\sigma_{M/F}$-fixed elements of $M'$, and let
$K'$ and $E'$ be the $\sigma_{L/K}$-fixed and
$\sigma_{L/E}$-fixed elements of $L'$.

Recall from Lemma \ref{lem:normalequivdefs} that
if $\gamma_v\sigma_{L_v/E_v}(\gamma_v)^{-1} =
\delta$ for $\gamma_v \in \S_{n, L_v/K_v}$, then
$\gamma_v$ is normal, so $\gamma_v \in L_v'$.
Thus, we are given that
\[
    \delta = \gamma_v\sigma_{L_v'/E_v'}(\gamma_v)^{-1}
\]
for some $\gamma_v \in L_v'$ with
$\gamma_v\sigma_{L_v'/K_v'}(\gamma_v) = 1$
for each place $v$ of $F$, and we want to show
that $\delta = \gamma\sigma_{L'/E'}(\gamma)^{-1}$
for some $\gamma \in L'$ with
$\gamma\sigma_{L'/K'}(\gamma) = 1$.

Since $\delta$ is regular semisimple, $F'$ is a product
of field extensions on $F$, and
it clearly suffices to prove the claim in the case 
that $F'$ is a field.

Then since $\delta\sigma_{M'/F'}(\delta) = 1$, there exists
$x \in M'$ such that $\delta = x\sigma_{M'/F'}(x)^{-1}$.
Since $\gamma_v\sigma_{L_v'/K_v'}(\gamma_v) = 1$, 
there exists $\xi_v \in L_v'$ such that
$\gamma_v = \xi_v\sigma_{L_v'/K_v'}(\xi_v)^{-1}$, so
\[
    x\sigma_{M'/F'}(x)^{-1} = \delta
    = \gamma_v\sigma_{L_v'/E_v'}(\gamma_v)^{-1}
    = \xi_v\sigma_{L_v'/K_v'}(\xi_v)^{-1}
    \sigma_{L_v'/M_v'}(\xi_v)\sigma_{L_v'/E_v'}(\xi_v)^{-1}.
\]
This implies that $x \in \xi_v\sigma_{L_v'/M_v'}(\xi_v)
\cdot F_v'^\times$, so 
\[ 
    x \in N_{L_v'/M_v'}(L_v'^\times) \cdot F_v'^\times
\]
for all places $v$ of $F$.
Then by Lemma \ref{lem:normtimesF},  
$x \in N_{L'/M'}(L'^\times) \cdot F'^\times$. 
Thus, there exists $\xi \in L'^\times$ and $f \in F'^\times$
such that $x = \xi\sigma_{L'/M'}(\xi) f$. 
Then, if we let $\gamma = \xi\sigma_{L'/K'}(\xi)^{-1}$,
we have
\[
    \delta = x\sigma_{M'/F'}(x)^{-1}
    = \xi\sigma_{L'/M'}(\xi)\sigma_{L'/E'}(\xi)^{-1}
    \sigma_{L'/K'}(\xi)^{-1}
    = \gamma\sigma_{L'/E'}(\gamma)^{-1}.
    \qedhere
\]
\end{proof}

\section{Relative trace formulas and geometric decompositions via representatives} \label{section: decomp and orbits}

In this section, we construct two relative trace formulas, and prove their geometric decompositions for good test functions, into sums of products of local orbital integrals, under the choice of global representatives. We use partial Fourier transforms as isomorphisms of Weil representations in the doubling space as in \cite{Li1992}. We follow the notation of \cite{TGGP-Wang} using right actions.

\subsection{Unitary side}

Let $V$ be a skew-hermitian space over $E$. We choose a 
polarization $\LL + \LL^\vee$ of $(\Res_{E/F} V)^\vee$.
Let $\omega = \omega_{\psi,\mu}$ be the Weil 
representation of $\U(V)(\AA)$ realized on 
$\cS(\BL(\AA))$.

\begin{definition}
Let $v$ be a place of $F$. We define the
partial Fourier transform
\[
^\ddagger \colon \cS(\LL (F_v)) \otimes \cS(\LL (F_v))
\to \cS(V^\vee(F_v))
\]
by
\[
    (\phi_{1,v} \otimes \phi_{2,v})^\ddagger(z)
    = \int_{\LL(F_v)} \phi_{1,v}(x' + x)
    \phi_{2,v}(x - x')
    \psi_v(-2\Tr_{E/F}\langle x', y\rangle)\,dx'
\]
for $z = x + y \in V^\vee$
with $x \in \LL$ and $y \in \LL^\vee$.
Here, $\langle -, -\rangle$ is the skew-hermitian form on $V^\vee$.
We define the global partial Fourier transform
$^\ddagger \colon \cS(\LL(\BA_F)) \otimes \cS(\LL(\BA_F)) \to \cS(V^\vee(\BA_F))$ using the same formula.
\end{definition}

We will consider relative trace formulas for the embedding $H=\U(V) \to G=\Res_{K/F} \U(V_K)$ with theta series. Instead of using explicit formulas, for geometric decompositions of relative trace formulas, we only need two properties on partial Fourier transforms:

\begin{itemize}
    \item The first thing is the equivariance under the diagonal $H(F_v) \subseteq H(F_v) \times H(F_v)$:
\begin{equation}
  ( \overline{\omega(h)\phi_{1,v} } \otimes 
  \omega(h)\phi_{2,v})^\ddagger(z)= (\ol{\phi_{1,v}} 
  \otimes \phi_{2,v})^\ddagger(zh), \quad h \in \U(V)
  (F_v), \quad z \in V^\vee(F_v).
\end{equation}
    \item For $\phi \in\CS(\mathbb L(\AA))$, we have the theta series
\[
\Theta(g, \phi)=\sum_{x \in \mathbb L(F)} \omega(g)\phi(x).
\]
The second property is the following global reduction formula via partial Fourier transform (see also \cite[Proposition 2.2]{HKS-Theta1996} and
\cite[Lemma 5.3]{Liu2014relative}).

\end{itemize}

\begin{theorem}[{\cite[Lemma~4.2.1]{Xue-GGP}}]  \label{key:reduction} For $\phi=(\phi_1 \otimes \phi_2)^\ddagger \in \CS(V^\vee(\BA_F))$ we have
\begin{equation}
\overline{ \Theta(g, \phi_1) }  \Theta(g, \phi_2) = \sum_{z \in V^\vee(E)} \phi(zg).
\end{equation}
Therefore, for $a \in H(F)$, we have 
\[
\overline{\Theta(h_1, \phi_1)}  \Theta(h_2, \phi_2)  = \sum_{z \in V^\vee(E)} ( \overline{ \omega(h_2^{-1}ah_1) \phi_1} \otimes  \phi_2)^\ddagger (zh_2).
\]
\end{theorem}

Following Definition \ref{normal reps: unitary}, we choose a set $R$ of normal representatives of
the regular semisimple $H(F) \times H(F)$-orbits
of $G(F)$. Let 
\begin{equation}
[Y_\rss]=(R \times V^\vee)_{\rss} / \sim
\end{equation}
be the set of regular semisimple equivalence classes.

\begin{definition}[cf. Definition \ref{def:glgood}]
\label{def:ugood} 
An element $(f = \otimes_{v} f_v, \phi_1 \otimes \phi_2 = \otimes_{v}  (\phi_{1,v} \otimes \phi_{2,v} ))$ of $C_c^\infty(G(\mathbb A)) \times \CS(\mathbb L(\mathbb A))^{\otimes 2}$ is called a
\emph{good test function} if it satisfies the following
properties.
\begin{enumerate}[(1)]
\item  \label{umatrixcoeff}
There exists a finite place $v_1$ of
$F$ which is split in $E$ such that $f_{v_1}$ is a truncated matrix coefficient of a supercuspidal representation, that is, of
the form
\[
    f_{v_1}(g) = \widetilde{f_{v_1}}(g) 
    \cdot \mathbf{1}_{G(F_{v_1})^1}(g),
\]
where $\widetilde{f_{v_1}}$ is a matrix coefficient of
a supercuspidal representation of $G(F_{v_1})$, and
\[
    G(F_{v_1})^1 = \{g \in G(F_{v_1}) : 
    \lvert \det g \rvert_{v_1} =  1  \}.
\]

\item \label{ursssupp} 
There exists a finite place $v_2$ of $F$
which is split in $E$ such that
$f_{v_2}$ is supported on the regular
semisimple $H(F_{v_2}) \times H(F_{v_2})$-orbits 
of $G(F_{v_2})$, and if $[\zeta, z] \in (R^V\times V^\vee)$ and
\[
    f_{v_2}(h_1^{-1}\zeta h_2)
    (\ol{\omega
    (h_2^{-1}h_1)\phi_{1,v_2}} \otimes 
    \phi_{2,v_2})^\ddagger(zh_2)\neq 0
\]
for some $h_1, h_2 \in \U(V)(F_{v_2})$,
then $[\zeta, z] \in (R \times V^\vee)_{\rss}$.

\item For any place $v$ of $F$,
the function $f_v$ is supported on
\[
    \{\zeta \in G(F_v) : H(F_v) \zeta H(F_v)
    \text{ contains a normal element}\}.
\]
\end{enumerate}
\end{definition}

\begin{definition}
For $f \in C_c^\infty(G(\AA))$ and
$\phi_1 \otimes \phi_2 \in 
\cS(\LL(\AA))^{\otimes 2}$, 
consider the relative trace formula distribution
\begin{equation}
    J(f, \phi_1 \otimes \phi_2)
    = \int_{[H]}\int_{[H]}
    K_f(h_1, h_2) \ol{\Theta(h_1, \phi_1)}
    \Theta(h_2, \phi_2)\,dh_1\,dh_2.
\end{equation}
Here $K_f(x, y)=\sum_{\gamma \in G(F)} f(x^{-1}\gamma y)$ is the automorphic kernel function.
\end{definition}

\begin{proposition} \label{prop:udecomp}
For a good test function $(f, \phi_1 \otimes \phi_2)$, we have 
\[
    J(f, \phi_1 \otimes \phi_2)= \sum_{[\zeta, z] \in [Y_\rss] 
    }
    \int_{H(\AA)}\int_{H(\AA)}
    f(h^{-1}_2 \zeta h_1)
    (\ol{\omega(h^{-1}_2h_1)\phi_1} \otimes \phi_2)^\ddagger(zh_2)\,dh_1\,dh_2,
\]
and both sides converge absolutely. 
\end{proposition}
\begin{proof}
We have
\begin{align*}
    &\sum_{\zeta \in G(F)}
    f(h_1^{-1}\zeta h_2)\ol{\Theta(h_1, \phi_1)}\Theta(h_2, \phi_2) \\
    &\overset{\mathrm{Def.} \ref{def:ugood} (2)}{=} \sum_{\zeta \in R }
    \sum_{(a, b) \in T_\zeta \bs (H\times H)(F)} 
    f(h_1^{-1}a^{-1}\zeta b h_2) \ol{\Theta(h_1, \phi_1)} \Theta(h_2, \phi_2)  \\
    &\overset{\mathrm{Thm.} \ref{key:reduction}}{=} \sum_{\zeta \in R }
    \sum_{(a, b) \in T_\zeta \bs (H\times H)(F)} 
    f(h_1^{-1}a^{-1}\zeta b h_2)
    \sum_{z \in V^\vee(E)}( \ol{\omega((bh_2)^{-1}ah_1)\phi_1} \otimes \phi_2 )^\ddagger(zbh_2)\\
    &= \sum_{\zeta \in R} 
    \sum_{(a, b)\in T_\zeta \bs (H \times H)(F)}    f(h_1^{-1}a^{-1}\zeta bh_2) 
    \sum_{z \in T_\zeta(F) \backslash V(F)}\sum_{t \in T_{[\zeta, z]} \bs T_\zeta}
    ( \ol{\omega (h^{-1}_2b^{-1}ah_1)\phi_1} \otimes \phi_2)^\ddagger(ztbh_2) \\
    &=\sum_{[\zeta, z] \in [Y_\rss]}
    \sum_{(a, b) \in (H\times H)(F)}
    f(h_1^{-1}a^{-1}\zeta bh_2)
    ( \ol{\omega(h^{-1}_2b^{-1}ah_1)\phi_1} \otimes \phi_2)^\ddagger(zbh_2),
\end{align*}
Thus,
\begin{align*}
    J(f, \phi_1 \otimes \phi_2)= \sum_{[\zeta, z] \in [Y_\rss]}
    \int_{H(\AA)}\int_{H(\AA)}
    f(h_1^{-1}\zeta h_2)
    (\ol{\omega(h^{-1}_2h_1)\phi_1} \otimes \phi_2)^\ddagger(zh_2)\,dh_1\,dh_2.
\end{align*}
Both sides converge because of Definition 
\ref{def:ugood}\eqref{ursssupp} .
\end{proof}

This motivates the following definition.

\begin{definition} \label{defn: local orb: unitary side}
For regular semisimple $[\zeta, z] \in G(F_v) \times V_v^\vee$,
$f_v \in C_c^\infty(G(F_v))$, and
$\phi_{1,v} \otimes \phi_{2,v} \in \cS(\LL(F_v))^{\otimes 2}$, 
define the \emph{(local) orbital integral} (cf. \cite[(4.1.2) and (4.1.4)]{Xue-GGP})
\begin{align*}
    &\Orb([\zeta, z], f_v, \phi_{1,v} \otimes \phi_{2,v}) = \int_{H(F_v)}\int_{H(F_v)} f_v(h_1^{-1}\zeta h_2)
    ( \ol{\omega_v(h^{-1}_2h_1)\phi_{1,v}} \otimes \phi_{2,v})^\ddagger(zh_2)\, dh_1\, dh_2.
\end{align*}
\end{definition}

\begin{corollary}
We have the decomposition
\[
    J(f, \phi_1 \otimes \phi_2)
    = \sum_{[\zeta, z] \in [Y_\rss]}
    \Delta_H^{*,-2}
    \prod_v \Orb([\zeta, z], f_v, \phi_{1,v} \otimes \phi_{2,v}).
\]    
\end{corollary}

\subsection{General linear side}

Recall $H_1=\Res_{E/F}\GL_n$.  

\begin{definition}[Weil representation] \label{defn theta series: GL side}

Let $R_\mu$ be the Weil representation of $H_1(\BA_F)=\GL_n(\AA_E)$, realized on $\cS(\AA_{E,n})$ by
\begin{equation}
   (R_\mu(g)\phi')(z) = \mu(\det g) \lvert \det g \rvert^{\frac12}
    \phi'(zg). 
\end{equation}    
For $\phi' \in \cS(\AA_{E,n})$,
define the theta series by   
\begin{equation}
\Theta_\mu(s, g, \phi') = \sum_{z \in E_n \setminus \{0\}} (R_\mu(g) \phi')(z) \lvert \det g \rvert^{s-\frac 12}
\end{equation}
for $g \in \GL_n(\AA_E)$, and let 
$\Theta_\mu(g, \phi') = \Theta_\mu(\frac12, g, \phi')$.
\end{definition}

Recall we fix a nonzero totally imaginary element $j$ of $E$. Following \cite[Section 3.1]{Xue-GGP} 
and \cite[(7.3)]{GGP-FourierJacobi}, we define partial Fourier transforms.

\begin{definition} Let $v$ be a place of $F$. Define the partial Fourier transform
\[
^\dagger \colon \cS(E_{v,n}) \to \cS(F_{v,n} \times
F_v^{-,n})
\]
by
\[
    \phi_v'^\dagger(x, y) =
    \lvert j \rvert_v \int_{F_{v,n}} \phi_v'(x + jw^\vee) 
    \psi_v(jw^\vee y)\, dw^\vee.
\]
Here, the Haar measure is chosen to be the self-dual measure with respect to $\psi_v$. 

\end{definition}

We define the global partial Fourier transform
$$
^\dagger \colon \cS(\AA_{E,n}) \to  \cS(\AA_n \times \AA^{-,n})
$$ using the same formulas.
Instead of using the explicit formula, for geometric decomposition we need two things: let $H_{10}=\GL_n \leq H_1$. The first thing is for any $h \in H_{10}(F_v)=\GL_n(F_v)$
\[
(R_\mu(h)\phi')^\dagger(x, y) = \eta_{E/F}(\det h) \phi'^\dagger(xh, h^{-1}y).
\]
The second thing is the following global reduction formula proved by the Possion summation formula.
\begin{theorem}
\label{key:reduction-GL}
For any $\phi' \in\CS(\BA_{E, n})$, we have 
\[
\sum_{v \in E_n} \phi'(v) = \sum_{(x, y) \in F_n \times F^{-,n}} \phi'^\dagger (x,y).
\]
\end{theorem}

Hence for any $a \in \GL_n(\BA_E)$ and $h \in \GL_n(\BA_F)$, we have
\begin{align*}
\Theta_\mu(ah, \phi') &=
\eta_{E/F}(\det a) \lvert \det h a \rvert^{\frac 12} \sum_{v \in E_n} \phi'(v h a) \\
&= \eta_{E/F}(\det h)^{-1} \sum_{(x, y)
\in F_n \times F^{-,n}} (R_\mu(ha) \phi')^\dagger (x, y)\\
&=
\sum_{(x, y) \in F_n \times F^{-,n}} (R_\mu(a) \phi')^\dagger (xh, h^{-1}y).
\end{align*}

\begin{definition}[cf. Definition \ref{def:ugood}]
\label{def:glgood}
An element $(f' = \otimes_v f_v', \phi' = \otimes_v \phi'_v)$
of $C_c^\infty(G'(\AA_F)) \times \cS(\AA_{E,n})$ is
called a \emph{good test function} if it satisfies the 
following properties.
\begin{enumerate}[(1)]
\item \label{glmatrixcoeff}
There exists a finite place $v_1$ of
$F$ which is split in $E$ such that $f_{v_1}'$ is of
the form
\[
    f_{v_1}'(g) = \widetilde{f_{v_1}'}(g) 
    \cdot \mathbf{1}_{G^{\prime}(F_{v_1})^1}(g),
\]
where $\widetilde{f_{v_1}'}$ is a matrix coefficient of a
supercuspidal representation of $G'(F_{v_1})$, and
\[
    G'(F_{v_1})^1 = \{g \in G'(F_{v_1}) : \lvert \det g \rvert_{v_1} = 1 \}.
\]
\item \label{glrsssupp} 
There exists a finite
place $v_2$ of $F$ which is split in $E$ such that
$\widetilde f_{v_2}'$ is supported on the regular
semisimple $\GL_n(E_{v_2})$-orbits of $\S_n(K_{v_2})$,
and if $[\gamma, x, y] \in X$ and
\[
    \widetilde f_{v_2}'(g^{-1}\gamma \sigma_{E/F}(g))
    \ol{(R_{\mu_{v_2}}(g)\phi_{v_2}')^\dagger(x, y)} \neq 0
\]
for some $g \in \GL_n(E_{v_2})$, then $[\gamma, x, y] \in X_\rss$.

\item \label{glarch}
For all archimedean places $v$ of $F$, we have
$E_v = F_v \times F_v$, and $\phi_v'$ is a finite 
linear combination of functions of the form 
$\phi_{1,v}' \otimes \phi_{2,v}'$ where
$\phi_{1,v}', \phi_{2,v}' \in \cS(F_{v,n})$.
\end{enumerate}
\end{definition}

We choose the following
distribution to compare with unitary sides. Recall we have subgroups
$H_1=\Res_{E/F} \GL_n \to G’=\Res_{L/F} \GL_n \leftarrow H_2=\Res_{K/F} \GL_n.$ Let $\eta_{L/K} : K^\times \backslash \BA_K^\times \to \BC$ be the quadratic character associated to $L/K$ given by global class field theory. For $h_2' \in H_2(F)$, we write $\eta_{L/K}(h)=\eta_{L/K}(\det h)$.

\begin{definition}
For $f' \in C_c^\infty(G'(\AA))$ and $\phi' \in
\cS(\AA_{E,n})$, define the distribution
\begin{equation}
   I(f', \phi') =
    \int_{[H_1]}\int_{[H_2]} K_{f'}(h_1, h_2) \ol{\Theta_\mu(h_1, \phi')}\eta_{L/K}(h_2)^{n+1}\, dh_1\,dh_2. 
\end{equation}
where 
$K_{f'}(h_1, h_2) = \sum_{\gamma \in G'(F)} f'(h_1^{-1} \gamma h_2).$
\end{definition}

We have the natural anti-symmetrization map 
$$
\xi \mapsto \gamma=\xi \sigma_{L/K}(\xi)^{-1} \colon G'(F)/H_2(F) \to S_{n,L/K} \subseteq G'(F).
$$

Following Definition \ref{normal reps: GL side}, let $S$ be a
set of normal representative of the $H_1(F)$-orbits 
on $S_{n,L/K}$. Let
\[
[X_\rss]=(S \times F_n \times F^{-,n})_\rss / \sim
\]
be the set of equivalence classes of regular semisimple pairs in $S \times F_n \times F^{-,n}$.

\begin{definition}\label{defn: local orb GL side}
For $[\xi, x, y] \in G'(F) \times F_n \times F^{-,n}$,
$f_v' \in C_c^\infty(G'(F_v))$ with $\xi\sigma_{L/K}(\xi)^{-1} \in S$, and $\phi_v' \in \cS(E_{v,n})$,
define the \emph{(local) orbital integral}
\begin{equation}
    \Orb([\xi, x, y], f_v', \phi_v'):= \int_{H_1(F_v)}\int_{H_2(F_v)}
        f_v'(h_1^{-1}\xi h_2) 
        \ol{(R_{\mu_v}(h_1) \phi_v')^\dagger(x, y)}
        \eta_{L_v/K_v}(h_2)^{n+1}\, dh_1 \, dh_2.
\end{equation}        
\end{definition}

Note that we have
\[
    \Orb([h_1^{-1}\xi h_2, xh_1, h_1^{-1}y],
    f_v', \phi_v')
    = \eta_{E_v/F_v}(h_1)\eta_{L_v/K_v}(h_2)^{n+1}
   \Orb([\xi, x, y], f_v', \phi_v') .
\]

To make the orbital integrals only depend on $\gamma$
rather than a lifting $\xi$, we fix a character $\eta_v' \colon L_v^\times \to \BC^\times$
satisfying $\eta_v'|_{K_v^\times} 
= \eta_{L_v/K_v}^{n+1}$. Then $\xi \in G'(F)/H_2(F)$ only depends on $\gamma \in \mathrm{S}_{n,L/K}(F)$. Define a new function $\widetilde {f_v'}$ on $\mathrm{S}_{n,L/K}(F_v)$ by
\begin{equation}
\label{eq:ftilde}
 \widetilde {f_v'} (\gamma) \coloneqq
    \eta_v'(\xi)\int_{H_2(F_v)} f_v'(\xi h_2)
    \eta_{L_v/K_v}(h_2)^{n+1}\, dh_2.    
\end{equation}   

We define a simplified version of the orbital
integral in terms of $\wt f'_v$.

\begin{definition} \label{defn: semip orb int on GL side}
For $[\gamma, x, y] \in [X_\rss]$, define the orbital integral for $\widetilde f_v' \in C_c^\infty(\mathrm{S}_{n,L/K})$ and $\phi_v' \in \cS(E_{v,n})$ as 
\[
    \Orb([\gamma, x, y], \widetilde f_v', \phi_v')
    = \int_{H_1(F_v)}
    \widetilde f_v'(h_1^{-1}\gamma \sigma_{L/K}(h_1) ) \ol{(R_{\mu_v}(h_1)\phi_v')^\dagger(x, y)}
    \eta_v'(h_1)\, dh_1.
\]
Note as $h \in H_1(F_v)$, we have $\sigma_{L/K}(h_1)=\sigma_{E/F}(h_1)$.

\end{definition}

By definition, for $\xi \in G'(F)$ with
$\xi\sigma_{L/K}(\xi)^{-1} = \gamma$, we have
\[
    \Orb([\xi, x, y], f_v', \phi_v')
    = \eta_v'(\xi)^{-1}\cdot \Orb([\gamma, x, y], \widetilde f_v', \phi_v').
\]    

\begin{remark}
We simplified the orbital integrals into the action of $H_{1}$ on $G'/H_2$. In the split case ($K=F \times F$), we are consider the $\sigma_{E/F}$-conjugacy action of $\GL_n(E)$ on $\mathrm{S}_{n,E/F} \times \mathrm{S}_{n, E/F}$. As $\GL_n(E)$ acts transitively on $\mathrm{S}_{n,E/F}$, we may choose a good representative $\xi \in \{1\} \times \mathrm{S}_{n,E/F}$, then we reduce to the action of $\GL_n(F)$ on $\mathrm{S}_{n,E/F}$ by conjugacy. 
\end{remark}

\begin{proposition}
\label{prop:gldecomp}
For good test functions $(f', \phi')$, we have
\[
    I(f', \phi') = \Delta_{H_1}^{*,-1} \Delta_{H_2}^{*,-1} \sum_{[\gamma, x, y] \in [X_\rss]}
    \prod_v \Orb([\xi, x, y], f_v', 
    \phi_v'),
\]
where $\xi$ is an element of $\GL_n(L)$ such
that $\gamma = \xi\sigma_{L/K}(\xi)^{-1}$, and
the product of local orbital integrals does
not depend on the choice of $\xi$.
\end{proposition}
\begin{proof}
First note that
\begin{align*}
    &\int\limits_{\GL_n(K) \bs \GL_n(\AA_K)}
    \sum_{\xi \in \GL_n(L)}
    f'(g^{-1}\xi h) \eta_{L/K}(h)^{n+1}\,dh\\ 
    &= \int\limits_{\GL_n(K)\bs \GL_n(\AA_K)}
    \sum_{\xi \in \GL_n(L)/\GL_n(K)}
    \sum_{a \in \GL_n(K)} f'(g^{-1}\xi a h)
    \eta_{L/K}(ah)^{n+1}\, dh \\
    &= \sum_{\xi \in \GL_n(L)/\GL_n(K)}
    \int_{\GL_n(\AA_K)} f'(g^{-1}\xi h)\eta_{L/K}(h)^{n+1}\, dh\\
    &= \sum_{\gamma \in \S_n(K)}
    \int_{\GL_n(\AA_K)} f'(g^{-1}\xi h)\eta_{L/K}(h)^{n+1}\, dh,
\end{align*}
where in the last sum, $\xi \in \GL_n(L)$ such that $\gamma = \xi\xi^{\sigma, -1}$,
and the integral does not depend on the choice of
$\xi$.
Then
\begin{align*}
    &\sum_{\gamma \in \S_n(K)}
    \int_{\GL_n(\AA_K)} f'(g^{-1}\xi h)\eta_{L/K}(h)^{n+1}
    \ol{\Theta_\mu(g, \phi')} \, dh \\
    &= \sum_{\gamma \in S} \sum_{a \in T_\gamma \bs \GL_n(E)}
    \int_{\GL_n(\AA_K)}
    f'(g^{-1}a^{-1}\xi h) \eta_{L/K}(h)^{n+1}
    \ol{\Theta_\mu(ag, \phi')} \, dh\\
    &= \sum_{\gamma \in S} \sum_{a \in T_\gamma\bs \GL_n(E)}
    \int_{\GL_n(\AA_K)}
    f'(g^{-1}a^{-1}\xi h)\eta_{L/K}(h)^{n+1}\sum_{[x, y] \in F_n \times F^{-,n}}
    \ol{(R_\mu(ag)\phi')^\dagger(x, y)} \, dh\\
    &= \sum_{\gamma \in S} \sum_{a \in T_\gamma \bs \GL_n(E)}
    \int_{\GL_n(\AA_K)} f'(g^{-1}a^{-1}\xi h)
    \eta_{L/K}(h)^{n+1}\\
    &\hspace{15em}\times \sum_{[x, y] \in (F_n \times F^{-,n})/T_\gamma} 
    \sum_{t \in T_{[\gamma, x, y]} \bs T_\gamma}
    \ol{(R_\mu(tag)\phi')^\dagger(x, y)}\, dh \\
    &= \sum_{[\gamma, x, y] \in [X_\rss]}
    \sum_{a \in \GL_n(E)} 
    \int_{\GL_n(\AA_K)}
    f'(g^{-1}a^{-1}\xi h) \eta_{L/K}(h)^{n+1}
    \ol{(R_\mu(ag)\phi')^\dagger(x, y)}\, dh,
\end{align*}
where $T_{[\gamma, x, y]}$ denotes the
stabilizer of $[x, y]$ in $T_\gamma$.
Combining the above calculations gives
\begin{align*}
    I(f', \phi') &= 
    \sum_{[\gamma, x, y] \in [X_\rss]}
    \int_{\GL_n(\AA_E)}
    \int_{\GL_n(\AA_K)} 
    f'(g^{-1}\xi h) \ol{(R_\mu(g)\phi')^\dagger(x, y)}\eta_{L/K}(h)^{n+1}\,dh\, dg. \qedhere
\end{align*}
\end{proof}

\subsection{Definition of smooth transfer} 
For $[\delta, x, y] \in \S_{n, L/E} \times V^\vee \times V$,
define
\[
    T_{[\delta, x, y]} 
    = j_M^{n(n-1)/2} \cdot \det \begin{bmatrix}
        x \\ x\delta \\ \vdots \\ x\delta^{n-1}
    \end{bmatrix}.
\]
Note that
\[
    \sigma_{L/E}(T_{[\delta, x, y]})
    = (\det \delta)^{-(n-1)} T_{[\delta, x, y]}.
\]
For $[\xi, x, y] \in \GL_n(L) \times F_n \times F^{-,n}$,
let $\gamma = \xi\sigma_{L/K}(\xi)^{-1}$
and $\delta = \gamma\sigma_{L/E}(\gamma)^{-1}$.
We have
\[
    \sigma_{L/E}(\det(\xi\sigma_{L/M}(\xi))^{-1}) = \det \delta\cdot \det(\xi\sigma_{L/M}(\xi))^{-1},
\]
so
\[
    \det(\xi\sigma_{L/M}(\xi))^{-(n-1)}T_{[\delta, x, y]}
    \in E.
\]
\begin{definition}[Transfer factor]
If $\gamma$ is normal, then $\delta \in \GL_n(M)$,
so it is easy to see that $\det(\xi\sigma_{L/M}(\xi))^{-(n-1)}T_{[\delta, x, y]}$
is also an element of $M$, so it is an element of $F$.
Thus, we may define
\begin{align*}
    \Omega([\xi, x, y]) &= \eta_{E/F}(\det(\xi\sigma_{L/M}(\xi))^{-(n-1)} 
    T_{[\delta, x, y]}).
\end{align*}    
\end{definition}

This transfer factor satisfies 
\begin{equation}
\label{eq:Omegainvariance}
    \Omega([g^{-1}\xi h, xg, g^{-1}y]) 
    = \eta_{E/F}(g)\eta_{L/K}(h)^{n+1} \Omega([\xi, x, y])
\end{equation}
for $g \in \GL_n(F)$, $h \in H_2(F)$.

\begin{definition}
For $[\gamma, x, y] \in \S_{n,L/K} \times F_n
\times F^{-,n}$ with $\gamma$ normal, define normalized transfer factor
\[
    \omega([\gamma, x, y]) =
    \eta'(\xi)\Omega([\xi, x, y]),
\]
where $\xi$ is an element of $\GL_n(L)$ such
that $\gamma = \xi\sigma_{L/K}(\xi)^{-1}$.
This does not depend on the choice of $\xi$, and
satisfies
\[
    \omega([h^{-1}\gamma h, xh, h^{-1}y])
    = \eta_{E/F}(h)\omega([\gamma, x, y])
\]
for $h \in \GL_n(F)$.
\end{definition}

\begin{definition}[Smooth transfer] \label{defn: smooth transfer}
We say that $(f', \phi')$ and $\{(f^V, \phi_1^V \otimes 
\phi_2^V)\}_{V \in \SHerm_n^\times(F)}$ are
\emph{$(S, R^V)$-smooth transfer} of each other if 
\[
    \Orb([\gamma, x, y], \widetilde f', \phi')
    = \omega([\gamma, x, y])
    \Orb([\zeta, z]^V, f^V, \phi_1^V\otimes \phi_2^V).
\]
for all matching $[\gamma, x, y] \in X_\rss$
and $[\zeta, z]^V \in Y_\rss^V$.
This is equivalent to
\[
    \Orb([\xi, x, y], f', \phi')
    = \Omega([\xi, x, y])
    \Orb([\zeta, z]^V, f^V, \phi_1^V\otimes \phi_2^V),
\]
where $\gamma = \xi\sigma_{L/K}(\xi)^{-1}$. 
In particular, it does not depend on the choice of
$\eta'$.
\end{definition}

\section{Matching of local orbital integrals} \label{section: matching orbits and integrals}

Given geometric decompositions in previous section, in this section we show matchings of local orbital integrals under compatible representatives (Definition \ref{def:compatible}) for the purpose of global comparison of relative trace formulas.

In this section, all objects are over $F_v$ for a local place $v$ of $F$, and we suppress $v$ from the notation.

\subsection{Kottwitz representatives}
Suppose that $E$ and $K$
are both quadratic \'etale $F$-algebras. 
In this subsection, we give definitions of Kottwitz,
$k$-Kottwitz, and compatible representatives.

In the subsequent subsections, we will prove the
fundamental lemma and existence of smooth
transfer for certain choices of representatives,
in the following cases.

\begin{itemize}
    \item In Section \ref{sec:FL}, we prove the 
    fundamental lemma when $E/F$ is inert 
    and $K/F$ is inert or split, under the assumption
    that every element of $S$ and $R^V$ is
    Kottwitz.

    \item In Section \ref{sec:splittransfer},
    we prove both the fundamental lemma and the
    existence of smooth transfer when $E/F$
    is split, and $K/F$ is ramified, inert, or split,
    under the assumption that $S$ and $R^V$
    are compatible.
    
    \item In Section \ref{sec:Ksplittransfer},
    we prove the existence of smooth transfer  
    when $E/F$ ramified, inert, or split, 
    and $K/F$ split, under the assumption
    that every element of $S$ and $R^V$ is
    $k$-Kottwitz, for a suitable choice of $k \ge 0$
    depending on the test functions.
\end{itemize}

\begin{center}
\begin{tabular}{|c|c|c|c|}
\hline
    \diagbox[width=7em]{$E/F$}{$K/F$} & inert & split & ramified \\\hline
    inert & FL for Kottwitz reps. & 
    \begin{tabular}{@{}c@{}}
        FL for Kottwitz reps. \\ Transfer for $k$-Kottwitz reps.
    \end{tabular} & \\\hline
    split & \begin{tabular}{@{}c@{}}
        FL and transfer for \\ compatible reps.
    \end{tabular} & \begin{tabular}{@{}c@{}}
        FL and transfer for \\ compatible
        reps.
    \end{tabular}& \begin{tabular}{@{}c@{}}
        Transfer for \\ compatible
        reps.
    \end{tabular}\\\hline
    ramified & \begin{tabular}{@{}c@{}}
          \hspace{1em} \\ \hspace{1em} 
    \end{tabular} & Transfer for $k$-Kottwitz reps. & \\\hline
\end{tabular}
\end{center}

\begin{definition}
\label{def:kottwitzgeneral}
We say that an element $\xi \in \GL_n(L)$ is \emph{Kottwitz}
if either of the following holds.
\begin{enumerate}[(1)]
\item We have
\[
    \xi = c_{n-1}(\xi\sigma_{L/E}(\xi)^{-1})^{n-1}
    + \cdots + c_0
\]
for some $c_0, \dots, c_{n-1} \in \O_L$.
\item 
There does not exist $g \in \GL_n(L)$ such that
$g^{-1}\xi\sigma_{L/E}(\xi)^{-1} g \in \GL_n(\O_L)$.
\end{enumerate}
\end{definition}

\begin{definition}
\label{def:kkottwitz}
Let $\varpi_F$ be a uniformizer of $F$, and let
$k \ge 0$ be an integer. We say that an element
$(\xi_1, \xi_2) \in \GL_n(E) \times \GL_n(E)$
is \emph{$k$-Kottwitz} if 
\[
    \xi_1 = c_{n-1}(\xi_1\xi_2^{-1})^{n-1}
    + \cdots + c_0
\]
for some $c_0 \in 1 + \varpi_F^k \O_E$ 
and $c_1, \dots, c_{n-1} \in \varpi_F^k \O_E$.
\end{definition}

\begin{lem}
\label{lem:kottwitzproperties}
Suppose $\xi \in \GL_n(L)$ is Kottwitz, and
an element of either $\S_{n,L/K}$ or $\U(V)(K)$.
\begin{enumerate}[\normalfont(1)]
    \item 
    \label{kot}
    If $g^{-1}\xi\sigma_{L/E}(\xi)^{-1} g
    \in \GL_n(\O_L)$ for $g \in \GL_n(L)$, then
    $g^{-1}\xi g \in \GL_n(\O_L)$.
    
    \item
    \label{kot2}
    If $h_1^{-1}\xi h_2 \in \GL_n(\O_L)$ for 
    $h_1, h_2 \in \GL_n(E)$, then $h_2^{-1}h_1 \in 
    \GL_n(\O_E)$.
    
    \item
    \label{glkot1} Assume that $E/F$ is unramified.
 If $g^{-1}\xi\sigma_{E/F}(g)
    \in \GL_n(\O_L)$ for $g \in \GL_n(E)$, then
    $g \in \GL_n(F)\GL_n(\O_E)$.
\end{enumerate}
\end{lem}
\begin{proof}
Since $\xi$ is Kottwitz,
$g^{-1}\xi\sigma_{L/E}(\xi)^{-1} g \in \GL_n(\O_L)$ 
implies $g^{-1}\xi g \in \Mat_n(\O_L)$.
If $\xi \in \S_{n,L/K}$, then $\sigma_{L/E}(\xi)^{-1}
= \sigma_{L/M}(\xi)$, and if $\xi \in \U(V)(K)$,
we have $\sigma_{L/E}(\xi)^{-1} = \sigma_{L/E}(\xi^\ast)$.
In either case, $\det \xi \in \O_L$ implies
$\det(\sigma_{L/E}(\xi)^{-1}) \in \O_L$ as well.
Since $\det\xi \det(\sigma_{L/E}(\xi)^{-1}) 
\in \O_L^\times$, we have $\det \xi \in \O_L^\times$,
so $g^{-1}\xi g \in \GL_n(\O_L)$.
This proves \eqref{kot}.

\eqref{kot2} follows from the fact 
$h_1^{-1}\xi h_2 \in \GL_n(\O_L)$ implies
$h_1^{-1}\xi\sigma_{L/E}(\xi)^{-1}
h_1 \in \GL_n(\O_L)$, which implies $h_1^{-1}\xi h_1
\in \GL_n(\O_L)$ by \eqref{kot}.
\eqref{glkot1} follows from \eqref{kot2}
and the fact that if $g^{-1}\sigma_{E/F}(g)
\in \GL_n(\O_E)$ for $g \in \GL_n(E)$, then
$g \in \GL_n(F)\GL_n(\O_E)$ by 
\cite[Lemma~8.7]{Kottwitz1980}.
\end{proof}

\begin{lem}
\label{lem:kottwitzexist}
Assume that $E/F$ and $K/F$ are both inert.
Let $\gamma_0 \in \S_{n,L/K}$ be normal, and let
$\delta = \gamma_0\sigma_{L/E}(\gamma_0)^{-1}$.
Suppose there exists $g \in \GL_n(L)$ such that
$g^{-1}\delta g \in \GL_n(\O_{L})$. 
Then there exists $\gamma \in \S_{n,L/K}$ with 
$\gamma\sigma_{L/E}(\gamma)^{-1} = \delta$ such
that $\gamma$ is Kottwitz.
\end{lem}
\begin{proof}
  Since $K = E$, we have $L = E \times E$ 
  and $\gamma_0 = (\gamma_{01}, \sigma_{E/F}(\gamma_{01})^{-1})$
  for some $\gamma_{01} \in \GL_n(E)$.
  Then $\delta = (\delta_1, \delta_1^{-1})
  = (\gamma_{01}\sigma_{E/F}(\gamma_{01}),
  \sigma_{E/F}(\gamma_{01})^{-1}\gamma_{01}^{-1})$,
  and $\delta_1 \in \GL_n(F)$ since $\gamma_0$ is normal.
  The assumption that 
  $g^{-1}\delta g \in \GL_n(\O_{L})$ for $g = (g_1, g_2) \in
  \GL_n(E) \times \GL_n(E)$ implies that
  $g_1^{-1} \delta_1 g_1
  \in \GL_n(\O_{E})$.
  By {\cite[Lemma 5.4]{TGGP-Wang}} (which is essentially
  {\cite[Lemma 8.8]{Kottwitz1980}}), there exists   
  $\gamma_1 \in \GL_n(E)$ with 
  $\gamma_1\sigma_{E/F}(\gamma_1) = \delta_1$
  such that
  \[
    \gamma_1 = 
    c_{n-1}(\gamma_1\sigma_{E/F}(\gamma_1))^{n-1} 
    + \cdots + c_0
  \]
  for $c_0, \dots, c_{n-1} \in \O_{E}$.
  Since the characteristic polynomial of
  $\gamma_{01}\sigma_{E/F}(\gamma_{01})$ has coefficients
  in $\O_{E}$, this implies that
  $\sigma_{E/F}(\gamma_1)^{-1}$ also has the above
  property,
  so $\gamma = (\gamma_1, \sigma_{E/F}(\gamma_1)^{-1})$
  is Kottwitz as an element of $\GL_n(L)$.
\end{proof}

In the case that $K = F \times F$, we fix an
identification of $L$ with $E \times E$ such
that $\sigma_{L/E}(\xi_1, \xi_2) = (\xi_2, \xi_1)$.

\begin{lem}
\label{lem:kkottwitzexist}
Assume that $K/F$ is split ($E/F$ may be ramified), 
and let $k \ge 0$ be an integer. 
Suppose $\gamma_0 \in \S_{n,L/K}$ is normal,
and let $\delta = \gamma_0\sigma_{L/E}(\gamma_0)^{-1}$.
Then there exists $\gamma \in \S_{n,L/K}$ with
$\gamma\sigma_{L/E}(\gamma)^{-1} = \delta$ such
that $\gamma$ is both Kottwitz and $k$-Kottwitz.
\end{lem} 
\begin{proof}
In this case, we have $L = E \times E$ and 
$\gamma_0 = (\gamma_{01}, \gamma_{02})$ for some
$\gamma_{01}, \gamma_{02} \in \GL_n(E)$ satisfying
$\gamma_{01}\sigma_{E/F}(\gamma_{01}) 
= \gamma_{02}\sigma_{E/F}(\gamma_{02}) = 1$,
and $\delta = (\gamma_{01}\gamma_{02}^{-1},
\gamma_{02}\gamma_{01}^{-1})$.
Let $\gamma = (1, \gamma_{02}\gamma_{01}^{-1})$,
which is Kottwitz because
\[
    \gamma = (0, 1)\delta + (1, 0).
\]
In addition, $\gamma \in \S_{n,L/K}$ because
$\gamma_{01}$ and $\gamma_{02}$ commute, since
$\gamma$ is normal.
\end{proof}

\begin{definition}
\label{def:compatible}
    We say that sets of normal representatives (c.f. \ref{normal reps: unitary}, \ref{normal reps: GL side})
    $S$ and $R^V$ are \emph{compatible} if
    for all $\gamma \in S$, $\zeta \in R^V$,
    and $h \in \GL_n(E)$ such that
    \[
        h^{-1}\gamma\sigma_{L/E}(\gamma)^{-1} h =
        \zeta\sigma_{L/E}(\zeta)^{-1},
    \]
    we also have
    \[
        h^{-1}\gamma h = \zeta.
    \]
\end{definition}

It is clear that if $S$ and $R^V$ are compatible sets
of representatives,
and every element of $S$ is Kottwitz (resp. $k$-Kottwitz),
then every element of $R^V$ is Kottwitz (resp.
$k$-Kottwitz) as well.

\begin{proposition}
\label{lem:compatibleexist}
For every set of representatives $S$, there exists
a set of representatives $R^V$ which is
compatible with $S$.
\end{proposition}
\begin{proof}
Let $S$ be a set of normal representatives of the
regular semisimple $\GL_n(E)$-orbits of $\S_{n, L/K}$.
Recall that $R^V$ is a set of representatives
of the regular semisimple $\U(V)(F) \times
\U(V)(F)$-orbits of $\U(V)(K)$ which
contain a normal element.

Suppose $\zeta'$ is a normal element
of $\U(V)(K)$. 
We choose the representative of
$\U(V)(F)\zeta' \U(V)(F)$ as follows.
By Proposition \ref{lem:matching},
there exists $\gamma \in S$ such that 
\[
    h^{-1}\gamma\sigma_{L/E}(\gamma)^{-1} h
    = \zeta' \sigma_{L/E}(\zeta')^{-1}
\]
for some $h \in \GL_n(E)$.
Then by Lemma \ref{lem:normmatch} (1), 
$\zeta \coloneqq h^{-1}\gamma h$ is a normal
element of $\U(V)(K)$ which satisfies
$\zeta\sigma_{L/E}(\zeta)^{-1} =
\zeta'\sigma_{L/E}(\zeta')^{-1}$, so
$\zeta \in \U(V)(F)\zeta' \U(V)(F)$.

If $h'$ is another element of $\GL_n(E)$ such that
\[
     h'^{-1}\gamma\sigma_{L/E}(\gamma)^{-1} h'
    = \zeta\sigma_{L/E}(\zeta)
    = h^{-1}\gamma\sigma_{L/E}(\gamma)^{-1} h,
\]
then $h'^{-1}\gamma h' = h^{-1}\gamma h = \zeta$ because
the centralizer of $\gamma\sigma_{L/E}(\gamma)^{-1}$
is a torus which contains $\gamma$, by normality.
Thus, taking $\zeta$ to be the representative
of the orbit $\U(V)(F)\zeta' \U(V)(F)$, we obtain
a set of representatives $R^V$ compatible with $S$.
\end{proof}

\subsection{Fundamental lemma when $E/F$ and $K/F$ are unramified}
\label{sec:FL}
In this subsection, suppose that $E$ is an unramified 
quadratic extension of  $F$ and $K$ is an unramified 
quadratic \'etale  $F$-algebra. 
Let $V^+$ denote the split skew-Hermitian space of dimension
$n$ over $E$. 

Let $\Lambda$ be a self-dual lattice of $(V^+)^\vee$.
Let $\LL^+ + \LL^{+, \vee}$ be any polarization of
$(\Res_{E/F} V^+)^\vee$. Let 
\[
    \sum_{i=1}^r \ol{\phi_1^{(j)}}\otimes \phi_2^{(j)} 
\in \cS(\LL^+)^{\otimes 2}
\]
be the inverse partial Fourier transform of
$\mathbf{1}_{\Lambda} \in \cS((V^+)^\vee)$. 
Note that $\sum_{j=1}^r \ol{\phi_1^{(j)}}
\otimes \phi_2^{(j)}$ is fixed by 
$\U(V^+)(\O_F) \times \U(V^+)(\O_F)$ under the
representation $\ol{\omega}\otimes \omega$ of
$\U(V^+) \times \U(V^+)$

Below, we take $(f', \phi')$ and $(f, \phi_1 \otimes \phi_2)$
to be the \emph{unramified test functions}
\[
    (f', \phi') = (\mathbf{1}_{\GL_n(\O_L)}, \mathbf{1}_{\O_{E,n}})
\]
and
\[
    (f, \phi_1 \otimes \phi_2) =
    \left(\mathbf{1}_{\U(V^+)(\O_K)}, \sum_{j=1}^r
    \phi_1^{(j)} \otimes \phi_2^{(j)}\right).
\]
If $\eta'$ is unramified, then we have
\[
    (\widetilde f', \phi') = (\mathbf{1}_{\S_{n,\O_L/\O_K}},
    \mathbf{1}_{\O_{E,n}}),
\]
where $\S_{n,\O_L/\O_K} \coloneqq \S_{n, L/K}
\cap \GL_n(\O_L)$.

\begin{lem}
\label{lem:GLFLred}
    Suppose that $[\gamma, x, y] \in X_\rss$ and 
    $\gamma$ is Kottwitz.
    Suppose $\eta'$ is unramified. Then
    \begin{align*}
        &\Orb([\gamma, x, y], \widetilde f', \phi')\\
        &= 
        \int_{\GL_n(F)}
        \mathbf{1}_{\S_{n, \O_L/\O_K} \cap \S_{n, L/E}}
        (g^{-1}\gamma\sigma_{L/E}(\gamma)^{-1} g)
        (\phi')^\dagger(xg, g^{-1}y)\eta_{E/F}(g)\,dg.
    \end{align*}
\end{lem}
\begin{proof}
Note that since $\gamma \in \S_{n,L/K}$ is normal, we have
\[
    g^{-1}\gamma\sigma_{L/E}(\gamma)^{-1} g \in
    \S_{n,L/K} \cap \S_{n, L/E}
\]
for all $g \in \GL_n(F)$.
By Lemma \ref{lem:kottwitzproperties} \eqref{glkot1}
and \eqref{kot}, for $g \in \GL_n(E)$, we have 
$g^{-1}\gamma\sigma_{E/F}(g) \in \S_{n, \O_L/\O_K}$
if and only if $g \in \GL_n(F)\GL_n(\O_E)$
and $g^{-1}\gamma\sigma_{L/E}(\gamma)^{-1} g
\in \GL_n(\O_L)$.

Since $\phi'$ is fixed by $\GL_n(\O_E)$, we have
\begin{align*}
    \Orb([\gamma, x, y], \widetilde f', \phi')
    &= \int_{\GL_n(E)} \mathbf{1}_{\S_{n, \O_L/\O_K}}
    (g^{-1}\gamma \sigma_{E/F}(g)) \ol{(R_{\mu}(g)\phi')^\dagger(x, y)}
    \eta'(g)\, dg \\
    &= \int_{\GL_n(F)} \mathbf{1}_{\S_{n, \O_L/\O_K} \cap \S_{n, L/E}}
    (g^{-1} \gamma\sigma_{L/E}(\gamma)^{-1} g)
    \ol{(R_\mu(g)\phi')^\dagger(x, y)}\, dg \\
    &= \int_{\GL_n(F)} \mathbf{1}_{\S_{n, \O_L/\O_K} \cap
    \S_{n, L/E}}
    (g^{-1}\gamma\sigma_{L/E}(\gamma)^{-1} g)
    (\phi')^\dagger(xg, g^{-1} y) \eta_{E/F}(g) \,dg.
\end{align*}
In the second line, we used the fact that
$\eta'|_{F^\times} = 1$ for all choices of $\eta'$,
and $\eta'|_{\O_E^\times} = 1$
since $\eta'$ is unramified, so $\eta'|_{E^\times} = 1$.
In the third line, we used the fact that
\[
    (R_\mu(g)\phi')^\dagger(x, y) = (\phi')^\dagger(xg,
    g^{-1}y)\eta_{E/F}(g)
\]
for all $g \in \GL_n(F)$.
\end{proof}

\begin{lem}
\label{lem:UFLred}
    Suppose that $[\zeta, z]^{V^+} \in Y_\rss^{V^+}$
    and $\zeta$ is Kottwitz. 
    Then
    \begin{align*}
        &\Orb([\zeta, z]^{V^+}, f, \phi_1 \otimes \phi_2)
        \\&= \int_{\U(V^+)(F)}
        \mathbf{1}_{\U(V^+)(\O_K) \cap \S_{n, L/E}}(h^{-1}
        \zeta\sigma_{L/E}(\zeta)^{-1} h)
        (\ol{\phi_1} \otimes \phi_2)^{\ddagger}(zh)\, dh.
    \end{align*}
\end{lem}
\begin{proof}
Note that 
\[
    h^{-1}\zeta\sigma_{L/E}(\zeta)^{-1} h \in 
    \U(V^+)(K) \cap \S_{n, L/E}
\]
for all $h \in \U(V^+)(F)$.
By Lemma \ref{lem:kottwitzproperties} \eqref{kot2}, 
for $g, h \in \U(V^+)(F)$, we have
$g^{-1}\zeta h \in \U(V^+)(\O_K)$ if and only if
$h^{-1} g \in \U(V^+)(\O_F)$ and 
$h^{-1}\zeta h \in \U(V^+)(\O_K)$.
By Lemma \ref{lem:kottwitzproperties} \eqref{kot},
$h^{-1}\zeta h \in \U(V^+)(\O_K)$ if and only if
$h^{-1} \zeta \sigma_{L/E}(\zeta)^{-1} h \in
\U(V^+)(\O_K) \cap \S_{n, L/E}$.

Thus, since $\ol{\phi_1} \otimes \phi_2$ is fixed by 
$\U(V^+)(\O_F) \times \U(V^+)(\O_F)$,
we have
\begin{align*}
    \Orb([\zeta, z]^{V^+}, f, \phi_1 \otimes \phi_2) &= 
    \int_{\U(V^+)(F)} \mathbf{1}_{\U(V^+)(\O_K)}
    (h^{-1}\zeta h) 
    (\ol{\phi_1} \otimes \phi_2)^\ddagger(zh)\, dh \\
    &= \int_{\U(V^+)(F)} \mathbf{1}_{\U(V^+)(\O_K) \cap \S_{n, L/E}}
    (h^{-1} \zeta\sigma_{L/E}(\zeta)^{-1} h) 
    (\ol{\phi_1} \otimes \phi_2)^\ddagger(zh)\, dh.
    \qedhere
\end{align*}
\end{proof}

\begin{proposition}[Fundamental lemma]
\label{prop:FL}
    Suppose $[\gamma, x, y] \in X_\rss$
    matches with $[\zeta, z] \in Y_\rss^{V^+}$,
    and that $\gamma$ and $\zeta$ are Kottwitz.
    Then
    \[
        \Orb([\gamma, x, y], \widetilde f', \phi') =
        \omega([\gamma, x, y])
        \Orb([\zeta, z], f, \phi_1 \otimes \phi_2).
    \]
    If $[\gamma, x, y]$ matches with
    $[\zeta, z] \in Y_\rss^{V^-}$, with
    $\gamma$ and $\zeta$ Kottwitz, then
    \[
        \Orb([\gamma, x, y], \widetilde f', \phi') = 0.
    \] 
\end{proposition}   
\begin{proof}
If $K = F \times F$, we identify
$\S_{n,L/K} \cap \S_{n,L/E}$ with $\S_{n, E/F}$
and $\U(V^+)(K) \cap \S_{n, L/E}$ with $\U(V^+)(F)$.
Then by Lemmas \ref{lem:GLFLred} and \ref{lem:UFLred}, 
the proposition
follows from \cite[Proposition~5.14]{Liu2014relative} and
\cite[Theorem~5.15]{Liu2014relative},
the fundamental lemma in the Fourier--Jacobi case.

If $K = E$, the proposition follows from
\cite[Proposition~4.18]{TGGP-Wang} and
\cite[Theorem~4.19]{TGGP-Wang},
the fundamental lemma in the $E = K$ case.
\end{proof}

\subsection{Existence of smooth transfer when $E/F$ is split}
\label{sec:splittransfer}
Suppose $E = F \times F$. Assume that the skew-Hermitian
form on $V = F^n \times F^n$ is given by
\[
    \langle v, w \rangle =
    (\transp{w}v, -\transp{v}w).
\]
We have $\GL_n(L) = \GL_n(K) \times \GL_n(K)$,
and we identify
\[
    \U(V)(K) = \{(g, \transp{g^{-1}}) : g
    \in \GL_n(K)\}
\]
with $\GL_n(K)$ via projection to the first coordinate.
We choose the polarization of
$V^\vee \cong F_n \times F_n$ with
$\LL = F_n \times \{0\}$ and
$\LL^\vee = \{0\} \times F_n$,
and identify $\LL$ with $F_n$ in the obvious way.

\begin{definition}
\label{def:splittransfer}
We say that a test function $(f', \phi') 
\in C_c^\infty(\GL_n(K) \times
\GL_n(K)) \otimes \cS_n(E_n)$ and
a test function $(f, \phi_1 \otimes \phi_2) \in
C_c^\infty(\GL_n(K)) \otimes \cS(F_n)^{\otimes 2}$
are \emph{split smooth transfer} of each other 
if there exists $f_1', f_2' \in C_c^\infty(\GL_n(K))$
such that 
\[
    f = f_1' \ast f_2'^\vee, \quad
    f' = f_1' \otimes f_2', \quad
    \phi' = \phi_1 \otimes \ol{\phi_2},
\]
where $f^\vee$ is the function given by
$f^\vee(g) \coloneqq f(g^{-1})$ and $\ast$
denotes convolution.
\end{definition}

\begin{proposition}
\label{prop:splittransfer}
Suppose that $(f', \phi')$ and $(f, \phi_1 \otimes
\phi_2)$ are split smooth transfer of each
other as in Definition \ref{def:splittransfer}.
Assume that $S$ and $R^V$ are compatible as in
Definition \ref{def:compatible}. For $[\gamma, x, y] \in X_\rss$
and $[\zeta, z]^V \in Y_\rss^V$ which match
in the sense of Proposition \ref{lem:matching}, 
we have
\[
    \Orb([\xi, x, y], f', \phi')
    = \Orb([\zeta, z]^V, f, \phi_1 \otimes \phi_2)
\]
for any $\xi = (\xi_1, \xi_2) \in \GL_n(L)$ such
that $\gamma = (\gamma_1, \gamma_1^{-1}) = 
\xi\sigma_{L/K}(\xi)^{-1} = (\xi_1\xi_2^{-1},
\xi_2\xi_1^{-1})$.
\end{proposition}   

\begin{proof}
Write $x \in F_n$, $y = (y_0, -y_0)$ for $y_0 \in F^n$.
By the same calculation as in the proof of \cite[Proposition~4.14]{TGGP-Wang}, we have 
\begin{align*}
    \ol{(R_\mu(g_1, g_2)\phi')^\dagger(x, y)}
    = (\ol{\omega(g_2^{-1}g_1, \transp(g_1^{-1}g_2))\phi_1}
    \otimes \phi_2)^\ddagger((x, \transp y_0)(g_2, \transp g_2^{-1})),
\end{align*}
for any $g_1, g_2 \in \GL_n(F)$.
\begin{align*}
    &\Orb([\xi, x, y], f', \phi') \\
    &= \iint\limits_{\GL_n(F) \times \GL_n(F)} 
    \int\limits_{\GL_n(K)}
    f'((g_1^{-1}, g_2^{-1})(\xi_1, \xi_2)(h, h))
    \ol{(R_\mu(g_1, g_2)\phi')^\dagger(x, y)}\, dh \,dg_1\,dg_2 \\
    &= \iint\limits_{\GL_n(F) \times \GL_n(F)}
    \int\limits_{\GL_n(K)} f_1'(g^{-1}\xi_1\xi_2^{-1}g_2 h)f_2'(h)
    \ol{(R_\mu(g_1, g_2)\phi')^\dagger(x, y)}\, dh \,dg_1\,dg_2\\
    &= \iint\limits_{\GL_n(F) \times \GL_n(F)}
    f(g_1^{-1}\gamma_1 g_2)
    (\ol{\omega(g_2^{-1}g_1, \transp(g_1^{-1}g_2))\phi_1} \otimes \phi_2)^\ddagger((x, \transp y_0)(g_2, \transp g_2^{-1})) \,dg_1\,dg_2.
\end{align*}

On the other hand,
let $z = (z_1, z_2) \in F_n \times F_n$.
Since $[\gamma, x, y]$ and $[\zeta, z]^V$, match,
there exists 
$h = (h_1, h_2) \in \GL_n(F) \times \GL_n(F)$
such that
\[
    [\gamma\sigma_{L/E}(\gamma)^{-1}, x, y].h
    = [\zeta\sigma_{L/E}(\zeta)^{-1}, z, z^\ast].
\]
Since $S$ and $R^V$ are compatible, we in fact have
\[
    [\gamma, x, y].h = [\zeta, z, z^\ast],
\]
which implies
\begin{equation}
\label{eq:changevar}
    h_1^{-1}\gamma_1 h_1 = \zeta_1, \quad
    xh_1 = z_1, \quad \transp{y_0}\transp{h_1^{-1}} = z_2.
\end{equation}
Then the local obrital integral \eqref{defn: local orb: unitary side}
\begin{align*}
    &\Orb([\zeta, z], f, \phi_1 \otimes \phi_2) \\
    &= \iint\limits_{\GL_n(F) \times \GL_n(F)}
    f(g_1^{-1}\zeta_1 g_2)
    (\ol{\omega(g_2^{-1}g_1,
    \transp(g_1^{-1}g_2))\phi_1} \otimes \phi_2)^\ddagger((z_1, z_2)(g_2,
    \transp g_2^{-1}))\, dg_1\,dg_2
\end{align*}
is easily seen to be equal to our above expression
for $\Orb([\xi, x, y], f', \phi')$,
by \eqref{eq:changevar}.
\end{proof}

From now on, when we say that test functions
$(f', \phi')$ and $(f, \phi_1 \otimes \phi_2)$ are
smooth transfer of each other at a place $v$ which
is split in $E$, we mean that they are split smooth
transfer of each other, as in Definition 
\ref{def:splittransfer}.
For such test functions,
it is easy to see that $(f', \phi')$ satisfies
condition \eqref{glmatrixcoeff} or \eqref{glrsssupp}
of Definition \ref{def:glgood}
if and only if
$(f, \phi_1 \otimes \phi_2)$ satisfies condition
\eqref{umatrixcoeff} or \eqref{ursssupp}
of Definition \ref{def:ugood}, respectively.

\subsection{Existence of smooth transfer when $K/F$ is split at non-archimedean places}
\label{sec:Ksplittransfer}

In the case $K = F \times F$, we have
\[
    G' = \Res_{E/F}(\GL_n\times \GL_n), \quad 
    G = \U(V) \times \U(V)\quad
    \S_{n, L/K} = \S_{n, E/F} \times \S_{n, E/F}.
\]

First, we recall the definitions of
the Fourier--Jacobi orbital integrals, and the
existence of smooth transfer in the Fourier--Jacobi
case.
Let $(G'(F) \times F_n \times F^{-,n})_\rss$ denote the
set of elements $[(\xi_1, \xi_2), x, y]$ such that
\[
    [\xi_1^{-1}\xi_2\sigma_{E/F}(\xi_1^{-1}\xi_2)^{-1}, x, y]
\]
is a strongly regular semisimple element of $\Mat_n(E) \times E_n \times E^n$.
Let $(G(F) \times V^\vee)_\rss$ denote the set
of elements $[(\zeta_1, \zeta_2), z]$ such that
\[
    [\zeta_1^{-1}\zeta_2, z, z^\ast]
\] 
is a strongly regular semisimple element of
$\Mat_n(E) \times E_n \times E^n$.

\begin{definition}[{\cite[(3.1.1)]{Xue-GGP}}, cf. Definition \ref{defn: local orb GL side}]

For $[(\xi_1, \xi_2), x, y] \in (G'(F)
\times F_n \times F^{-,n})_\rss$, 
$f' \in C_c^\infty(G'(F))$, 
and $\phi' \in \cS(E_n)$, define
\begin{align*}
    &\Orb^{\FJ}([(\xi_1, \xi_2), x, y], f', \phi')\\
    &= \int_{\GL_n(E)}
    \iint_{\GL_n(F)\times\GL_n(F)}
    f'(g^{-1}(\xi_1, \xi_2)(h_1, h_2))
    \ol{(R_\mu(\xi_1^{-1} g)\phi')^\dagger
    (x, y)}
    \eta_{E/F}(h_1h_2)^{n+1}\,dh_1\,dh_2\,dg.
\end{align*}
Note that this orbital integral satisfies
\[
    \Orb^{\FJ}([g^{-1}(\xi_1, \xi_2)(h_1, h_2), xh_1,
    h_1^{-1}y])
    = \eta_{E/F}(h_1)^n\eta_{E/F}(h_2)^{n+1}
    \Orb^{\FJ}([(\xi_1, \xi_2), x, y])
\]
for $g \in H_1(F)$, $(h_1, h_2) \in H_2(F)$.
\end{definition}

\begin{definition}[{\cite[(4.1.2)]{Xue-GGP}}, cf.
Definition \ref{defn: local orb: unitary side}]
For $[(\zeta_1, \zeta_2), z] 
\in (G(F) \times V^\vee)_\rss$,
$f \in C_c^\infty(G(F))$, and $\phi_1 \otimes
\phi_2 \in \cS(\LL)^{\otimes 2}$, define
\begin{align*}
    \Orb^{\FJ}([\zeta, z], f, \phi_1 \otimes
    \phi_2) =
    \iint_{\U(V)(F) \times \U(V)(F)}
    f(g^{-1}\zeta h)
    (\ol{\omega(h^{-1}
    \zeta_1^{-1} g)\phi_1} \otimes \phi_2)^\ddagger(zh)\,dh\,dg.
\end{align*}
\end{definition}

\begin{definition}
For $[(\xi_1, \xi_2), x, y] \in G'(F) \times
F_n \times F^{-,n}$, define transfer factor
\[
    \Omega^{\FJ}([(\xi_1, \xi_2), x, y])
    = \Omega([(1, \xi_1^{-1}\xi_2), x, y]).
\]
Note that by \eqref{eq:Omegainvariance}, this
transfer factor satisfies
\begin{align*}
    \Omega^{\FJ}([g^{-1}(\xi_1, \xi_2)(h_1, h_2),
    xh_1, h_1^{-1}y])
    &= \Omega([(1, h_1^{-1}\xi_1^{-1}\xi_2 h_2),
    xh_1, h_1^{-1} y]) \\
    &= \Omega([h_1^{-1}(1, \xi_1^{-1}\xi_2)
    (h_1, h_2), xh_1, h_1^{-1}y]) \\
    &= \eta_{E/F}(h_1)^n\eta_{E/F}(h_2)^{n+1}
    \Omega^{\FJ}([(\xi_1, \xi_2), x, y])
\end{align*}
for $g \in H_1(F)$, $(h_1, h_2) \in H_2(F)$.
\end{definition}

\begin{theorem}[{\cite[Proposition 5.2.1]{Xue-GGP}},
non-simplified version]

\label{thm:FJtransfer}
For any test function $(f', \phi')$,
there exists a collection of test functions
$\{(f^V, \phi_1^V \otimes \phi_2^V)\}_{V \in \SHerm_n^\times(F)}$ such that
\[
    \Orb^{\FJ}([(\xi_1, \xi_2), x, y], f', \phi')
    = 
    \Omega^{\FJ}([(\xi_1, \xi_2), x, y])
    \Orb^{\FJ}([(\zeta_1, \zeta_2), z]^V, f^V, 
    \phi_1^V \otimes \phi_2^V)
\]
for all $[(\xi_1, \xi_2), x, y] \in 
(G'(F) \times F_n \times F^{-,n})_\rss$
and $[(\zeta_1, \zeta_2), z]^V 
\in (G(F) \times V^\vee)_\rss$ which match 
in the sense that
\[
    [(\xi_1^{-1}\xi_2)\sigma_{E/F}(\xi_1^{-1}\xi_2)^{-1},
x, y].h = [\zeta_1^{-1}\zeta_2, z, z^\ast]
\]
for some $h \in \GL_n(E)$.
The other direction also holds.
\end{theorem}

We will reduce the existence of smooth transfer
for our orbital integrals to Theorem
\ref{thm:FJtransfer} for suitable choices of
representatives.

Recall that $\eta'$ is a character of $L^\times$
whose restriction to $K^\times$ is equal
to $\eta_{L/K}^{n+1}$, and that in order to simplify the orbital integral, we defined in \eqref{eq:ftilde} the function $\wt f'
\in C_c^\infty(\S_{n,L/K})$.
Since $K = F \times F$, we have that
$\eta'$ is of the form
$\eta_1' \otimes \eta_2'$ where $\eta_i'$
is a character of $E^\times$ whose restriction to 
$F^\times$ is equal to $\eta_{E/F}^{n+1}$.
If $f' = f_1' \otimes f_2'$ for $f_i' \in 
C_c^\infty(\GL_n(E))$, then
$\wt f' = \wt f_1' \otimes \wt f_2'$, where
$\wt f_i' \in C_c^\infty(\S_{n,E/F})$
is given by
\[
    \wt f_i'(\gamma) 
    =\eta_i'(\xi) \int_{\GL_n(F)} f_i'(\xi h)
    \eta_{E/F}(h)^{n+1}\, dh
\]
for $\xi \in \GL_n(E)$ such that $\gamma
= \xi\sigma_{E/F}(\xi)^{-1}$.

Recall that we have defined in Definition
\ref{def:normalgeneral} the notion
of normal elements of $\GL_n(L)$. In the case
$K = F \times F$, an element $(\xi_1, \xi_2)
\in \GL_n(L) = \GL_n(E) \times \GL_n(E)$ is normal
if and only if $\xi_1$ and $\xi_2$ commute.

\begin{proposition}[Existence of smooth transfer, 
simplified version]
\label{prop:Ksplittransfer}
For any test function $(f', \phi')$,
there exists a collection of test functions 
$\{(f^V, \phi_1^V \otimes \phi_2^V)\}_{V \in \SHerm_n^\times(F)}$
and an integer  $k \ge 0$ such that if 
$[(\gamma_1, \gamma_2), x, y] \in 
(\S_{n, L/K} \times F_n \times F^{-,n})_\rss$ and
$[(\zeta_1, \zeta_2), z]^V \in (G(F) \times V^\vee)_\rss$ 
satisfy
\[
    [(\gamma_1 \gamma_2^{-1}, 
    \gamma_2 \gamma_1^{-1}), x, y].h
    = [(\zeta_1 \zeta_2^{-1}, \zeta_2 \zeta_1^{-1}), 
    z, z^\ast]
\]
for some $h \in \GL_n(E)$, and $(\gamma_1, \gamma_2)$ 
and $(\zeta_1, \zeta_2)$ are normal and $k$-Kottwitz,
then we have
\begin{equation}
\label{eq:transferexistence}
    \Orb([(\gamma_1, \gamma_2), x, y], \wt f', \phi')
    = \omega([(\gamma_1, \gamma_2), x, y])
    \Orb([(\zeta_1, \zeta_2), z]^V, f^V, \phi_1^V \otimes \phi_2^V),
\end{equation}
where the orbital integrals are given by
\eqref{defn: semip orb int on GL side} and
\eqref{defn: local orb: unitary side}.
The other direction also holds.
\end{proposition}
\begin{proof}
Since \eqref{eq:transferexistence} does not
depend on the choice of $\eta'$ (although
both $\wt f'$ and 
$\omega([(\gamma_1, \gamma_2), x, y])$ do), 
for convenience, we assume that $\eta_1' = \eta_2'$.
Let $\xi_1, \xi_2, \xi \in \GL_n(E)$ be such
that $\gamma_1 = \xi_1\sigma_{E/F}(\xi_1)^{-1}$,
$\gamma_2 = \xi_2\sigma_{E/F}(\xi_2)^{-1}$,
and $\gamma_1^{-1}\gamma_2 = \xi\sigma_{E/F}(\xi)^{-1}$.
The transfer factor then satisfies
\begin{align}
    \omega([(\gamma_1, \gamma_2), x, y])
    &= \eta_1'(\xi_1\xi_2)
    \eta_{E/F}(\xi_2\sigma_{E/F}(\xi_1)\xi^{-1})^{-(n-1)}\Omega([(1, \xi), x, y])\nonumber\\
    &= \eta_1'(\xi_1\sigma_{E/F}(\xi_1)^{-1})
    \eta_1'(\xi)
    \Omega([(1, \xi), x, y])\nonumber\\
    &= \eta_1'(\gamma_1) \eta_1'(\xi)
    \Omega^{\FJ}([(1, \xi), x, y]). \label{eq:fjtf}
\end{align}

Let $\{(f, \phi_1 \otimes \phi_2)\}_{V \in 
\SHerm_n^\times(F)}$ be a collection of test
functions on the unitary side which matches
with $(f', \phi')$ as in Theorem
\ref{thm:FJtransfer}.
Suppose $f^V = f_1^V \otimes f_2^V$ for
$f_i^V \in C_c^\infty(\U(V)(F))$.
For $m \ge 1$ an integer, let
\[
    K_m = 1 + \varpi_F^m \Mat_n(\O_E).
\]
Since $\wt f_i'$ and $f_i^V$ are
compactly supported locally constant functions,
there exists $m \ge 1$ such that
\begin{enumerate}[(1)]
\item
for all $\gamma \in \S_{n,E/F}$ and 
$g \in K_m \cap (\S_{n,E/F})\gamma^{-1}$,
we have
\[
    \wt f_1'(g\gamma) = \wt f_1'(\gamma),\quad
    \wt f_2'(g\gamma) = \wt f_2'(\gamma);
\]   

\item
for all $V \in \SHerm_n^\times(F)$, 
$\zeta \in \U(V)(F)$, and 
$g \in K_m \cap \U(V)(F)$, we have
\[
    f_1^V(g\zeta) = f_1^V(\zeta),\quad
    f_2^V(g\zeta) = f_2^V(\zeta);
\]

\item \label{item:eta}
$\eta_1'(g) = 1$ for all $g \in K_m$.
\end{enumerate}

Then let $k \ge 0$ be an integer with the following
properties.
\begin{enumerate}[(1)]
\item
If 
\[
    g^{-1}\gamma_1 \sigma_{E/F}(g) \in 
    \Supp(\wt f_1'), \quad
    g^{-1}\gamma_2\sigma_{E/F}(g) \in \Supp(\wt f_2')
\]
for $k$-Kottwitz $(\gamma_1, \gamma_2)
\in \S_{n,E/F} \times \S_{n,E/F}$
and $g \in \GL_n(E)$, then 
$g^{-1}\gamma_1 g \in K_m$.

\item 
For all $V \in \SHerm_n^\times(F)$, if 
\[
    g^{-1}\zeta_1 h \in \Supp(f_1^V),\quad
    g^{-1}\zeta_2 h \in \Supp(f_2^V)
\]
for $k$-Kottwitz $(\zeta_1, \zeta_2)
\in \U(V)(F) \times \U(V)(F)$ and
$g, h \in \U(V)(F)$, 
then $g^{-1}\zeta_1 g \in K_m$.
\end{enumerate}
Indeed, since $\Supp(\wt f_i')$ 
and $\Supp(f_i)$ are compact subsets of $\GL_n(E)$,
we may choose $k \ge m$ such that
\[
    \varpi_F^k (\Supp(\wt f_1')\Supp(\wt
    f_2')^{-1})^i \subseteq \varpi_F^m \Mat_n(\O_E)
\]
and
\[
    \varpi_F^k (\Supp(f_1)\Supp(f_2)^{-1})^i
    \subseteq \varpi_F^m \Mat_n(\O_E)
\]
for $1 \le i \le n - 1$.
Then $k$ satisfies the desired properties because
for $k$-Kottwitz $(\gamma_1, \gamma_2)$, we have
\[
    g^{-1}\gamma_1 g
    = c_{n-1}(g^{-1}\gamma_1\gamma_2^{-1} g)^{n-1}
    + \cdots + c_0
\]
with $c_0 \in 1 + \varpi_F^k\O_E$
and $c_1, \dots, c_{n-1} \in \varpi_F^k\O_E$,
and similarly for $k$-Kottwitz $(\zeta_1, \zeta_2)$.

Then $k$ and the test functions 
$\{(f, \phi_1 \otimes \phi_2)\}_{V \in \SHerm_n^\times(F)}$ satisfy the
properties required in the proposition. 
Indeed, suppose that $(\gamma_1, \gamma_2)$
is normal, i.e., $\gamma_1$ and $\gamma_2$
commute, and $k$-Kottwitz. Since
we also assumed that $\eta_1' = \eta_2'$, we have
\begin{align*}
&\Orb([(\gamma_1, \gamma_2), x, y], \wt f', \phi') \\
&= \int_{\GL_n(E)} 
\wt f_1'(g^{-1}\gamma_1 \sigma_{E/F}(g)) 
\wt f_2'(g^{-1}\gamma_2 \sigma_{E/F}(g))
\ol{(R_\mu(g)\phi')^\dagger(x, y)}
\eta_1'(g)\eta_2'(g)\,dg\\
&= \int_{\GL_n(E)}
\wt f_1'(g^{-1}\sigma_{E/F}(g))
\wt f_2'(g^{-1}\gamma_1^{-1}\gamma_2\sigma_{E/F}(g))
\ol{(R_\mu(g)\phi')^\dagger(x, y)}
\eta_1'(g)\eta_2'(g)\, dg \\
&= \eta_2'(\xi)
\int_{\GL_n(E)}\iint_{\GL_n(F) \times \GL_n(F)}
f_1'(g^{-1}h_1) f_2'(g^{-1}\xi h_2)
\ol{(R_\mu(g)\phi')^\dagger(x, y)}
\eta_{E/F}(h_1h_2)^{n+1}\,dh_1\,dh_2\,dg\\
&= \eta_1'(\xi) \cdot
\Orb^{\FJ}([(1, \xi), x, y], f', \phi')
\end{align*}
where $\xi \in \GL_n(E)$ such that
$\gamma_1^{-1}\gamma_2 = \xi\sigma_{E/F}(\xi)^{-1}$.
On the other hand, if $(\zeta_1, \zeta_2)$
is normal and $k$-Kottwitz, then
\begin{align*}
&\Orb([(\zeta_1, \zeta_2), z]^V, f^V, \phi_1^V \otimes \phi_2^V)\\
&= \iint_{\U(V)(F)\times \U(V)(F)}
f_1^V(g^{-1}\zeta_1 h)f_2^V(g^{-1}\zeta_2 h)
(\ol{\omega_{\psi,\mu}(h^{-1}g)\phi_1^V} \otimes \phi_2^V)^\ddagger(zh)\,dg\,dh\\
&= \iint_{\U(V)(F) \times \U(V)(F)}
f_1^V(g^{-1}h)f_2^V(g^{-1}\zeta_1^{-1}\zeta_2 h)
(\ol{\omega_{\psi,\mu}(h^{-1}g)\phi_1^V} \otimes \phi_2^V)^\ddagger(zh)\,dg\,dh\\
&= \Orb^{\FJ}([(1, \zeta_1^{-1}\zeta_2), z]^V, f^V,
\phi_1^V \otimes \phi_2^V).
\end{align*}

It is clear that if $[(\gamma_1, \gamma_2), x, y]$ and
$[(\zeta_1, \zeta_2), z]^V$ are regular semisimple 
and match in the sense of the proposition
statement, then $[(1, \xi), x, y]$ and 
$[(1, \zeta_1^{-1}\zeta_2), z]^V$
match as in Theorem \ref{thm:FJtransfer}.
Thus, since $\{(f, \phi_1 \otimes \phi_2)\}_{V \in \SHerm_n^\times(F)}$ is a smooth
transfer of $(f', \phi')$ for the Fourier--Jacobi
orbital integrals, we have
\begin{align*}
\Orb([(\gamma_1, \gamma_2), x, y], \wt f', \phi') 
&= \eta_1'(\xi)\Omega^{\FJ}([(1, \xi), x, y])
\Orb([(\zeta_1, \zeta_2), z]^V, f^V, \phi_1^V \otimes \phi_2^V)\\
&= \omega([(\gamma_1, \gamma_2), x, y])
\Orb([(\zeta_1, \zeta_2), z]^V, f^V, \phi_1^V \otimes \phi_2^V),
\end{align*}
where the last equality follows from \eqref{eq:fjtf}
and the assumption \eqref{item:eta}.

The proof for the other direction is the same.
\end{proof}

\begin{comment}

\subsection{Existence of transfers at archimedean places ($K/F$ split)} 

\end{comment}

\section{Comparison of relative trace formulas} \label{section: comparison RTF}

In this section, we are in the global situation to compare relative trace formulas, so $E/F$ and $K/F$ are quadratic extensions of number fields, and $L=E \otimes_F K$. We will use compatible (\ref{def:compatible}) global normal representatives on two sides. The main theorem in this section is Theorem \ref{prop:identity}.

\subsection{Choosing global representatives}
In this subsection, we prove that every regular semisimple
$\GL_n(E)$-orbit of $S_{n,L/K}$ contains a normal element which
is Kottwitz at all places of $F$ which are inert in $E$ and unramified in $K$, and $k$-Kottwitz ($k \geq 0$) at a given finite set of $K$-split places.

Let $\{v_1, \dots, v_m\}$ be a finite set of 
nonarchimedean places
of $F$ which are split in $K$. For each $1 \le i \le m$, 
let $u_i$ and $w_i$ be the places of $K$ lying above 
$v_i$. We identify $L_{v_i} = L_{u_i} \times L_{w_i}$
with $E_{v_i} \times E_{v_i}$, and say an
element of $\GL_n(L_{v_i})$ is $k$-Kottwitz if it is
$k$-Kottwitz in the sense of Definition 
\ref{def:kkottwitz}.

\begin{proposition}
\label{prop:globalkottwitz}
Let $k \ge 0$ be an integer.
Suppose $\gamma_0 \in \S_{n,L/K}$ is
regular semisimple and normal, and let $\delta
= \gamma_0\sigma_{L/E}(\gamma_0)^{-1}$. 
Then there exists $\gamma \in \S_{n,L/K}$ with
$\gamma\sigma_{L/E}(\gamma)^{-1} = \delta$,
such that $\gamma$ is Kottwitz at all places $v$ of $F$ 
which are inert in $E$ and unramified in $K$, and
$k$-Kottwitz at $v_1, v_2, \dots, v_m$.
\end{proposition}

\begin{proof} 
Let $L'$ be the centralizer of $\delta$ in
$\Mat_n(L)$. 
Since $\delta$ is regular semisimple,
$L'$ is the commutative $L$-subalgebra of $\Mat_n(L)$ generated by
$\delta$. 
Note that $\sigma_{L/E}$ and $\sigma_{L/K}$
act on $L'$, since $\delta\sigma_{L/E}(\delta)
= \delta\sigma_{L/K}(\delta) = 1$.
Let $E'$ and $K'$ denote the $\sigma_{L/E}$ and
$\sigma_{L/K}$-fixed elements of
$L'$, respectively, and let $F' = E' \cap K'$. 

Let $\Sigma_{\inert}$ be the set of places $v \neq
v_1, \dots, v_m$ of $F$ which are inert in $E$ and unramified in $K$, such that
\[
    g^{-1}\delta g \in \GL_n(\O_{L_v})
\]
for some $g \in \GL_n(L_v)$.
For $v \in \Sigma_{\inert}$, the characteristic 
polynomial of $\delta$ has coefficients in $\O_{L_v}$,
so $\delta \in \O_{L_v'}$. Note if $v \not \in \Sigma_{\inert}$ is inert in $E$ and unramified in $K$ and $v \not = v_1, \hdots, v_m$, then $\gamma$ is automatically Kottwitz at $v$.

By Lemmas \ref{lem:kottwitzexist} and
\ref{lem:kkottwitzexist}, for each 
$v \in \Sigma_{\inert}$, there exists Kottwitz
$\gamma_v \in L_v'$ with 
$\gamma_v \sigma_{L'/K'}(\gamma_v) = 1$ and
$\gamma_v\sigma_{L'/E'}(\gamma_v)^{-1} = \delta$.
Since $\delta \in \O_{L_v'}$ as noted above,
by definition of $\gamma_v$ being Kottwitz, 
$\gamma_v \in \O_{L_v'}$.
In addition, by Lemma \ref{lem:kkottwitzexist},
for each $1 \le i \le m$,
there exists $\gamma_{v_i} \in L_{v_i}'$ with 
$\gamma_{v_i}\sigma_{L'/K'}(\gamma_{v_i}) = 1$ and
$\gamma_{v_i}\sigma_{L'/E'}(\gamma_{v_i})^{-1} = \delta$,
such that $\gamma_{v_i}$ is both Kottwitz
and $k$-Kottwitz.

By Lemma \ref{lem:kottwitzapprox}, there exists an
integer $m_v \ge 0$ for each $v \in \Sigma_{\inert}$
such that if $\gamma \in L_v'$ satisfies
\[
    \gamma - \gamma_v \in \varpi_{F_v}^{m_v}
    \O_{L_v'},
\]
then $\gamma$ is Kottwitz.
Also, there exist integers $m_{v_1}, \dots, m_{v_i}
\ge 0$ such that if $\gamma \in L_{v_i}'$ such that
\[
    \gamma - \gamma_{v_i} \in \varpi_{F_{v_i}}^{m_{v_i}}
    \O_{L_{v_i}'},
\]
then $\gamma$ is both Kottwitz and $k$-Kottwitz.
Furthermore, from the proof of Lemma \ref{lem:kottwitzapprox}, we see that we can take 
$m_v = 0$ for all but finitely many $v$.

Let $\Sigma = \Sigma_{\inert} \cup
\{v_1, \dots, v_m\}$.
Let $\Sigma_{\text{f}}$ be a finite subset of
$\Sigma$ such that for all $v \in \Sigma - \Sigma_{\text{f}}$,
we have $v \in \Sigma_{\inert}$,
$\gamma_0 \in \O_{L_v'}^\times$, and $m_v = 0$.
Note that for all $v \in \Sigma $, we have 
$\gamma_v\gamma_0^{-1} \in E_v'$ and 
$\gamma_v\gamma_0^{-1}\sigma_{E'/F'}
(\gamma_v\gamma_0^{-1}) = 1$, 
so there exists $a_v \in \O_{E_v'}$ such that
\[
    \gamma_v = a_v^{-1} \gamma_0 \sigma_{E'/F'}(a_v).
\]  
Furthermore, for $v \in \Sigma - \Sigma_{\text{f}}$,
we have $\gamma_v\gamma_0^{-1} \in \O_{E_v'}^\times$
and the extension $E_v'/F_v'$ is unramified, 
so we can assume $a_v \in \O_{E_v'}^\times$.

For each $v \in \Sigma_{\text{f}}$, let $\varepsilon_v > 0$
be an integer such that
\begin{equation}
\label{eq:orda}
    \varpi_{F_v}^{\varepsilon_v} a_v^{-1}
    \in \varpi_{F_v}\O_{E_v'}
\end{equation}
and
\begin{equation}
\label{eq:ordgamma}
   \varpi_{F_v}^{\varepsilon_v} a_v^{-2}\gamma_0
   \in \varpi_{F_v}^{m_v}\O_{L_v'}.
\end{equation}
Since $\Sigma - \Sigma_{\text{f}} \subseteq \Sigma_{\inert}$,
the set of places of $E'$ lying above an element of 
$\Sigma - \Sigma_{\text{f}}$ has Dirichlet density $0$, so
by Lemma \ref{lem:aapprox}, there exists
$a \in E'$ such that 
\begin{equation}
\label{eq:aminusav}
    a - a_v \in \varpi_{F_v}^{\varepsilon_v} \O_{E_v'}
\end{equation}
for all $v \in \Sigma_{\text{f}}$, and $a \in \O_{E_v'}^\times$
for $v \in \Sigma - \Sigma_{\text{f}}$.

Let
\[
    \gamma = a^{-1}\gamma_0\sigma_{E'/F'}(a).
\]
Then for $v \in \Sigma_{\text{f}}$, we have
$aa_v^{-1} \in \O_{E_v'}^\times$ by \eqref{eq:orda}
and \eqref{eq:aminusav}. Also recall that 
$a_v \in \O_{E_v'}$, so by \eqref{eq:aminusav}
and \eqref{eq:ordgamma}, we have
\begin{align*}
    \gamma - \gamma_v &= \gamma_0(a^{-1}\sigma_{E'/F'}(a) 
    - a_v^{-1}\sigma_{E'/F'}(a_v)) \\
    &= a^{-1}a_v^{-1}\gamma_0(a_v\sigma_{E'/F'}(a)
    - a\sigma_{E'/F'}(a_v)) \in \varpi_{F_v}^{\varepsilon_v} 
    a_v^{-2}\gamma_0 \O_{E_v'} \subseteq
    \varpi_{F_v}^{m_v}\O_{L_v'}.
\end{align*}
In addition, for $v \in \Sigma - \Sigma_{\text{f}}$, 
we have
\[
    \gamma - \gamma_v = \gamma_0(a^{-1}\sigma_{E'/F'}(a) 
    - a_v^{-1}\sigma_{E'/F'}(a_v)) \in \varpi_{F_v}^{m_v}\O_{L_v'}
\]
since $m_v = 0$, $\gamma_0 \in \O_{L_v'}^\times$,
and $a, a_v \in \O_{E_v'}^\times$.
Thus, by the construction of the $m_v$, this means
that $\gamma$ is Kottwitz at all $v \in \Sigma_{\inert}$ and
both Kottwitz and $k$-Kottwitz at $v = v_1, \dots, v_m$.
\end{proof}

We now prove some lemmas used in the above proof.

\begin{lem}
\label{lem:kottwitzapprox}
Let $v$ be a finite place of $L$, and let $k \ge 0$ be
an integer. Let $L_v'$ be an \'etale $L_v$-algebra
of dimension $n$.
Suppose $\delta \in L_v'$ generates $L_v'$ as an
$L_v$-algebra. 
There exists an integer $m = m(\delta, k)$ such that if 
\[
    \xi = c_{n-1}\delta^{n-1} + \cdots + c_0
\]
for some $c_0 \in 1 + \varpi_{L_v}^k\O_{L_v}$
and $c_1, \dots, c_{n-1} \in \varpi_{L_v}^k\O_{L_v}$,
and $\xi' \in L_v'$ satisfies 
$\xi' - \xi \in \varpi_{L_v}^m \O_{L_v'}$, then
\[
    \xi' = c_{n-1}'\delta^{n-1} + \cdots + c_0
\]
for some $c_0' \in 1 + \varpi_{L_v}^k\O_{L_v}$
and $c_1', \dots, c_{n-1}' \in \varpi_{L_v}^k\O_{L_v}$ 
as well.

\end{lem}
\begin{proof}
Let $p(X) \in L_v[X]$ be the minimal polynomial of 
$\delta$, so $L_v' = L_v[X]/(p(X))$. 
Let $\sigma_1, \dots, \sigma_n$ be the $n$
distinct maps of $L_v$-algebras $L_v' \to \ol{L_v}$,
where $\ol{L_v}$ denotes an algebraic closure of $L_v$.

Let $m = m(\delta)$ be an integer such that
\begin{equation}
\label{eq:m} 
    \varpi_{L_v}^m \cdot \begin{pmatrix}
        \sigma_1(\delta)^{n-1} & \cdots & 1 \\
        \vdots & \ddots & \vdots \\
        \sigma_n(\delta)^{n-1} & \cdots & 1
    \end{pmatrix}^{-1} 
    \in \varpi_{L_v}^k \Mat_n(\O_{\ol{L_v}}).
\end{equation} 
Note that the above matrix is invertible because 
$p(X)$ has $n$ distinct roots.

Now suppose $\xi' - \xi \in \varpi_{L_v}^m\O_{L_v'}$.
This implies 
\begin{equation}
\label{eq:sigdiff}
    \sigma_i(\xi') - \sigma_i(\xi) \in \varpi_{L_v}^m
    \O_{\ol{L_v}}
\end{equation}
for $1 \le i \le n$.
Since $\delta$ generates $L_v'$, we know that
\[
    \xi' = c_{n-1}'\delta^{n-1} + \cdots + c_0'
\]
for some $c_0',\dots, c_{n-1}' \in L_v$.
We have
\[
    \sigma_i(\xi') = c_{n-1}'\sigma_i(\delta)^{n-1} + \cdots + c_0'
\]
for $1 \le i \le n$, so
\begin{align*}
    \begin{pmatrix}
        c_{n-1}' \\ \vdots \\ c_0' 
    \end{pmatrix} &= 
    \begin{pmatrix}
        \sigma_1(\delta)^{n-1} & \cdots & 1 \\
        \vdots & \ddots & \vdots \\
        \sigma_n(\delta)^{n-1} & \cdots & 1
    \end{pmatrix}^{-1} 
    \begin{pmatrix}
        \sigma_1(\xi') \\ \vdots \\ \sigma_n(\xi')
    \end{pmatrix} \\
    &= \begin{pmatrix}
        \sigma_1(\delta)^{n-1} & \cdots & 1 \\
        \vdots & \ddots & \vdots \\
        \sigma_n(\delta)^{n-1} & \cdots & 1
    \end{pmatrix}^{-1} 
    \begin{pmatrix}
        \sigma_1(\xi') - \sigma_1(\xi) \\ \vdots \\ \sigma_n(\xi') - \sigma_n(\xi)
    \end{pmatrix} +
    \begin{pmatrix}
        c_{n-1} \\ \vdots \\ c_0,
    \end{pmatrix}
\end{align*}
where the second equality follows from the analogous
formula for $\transp{(c_{n-1}, \dots, c_0)}$.
Then the assumptions that $c_0 \in 1 + \varpi_{L_v}^k
\O_{L_v}$ and $c_1, \dots, c_{n-1} \in 
\varpi_{L_v}^k\O_{L_v}$, combined with \eqref{eq:m} and
\eqref{eq:sigdiff}, imply that 
$c_0' \in (1 + \varpi_{L_v}^k\O_{\ol{L_v}}) \cap L_v
= 1 + \varpi_{L_v}^k\O_{L_v}$ and $c_1', \dots, c_{n-1}'
\in (\varpi_{L_v}^k\O_{\ol{L_v}}) \cap L_v 
= \varpi_{L_v}^k\O_{L_v}$ as well.
\end{proof}

\begin{lem}
\label{lem:aapprox}
Let $E'$ be a number field. Let $\Sigma_0$
be a set of finite primes of $E'$ with Dirichlet density
$0$, and let $\Sigma_{\textnormal{f}}$ be finite set of finite primes
which is disjoint from $\Sigma_0$.
For each $v \in \Sigma_{\textnormal f}$, let $\varepsilon_v$ 
be an integer, and let $a_v$ be an element of 
$E_v'^\times$.
Then there exists $a \in E'^\times$
such that $a \in \O_{E'_v}^\times$ for all $v \in \Sigma_0$, and
\[
    a - a_v \in \varpi_{E'_v}^{\varepsilon_v} \O_{E'_v}
\]
for all $v \in \Sigma_{\textnormal{f}}$.
\end{lem}
\begin{proof}
First, we note that there exists $c \in E'^\times$ 
such that $c^{-1} a_v\in \O_{E'_v}^\times$ for all 
$v \in \Sigma_{\text{f}}$ and $c \in \O_{E'_v}^\times$ 
for all $v \in \Sigma_0$. Indeed, consider the 
fractional ideal
\[
    \mathfrak n = \prod_{v \in \Sigma_{\textnormal{f}}} \mathfrak p_v^{\ord_v(a_v)},
\]
where $\fp_v$ denotes the prime ideal of $\O_{E'_v}$
corresponding to $v$. Since $\Sigma_0 \cup \Sigma_{\text{f}}$ 
has Dirichlet density $0$, there exists a prime ideal
$\mathfrak q$ in the same ideal class as $\mathfrak n$.
Let $c$ be such that $c \mathfrak q = \mathfrak n$.
Then $c$ satisfies the required property.

Since $\Sigma_{\text{f}}$ is finite, by weak
approximation, there exists $a' \in E'^\times$ such that
\[  
    a' - c^{-1} a_v \in \varpi_{E'_v}^{\max\{1,
    \varepsilon_v - \ord_v(a_v)\}} \cdot \O_{E'_v}
\]
for all $v \in \Sigma_{\text{f}}$.
Since $c^{-1} a_v \in \O_{E'_v}^\times$ for
$v \in \Sigma_{\text{f}}$,
clearly this implies $a' \in \O_{E'_v}^\times$
for $v \in \Sigma_{\text{f}}$. 

Let $\fm$ be the ideal
\[
    \fm = \prod_{v \in \Sigma_{\text{f}}} 
    \fp_v^{\max\{0,
    \varepsilon_v - \ord_v(a_v)\}}.
\]
As noted above, the fractional ideal
$(a')$ is coprime to $\fm$. Thus, the set of prime
ideals of $\O_{E'}$ in the same ray class
modulo $\fm$ as $(a')$ has positive Dirichlet density,
so there exists $\fp$ in this ray class
which is not an element of $\Sigma_0$.

Let $b \in E'_{\fm, 1}$ be such that $(ba') = \fp$.
This means $b$ is an element of $E'$ such that
\[
    b - 1 \in \varpi_{E'_v}^{\max\{0,
    \varepsilon_v - \ord_v(a_v)\}} \cdot \O_{E'_v}
\]
for all $v \in \Sigma_{\text{f}}$. 
We claim that $a = cba'$
satisfies the requirements of the lemma.

Indeed, for all $v \in \Sigma_0$,
we have $a \in \O_{E'_v}^\times$ since 
$(c^{-1} a) = \fp$ is a prime ideal not contained in $\Sigma_0$, and $c \in \O_{E'_v}^\times$.
In addition, for $v \in \Sigma_{\text{f}}$, we have
\begin{align*}
    a - a_v &= ca'(b - 1) + c(a' - c^{-1}a_v) \\
    &= ca_v^{-1}(a'a_v(b - 1) + a_v(a' - c^{-1}a_v))
    \in
    \varpi_{E'_v}^{\varepsilon_v} \O_{E'_v},
\end{align*}
since $a_v(b - 1), a_v(a' - c^{-1}a_v) \in 
\varpi_{E'_v}^{\varepsilon_v} \O_{E'_v}$ and 
$a', ca_v^{-1} \in \O_{E'_v}^\times$.
\end{proof}

\subsection{Spectral expansions}

For any reductive group $G$ over $F$, let $R$ denote the right translation action of
$G(\AA)$ on $L^2(G(F)\bs G(\AA))$.

For a place $v$ of $F$ and $f_{v} \in C_c^\infty(G(F_v))$, we say that $f_v$ is  cuspidal if for all parabolic $P=MN$ of $G_{F_v}$, we have
\[
    \int_{N(F_v)} f_v(nx) dn = 0, \quad \forall x \in G(F_v).
\]
Note that any truncated matrix coefficient of a supercuspidal representation is cuspidal. It is well known that  If $f = f_{v_1}f^{v_1} \in C_c^\infty(G(\AA))$ for, such that $f_{v_1}$ is cuspidal on $G(F_{v_1})$, then the image of $R(f)$ lies in $L^2_{\mathrm{cusp}}([G])$. 

Recall that we have the decomposition of $L^2([G])$:
\[
    L^2([G]) = L^2_{\operatorname{cusp}}([G]) \oplus L^2_{\operatorname{cusp}}([G])^\perp,
\]
and
\[
    L^2_{\operatorname{cusp}}([G]) = \bigoplus_{\Pi} L^2_\Pi([G]),
\]
where $\Pi$ runs through cuspidal automorphic representations with central character trivial on $A_{G}^\infty$. Let $\operatorname{proj}_{\Pi}$ denote the projection to the $\Pi$-part of $L^2([G])$. Recall that the $\Pi$-part of the kernel function $K_{f,\Pi}(x,y)$ satisfies the property that for any $\varphi \in L^2([G])$, we have
\[
    \operatorname{proj}_\Pi(R(f)\varphi) = \int_{[G]} K_{f,\Pi}(x,y) \varphi(y) dy.
\]

\begin{definition}
    For any irreducible cuspidal automorphic representation 
    $\pi$ of $G^V(\AA)$ and test function $(f, \phi_1 \otimes \phi_2)$,
    define
    \[
        J^V_\pi(f, \phi_1 \otimes \phi_2) 
        = \sum_{\varphi} \P(R(f)\varphi, \phi_1)\ol{\P(\varphi, \phi_2)},
    \]
    where the sum runs over an orthonormal basis of $\pi$.
\end{definition}

\begin{definition}
\label{def:globalzetaint}
Let $\Pi$ be an irreducible cuspidal automorphic
representation of $G'(\AA)$. 
For $\varphi \in \Pi$ and
$\phi' \in \cS(\AA_{E,n})$, define
\begin{equation}
        \lambda(s, \varphi, \mu, \phi') = \int\limits_{H_1(F)\bs H_1(\AA)} \varphi(g)\ol{\Theta_{\mu}(s, g, \phi')}\, dg
\end{equation}
Here $\Theta_{\mu}(s, g, \phi')$ is the theta series in Definition \ref{defn theta series: GL side}. Define
\begin{equation}
\beta(\varphi) = \int\limits_{H_2(F)A_{G'}^\infty \bs H_2(\AA)}
    \varphi(h)\eta_{L/K}(\det h)^{n+1}\, dh.
\end{equation}
\end{definition}

\begin{definition}
For any irreducible cuspidal automorphic representation $\Pi$ of $G'(\AA)$ and test function $(f', \phi')$, define
\[
    I_\Pi(s, f', \phi') = \sum_\varphi \lambda(s, R(f')\varphi, \mu, \phi')
    \ol{\beta(\varphi)},
\]
   where the sum runs over an orthonormal basis of $\Pi$. 
\end{definition}

\begin{proposition}
\label{prop:uspectral}
     If the test function $(f, \phi_1 \otimes \phi_2)$ is good (Definition \ref{def:ugood}), then 
    \[
        J^V(f , \phi_1 \otimes \phi_2)
        = \sum_\pi J^V_\pi(f, \phi_1 \otimes \phi_2),
    \]
    where the sum runs over all irreducible cupsidal automorphic 
    representations of $G^V(\AA)$.
    The right hand side is absolutely convergent.
    \end{proposition}
\begin{proof}
Recall that
\begin{align*}
    J^V(f, \phi_1 \otimes \phi_2)
    = \int\limits_{H(F)\bs H(\AA)}
    \int\limits_{H(F)\bs H(\AA)}
    K_{f}(g, h)
    \ol{\Theta(g, \phi_1)}
    \Theta(h, \phi_2)\,dg\,dh.
\end{align*}
Since $f_{v_1}$ is a truncated matrix coefficient of a supercuspidal representation, the image of $R(f)$ lies in $L^2_{\operatorname{cusp}}([G^V])$.
Therefore,
\[
    K_f(x,y) = \sum_{\pi} K_{f,\pi}(x,y),
\]
where in the sum, $\pi$ runs through all cuspidal automorphic representations of $G^V(\AA)$.
We have
\[
    K_{f, \pi}(g, h) =
    \sum_{\varphi}(R(f)\varphi)(g)
    \ol{\varphi(h)},
\]
where $\varphi$ runs over an orthonormal basis of $\pi$.
The proposition then follows.
The absolute convergence of the sum follows from
\cite[Proposition~A.1.2]{MR4332778}.
\end{proof}

\begin{proposition}
\label{prop:glspectral}
If the test function $(f', \phi')$ is good (Definition \ref{def:glgood}), then
\[
    I(f', \phi') = \sum_{\Pi} I_\Pi(f', \phi')
\]
where the sum runs over all irreducible cuspidal 
automorphic representations of $G'(\AA)$ with
central character trivial on $A_{G'}^\infty$.
The right hand side is absolutely convergent.
\end{proposition}
\begin{proof}
Let
\[
    K_{\Pi}^0(x,y) = \int_{A_{G'}^\infty} K_{f,\Pi}(x,ay) da.
\]
Then we have
\[
    K_{\Pi}^0(x,y) = \sum_{\varphi} R(f)\varphi(x) 
    \overline{\varphi(y)},
\]
where the sum runs through the orthonormal basis of $\Pi$. The remaining are similar to the proof of Proposition \ref{prop:uspectral}.
\end{proof}

\subsection{Relative trace identity}

We first state a lemma (an exercise in real analysis).
\begin{lem} \label{lem:l^1_linearly_independent}
    Let $X$ be a second countable compact Hausdorff space. Given a sequence of distinct points $x_i$ of $X$ and a sequence $c_i$ of complex numbers such that
    \begin{equation*}
        \sum_{i=1}^{\infty} \lvert c_i \rvert < \infty.
    \end{equation*}
    Suppose for any continuous function $f$ on $X$, we have
    \begin{equation*}
        \sum_{i=1}^{\infty} f(x_i)c_i = 0.
    \end{equation*}
    Then $c_i=0$ for any $i$. 
\end{lem}
The lemma formally says that the valuation functionals $f \mapsto f(x)$ on $C(X)$ as $x$ varies are “linear independent with $\ell^1$-coefficients”, where $C(X)$ is the Banach space of continuous functions on $X$.

\begin{proof}
    Without loss of generality, we show that $c_1=0$. For any $\varepsilon>0$, choose $N$ such that
    \begin{equation*}
        \sum_{i>N} \lvert c_i \rvert < \varepsilon.
    \end{equation*}
    Then pick a neighborhood $U$ of $x_1$ which only contains $x_i$ for $i>N$ and a function $f$ such that $0 \le f \le 1$, $f(x)=1$ and $f=0$ on $X \setminus U$, then we see that $\lvert c_1 \rvert < \varepsilon$. Since $\varepsilon$ is arbitrary, we conclude that $c_1 = 0$.

    A more conceptual proof: it is easy to check that for a Borel set $Y$, $Y \mapsto \sum_{x_i \in Y} c_i$ defines a Radon measure $\mu$ on $X$, the condition says $\int_X f d \mu = 0$ for all continuous function of $X$. Therefore $\mu=0$. In particular, $\mu \left( \{ x_i\} \right) = c_i = 0$.
\end{proof}

\begin{definition}[Global smooth transfer]
\label{defn: global smooth transfer}
Choose normal representatives $S$ and $R^V$ (Definition \ref{normal reps: unitary}, \ref{normal reps: GL side}). We say that global test functions $(f', \phi') = ( \Delta_{H_1}^{*} \Delta_{H_2}^* \otimes_v f'_v , \otimes_v \phi'_v)$ and
$\{(f^V, \phi_1^V \otimes \phi_2^V) = (\Delta_H^{*,2} \otimes_v f^V_v, \otimes_v (\phi_{1,v}^V \otimes \phi_{2,v}^V) ) \}
_{V \in [\SHerm_n^\times(F)]}$ match, if
\begin{enumerate}
    \item $(f'_v,\phi'_v)$ and $\{ (f^V_v,\phi^V_{1,v} \otimes \phi_{2,v}^V) \}$ are smooth transfer of each
other at each place $v$ of $F$, see Definition \ref{defn: smooth transfer}.
    \item for each place $v$ of $F$, the test function $(f_v^V, \phi_{1,v}^V
\otimes \phi_{2,v}^V)$ depends only on the isomorphism class of $V$ in $[\SHerm_n^\times(F_v)]$.
\end{enumerate}
\end{definition}

See Propositions \ref{prop:FL} (fundamental lemma), \ref{prop:splittransfer} ($E$-split transfer), and
\ref{prop:Ksplittransfer} ($K$-split transfer) for the existence of transfers.

\begin{theorem}
\label{prop:identity}
Let $\Pi$ be the base change of $\pi$.
For good matching test functions $(f',\phi')$
and $\{(f^V,\phi_1^V \otimes \phi_2^V) \}_{V\in [\SHerm_n^\times(F)]}$, we have
    \[
        I_\Pi(f', \phi') =
        \sum_{V \in [\SHerm_n^\times(F)]} \sum_{\pi'} J_{\pi'}^V(f^V, \phi_1^V \otimes \phi_2^V),
    \]
where $\pi'$ runs over all irreducible cuspidal
automorphic representations of $\U(V)(\AA_K)$
whose base change to $\GL_n(\AA_L)$ is also
$\Pi$.
\end{theorem}   
\begin{proof}

From the geometric decomposition results in Proposition 
\ref{prop:gldecomp} and \ref{prop:udecomp},
the definition of matching
\ref{defn: smooth transfer}, and Propositions
\ref{prop:uspectral} and \ref{prop:glspectral},
we have 
\[
    \sum_{\Pi} I_\Pi(f', \phi') =
    \sum_{V \in [\SHerm_n^\times(F)]} \sum_{\pi} 
    J^V_\pi(f^V, \phi_1^V \otimes \phi_2^V)
\]
where the first sum runs over cuspidal automorphic 
representations of $G'(\AA)$ with central
character trivial on $A_{G'}^{\infty}$ and the second sum runs over
cuspidal automorphic representations of $G^V(\AA)$ 
with central character trivial on $A_{G^V}^{\infty}$.

Let $\Sigma$ be the set of nonarchimedean places $v$
of $F$ which are split in $E$ and not equal to
$v_1, v_2$, such that
\begin{gather*}
    (f', \phi') = (\mathbf{1}_{\GL_n(\O_{L_v})}, \mathbf{1}_{\O_{E_v,n}}),\\ 
    (f^V, \phi_1^V \otimes \phi_2^V) =
    (\mathbf{1}_{\U(V)(\O_{K_v})}, \mathbf{1}_{\LL(\O_{F_v})} \otimes
    \mathbf{1}_{\LL(\O_{F_v})})
\end{gather*}
for all $V \in [\SHerm_n^\times(F)]$.

Note that if $\pi_v$ is not unramified for some
$v \in \Sigma$, then both sides of the desired identity
are zero. Thus, we assume that $\pi_v$ is unramified
for all $v \in \Sigma$.
Let $\Lambda_{\Pi_{\Sigma}}$ denote the character
of the split spherical Hecke algebra
\[
    \bigotimes_{v \in \Sigma}'
    \cH(\GL_n(L_v), \GL_n(\O_{L_v}))
\]
for $\Pi_\Sigma = \otimes_{v \in \Sigma}' \Pi_v$,
and similarly let $\lambda_{\pi_\Sigma}$ denote
the character of the split spherical Hecke algebra
\[
    \bigotimes_{v \in \Sigma}'
    \cH(\U(V)(K_v), \U(V)(\O_{K_v}))
\]
for $\pi_\Sigma = \otimes_{v \in \Sigma}' \pi_v$.

From Proposition \ref{prop:splittransfer}, we have a
transfer map
\[
    b\colon \bigotimes_{v \in \Sigma}'
    \cH(\GL_n(L_v), \GL_n(\O_{L_v})) \to
    \bigotimes_{v \in \Sigma}'
    \cH(\U(V)(K_v), \U(V)(\O_{K_v}))
\]
given by $b(\alpha) = \alpha_1 \ast \alpha_2^\vee$
where $\alpha = \alpha_1 \otimes \alpha_2$.
Then we have
\[
   \lambda_{\pi_\Sigma}(b(\alpha))
   = \Lambda_{\operatorname{BC}(\pi)_\Sigma}(\alpha) 
\]
for all $\alpha \in \bigotimes_{v \in \Sigma}'
\cH(\GL_n(L_v), \GL_n(\O_{L_v}))$.
The point is that $(f' \ast \alpha, \phi')$ and
$\{(f^V \ast b(\alpha), \phi_1^V
\otimes \phi_2^V)\}_{V \in [\SHerm_n^\times(F)]}$
are also good matching test functions, so
\[
    \sum_\Pi I_\Pi(f', \phi') \Lambda_{\Pi_\Sigma}(\alpha)
    = \sum_{V \in [\SHerm_n^\times(F)]} \sum_\pi
    J^V_\pi(f^V, \phi_1^V \otimes \phi_2^V)
    \Lambda_{\operatorname{BC}(\pi)_\Sigma}(\alpha).
\]

By \cite[Theorem~A]{ramakrishnan2018theorem},
if two cuspidal automorphic representations of
$\GL_n(\AA_L)$ are isomorphic at almost all primes of
$L$ which are of degree $1$ over $F$, then they are
isomorphic.
This, combined with Lemma \ref{lem:l^1_linearly_independent} yields the proof.
\end{proof}

\subsection{Local distributions}
Let $\Pi$ be a cuspidal automorphic representation
of $G'(\AA)=\GL_n(\mathbb A_L)$.
Let $\psi_L$ be the additive character of $\AA_L/L$ 
given by
\[
    \psi_L(x) = \psi_E(j_K(x - \sigma_{L/E}(x))).
\]
Let $\W(\Pi,\psi_L)$ be the Whittaker model of $\Pi$.

Let $v$ be a place of $F$. We also have a local
Whittaker model $\W(\Pi_v, \psi_{L_v})$.
Let $B_n$ be the standard upper triangular Borel of
$\GL_n$, and let $N_n$ be its unipotent radical.

\begin{definition}
For $W_{1,v}, W_{2,v} \in \W(\Pi_v, \psi_{L_v})$, let
\[
    \langle W_{1,v}, W_{2,v} \rangle_v
    =
    \int\limits_{N_{n-1}(L_v) \bs \GL_{n-1}(L_v)}
    W_{1,v}\left(\begin{pmatrix}
        g & \\ & 1
    \end{pmatrix}\right)
    \ol{W_{2,v}\left(\begin{pmatrix}
        g & \\ & 1
    \end{pmatrix}\right)}\,dg.
\]
\end{definition}

\begin{definition}(cf. Definition \ref{def:globalzetaint}) 
For $W_v \in \W(\Pi_v, \psi_{L_v})$ and $\phi_v' \in \cS(E_{v,n})$, let
\begin{equation}
\lambda_v(s, W_v, \mu_v, \phi_v')
    = \int\limits_{N_n(E_v)\bs \GL_n(E_v)} W_v(g)
    \ol{\mu_v(\det g) \phi_v'(e_n g)} \lvert \det g \rvert^s\, dg.   
\end{equation}
where $e_n = (0,\dots, 0, 1) \in E_n$.

For $W_v \in \W(\Pi_v, \psi_{L_v})$, define 
\begin{equation}
    \beta_v(W_v) = \int\limits_{N_{n-1}(K_v)\bs 
    \GL_{n-1}(K_v)}
    W_v\left(\begin{pmatrix}
    \epsilon_{n-1} h & \\ & 1 \end{pmatrix}\right)
    \eta_{L_v/K_v}(\det h)^{n+1}\, dh,
\end{equation}
where $\epsilon_{n-1}$ is the diagonal matrix
\[
    \epsilon_{n-1} = \begin{pmatrix}
        j_M^{n-1} & & \\ & \ddots & \\ & & j_M
    \end{pmatrix} 
    \in \GL_{n-1}(L_v).
\]
\end{definition}
Note that the matrix $\epsilon_{n-1}$ guarantees the integrand function is trivial on $N_{n-1}(K_v)$.

Let $\pi$ be a tempered cuspidal representation of $G^V(\mathbb A)$. Note that by the local twisted GGP conjecture \cite{le2025local}, for any place $v$, the vector space $\Hom_{\U(V)(F_v)}(\pi_v \otimes \ol{\omega_v}, \BC)$ is at most one dimensional. Moreover, if the period $\mathcal{P}$ is non-trivial on $\pi$, then it is one-dimensional for any place $v$. If this is the case, which we will assume from now on, 
choose a generator $\ell_v$ of $\Hom_{\U(V)(F_v)}(\pi_v \otimes \ol{\omega_v}, \BC)$.

\begin{definition}
    For $f_v \in C_c^\infty(G^V(F_v))$ and $\phi_{1,v} \otimes \phi_{2,v} \in \cS(\LL(F_v))^{\otimes 2}$, let
    \[
        J_{\pi_v}(f_v, \phi_{1,v} \otimes \phi_{2,v})
        = \sum_{\varphi_v} \frac{\ell_v(\pi_v(f_v)\varphi_v, \phi_{1,v})\ol{\ell_v(\varphi_v, \phi_{2,v})}}
        {\langle \varphi_v, \varphi_v \rangle_v},
    \]
    where $\varphi_v$ runs over an orthogonal basis of $\pi_v$.
\end{definition}

\begin{definition}
    For $f_v' \in C_c^\infty(G'(F_v))$ and $\phi_v' \in \cS(E_{v,n})$, let
    \[
        I_{\Pi_v}(s, f_v', \phi_v')
        = \sum_{W_v} 
        \frac{\lambda_v(s, \Pi_v(f_v')W_v, \mu_v, \phi_v')\ol{\beta_v(W_v)}}
        {\langle W_v, W_v \rangle_v},
    \]
    where $W_v$ runs over an orthogonal basis
    of $\W(\Pi_v, \psi_{L_v})$. 
\end{definition}

\begin{definition}
\label{def:postype}
We say that a test function on the unitary side is of 
\emph{positive type} if it is of the form
$(f_v, \phi_v \otimes \phi_v)$, where $f_v =
f_0 \ast f_0^\star$ for some
$f_0 \in C_c^\infty(G(F_v))$, 
and $\phi_v \in \cS(\LL(F_v))$. 
Here, $\ast$ denotes convolution, and
$f_0^\star(g) \coloneqq \ol{f_0(g^{-1})}$.
We can similarly define positive type global test
functions.
\end{definition}

Note that for any $f_1, f_2 \in C_c^\infty(G(F_v))$ and
$\phi_1, \phi_2 \in \cS(\LL(F_v))$, we have
\begin{align*}
    \sum_{\varphi_v} \ell_v(\pi_v(f_1)\varphi_v, \phi_1)
    \ol{\ell_v(\pi_v(f_2)\varphi_v, \phi_2)}
    &= \sum_{\varphi_v} \ell_v(\pi_v(f_1\ast f_2^\star) \varphi_v, \phi_1)
    \ol{\ell_v(\varphi_v, \phi_2)},
\end{align*}
where $\varphi_v$ runs over an orthonormal basis
of $\pi_v$.
Thus if $(f_v, \phi_v \otimes \phi_v)$ is a positive
type test function, then
\begin{equation}
\label{eq:postype}
    J_{\pi_v}(f_v, \phi_v \otimes \phi_v)
    = \sum_{\varphi_v} 
    \ell_v(\pi_v(f_0 \ast f_0^\star)\varphi_v,
    \phi_v)\ol{\ell_v(\varphi_v, \phi_v)}
    = \sum_{\varphi_v}
    \lvert \ell_v(\pi_v(f_0)\varphi_v, \phi_v)\rvert^2
    \ge 0.
\end{equation}

\begin{lem}
\label{lem:matrixcoeff}
    Let $v$ be a nonarchimedean place of $F$.
    Assume that $\ell_v \neq 0$. There exists
    a test function $(f_v, \phi_{1,v} \otimes \phi_{2,v})$
    such that 
    $J_{\pi_v}(f_v, \phi_{1,v} \otimes \phi_{2,v}) \neq 0$ 
    and $f_v$ is a truncated matrix coefficient of $\pi_v$.
\end{lem}
\begin{proof}
The proof is the same as that of 
\cite[Lemma~6.14]{TGGP-Wang}.
\end{proof}

\begin{lem}
\label{lem:rsssupp}
Let $v$ be a nonarchimedean place of $F$ which
is completely split in $L$ such that $\pi_v$
is supercuspidal. 
Assume that the distribution $J_{\pi_v}$ is
nonzero. 
Then there exists a test function $(f_v, \phi_v)$ 
with $\phi_v = \sum_{j=1}^r \phi_{1,v}^{(j)}
    \otimes \phi_{2,v}^{(j)}$
and an integer $k \ge 0$ such that 
\begin{enumerate}[(1)]
\item $J_{\pi_v}(f_v, \phi_v)$ is nonzero;
\item $f_v$ is supported on the regular
semisimple locus of $G(F_v)$ under the
$H(F_v) \times H(F_v)$-action;
\item if $[\zeta, z] \in G(F_v) \times V_v^\vee$ is 
such that $\zeta$ is $k$-Kottwitz as an
element of $\GL_n(L_v)$ (in the sense of Definition 
\ref{def:kkottwitz}), and 
\[
    f_v(g^{-1}\zeta h)\left(\sum_{j=1}^r
    \phi_{2,v}^{(j)} \otimes \ol{\omega_v(h^{-1} g)\phi_{1,v}^{(j)}}\right)^\ddagger(zh) \neq 0
\] 
for some $g, h \in H(F_v)$, then $[\zeta, z] \in 
(G(F_v) \times V_v^\vee)_\rss$.
\end{enumerate}
\end{lem}
\begin{proof}
This is shown in the proof of
{\cite[Lemma~6.13]{TGGP-Wang}}, once we identify
$G(F_v) = \U(V)(K_v)$ with $\GL_n(F_v) \times \GL_n(F_v)$.
Under this identification,
$\zeta$ being $k$-Kottwitz in the sense of Definition
\ref{def:kkottwitz} implies that $\zeta$ is
$k$-Kottwitz with respect to $\beta$
in the sense of \cite[Definition~4.10(2)]{TGGP-Wang}.
\end{proof}

\begin{lem}
\label{lem:archgood}
    Let $v$ be a place of $F$. 
    Then there exists $W_v \in \W(\Pi_v, \psi_{L_v})$ and $\phi_v' \in \cS(E_{v,n})$    such that 
    $\lambda_v(\frac12, W_v, \mu_v, \phi_v') \neq 0$.
\end{lem}
\begin{proof}
The case when $\mu$ is trivial is proved in \cite[Lemma 3.3.3]{BP21} or \cite{Yadav2025}. The same proof applied to the case when $\mu$ is not trivial. 
\end{proof}

\begin{lem}
\label{lem:splitlocaldistrid} 
Let $v$ be a place of $F$ that is split in $E$. 
There is a nonzero 
constant $c$ such that 
\[
    I_{\Pi_v}(\tfrac12, f'_v, \phi_v')
    = c J_{\pi_v}(f_v, \phi_{1,v} \otimes \phi_{2,v})
\]
for all matching test functions $(f_v', \phi_v')$
and $(f_v, \phi_{1,v} \otimes \phi_{2,v})$.
\end{lem}
\begin{proof}
Since $v$ is split in $E$, we have $L_v = K_v \times K_v$.
Then $\pi_v$ is a representation of
$\GL_n(K_v)$, and $\Pi_v = \pi_v \otimes \pi_v^\vee$
is a representation of $\GL_n(K_v) \times \GL_n(K_v)$.
Identify $\LL$ with $F_n$ as in Section 
\ref{sec:splittransfer}.
By the definition of matching at a split place given
in Proposition \ref{prop:splittransfer}, we have
\[
    f_{1,v}' \ast f_{2,v}'^\vee = f_v , \quad
    \phi_v' = \phi_{1,v} \otimes \ol{\phi_{2,v}}
\]
where $f_v \in C_c^\infty(\GL_n(K_v))$,
$\phi_{1,v} \otimes \phi_{2,v} \in \cS(F_n)\otimes
\cS(F_n)$, $f_v' = f_{1,v}' \otimes f_{2,v}' 
\in C_c^\infty(\GL_n(K_v) \times \GL_n(K_v))$,
and $\phi_v' \in \cS(F_n \times F_n)$.

Let $\psi_{K_v}$ be the additive character of $K_v$
given by
\[
    \psi_{K_v}(x) = \psi_{F_v}(j_K(x - \sigma_{K/F}(x))).
\]
Note that $\psi_{L_v}(x, y) = \psi_{K_v}(x)\psi_{K_v}(y)$.
Suppose $j_M = (j_0, -j_0) \in K \times K$, where
$j_0 \in K$.
Let 
\[
    \psi_{L_v}'(x, y) = \psi_{L_v}(j_0x, -j_0y).
\]
Then $\psi_{L_v}'(x, y) = \psi_{K_v}'(x)
\ol{\psi_{K_v}'(y)}$, where
$\psi_{K_v}'(x) = \psi_{K_v}(j_0x)$.

For $W_v \in \W(\Pi_v, \psi_{L_v})$, let
\[
    W_v'(g) \coloneqq W_v\left(\begin{pmatrix}
    \epsilon_{n-1} & \\ & 1
    \end{pmatrix} g\right).
\]
Then $W_v' \in \W(\Pi_v, \psi_{L_v}')
= \W(\pi_v, \psi_{K_v}') \otimes 
\ol{\W(\pi_v, \psi_{K_v}')}$. 
Suppose $W_v' = W_{1,v}' \otimes \ol{W_{2,v}'}$
for $W_{1,v}', W_{2,v}' \in \W(\pi_v, \psi_{K_v}')$.
Then we have
\[
    \beta_v(W_v) = \langle W_{1,v}', W_{2,v}' \rangle_v
\]
and
\[
    \langle W_v, W_v \rangle_v =
    \langle W_{1,v}', W_{1,v}'\rangle_v
    \langle W_{2,v}', W_{2,v}'\rangle_v.
\]
By \cite[Theorem~B]{Sun2012multiplicity}, the space
\[
    \Hom_{\GL_n(F_v) \times \GL_n(F_v)}(\pi_v 
    \otimes \ol{\omega_v} 
    \otimes \ol{\pi_v}
    \otimes \omega_v, \CC)
\]
has dimension $1$.
Thus, there exists a nonzero constant $c$ such that
\[
    \lambda_v(\tfrac12, \Pi_v(f_v')W_v, \mu_v, \phi_{1,v}
    \otimes \ol{\phi_{2,v}})
    = c\cdot \ell_v(\pi_v(f_{1,v}')W_{1,v}', 
    \phi_{1,v})\ol{\ell_v(\pi_v(\ol{f_{2,v}'}) W_{2,v}',
    \phi_{2,v})}.
\]

Thus,
\begin{align*}
    I_{\Pi_v}(\tfrac12, f_v', \phi_v')
    &=
    c\sum_{W_{1,v}'} \ell_v(\pi_v(f_{1,v}')
    W_{1,v}', \phi_{1,v})
    \ol{\ell_v(\pi_v(\ol{f_{2,v}'})W_{1,v}', \phi_{2,v})} \\
    &= c\sum_{W_{1,v}'} \ell_v(\pi_v(f_{1,v}'
    \ast f_{2,v}'^\vee)
    W_{1,v}', \phi_{1,v})
    \ol{\ell_v(W_{1,v}', \phi_{2,v})}\\
    &= c\sum_{W_{1,v}'} \ell_v(\pi_v(f_{v})
    W_{1,v}', \phi_{1,v})
    \ol{\ell_v(W_{1,v}', \phi_{2,v})}\\
    &= cJ_{\pi_v}(f_v, \phi_{1,v} \otimes \phi_{2,v}),
\end{align*}
where $W_{1,v}'$ runs over an orthonormal basis of
$\W(\pi_v, \psi_{K_v}')$.
\end{proof}

\section{Proof of the main theorem}
\label{section: proof of the main theorem}

\subsection{Review of Asai $L$-functions for $L/E$}

We recall the definition of the twisted Asai 
$L$-function, following \cite[Section~3.2]{BP21}.
Let $\Pi$ be a cuspidal automorphic representation
of $\GL_n(\AA_L)$. Let $v$ be a place of $E$.

Firstly, assume that $L_v$ is a field. Let $W_{L_v}'$ and $W_{E_v}'$ denote
the Weil--Deligne group of $L_v$ and $E_v$, 
respectively.
The local Langlands correspondence associates to $\Pi_v$ 
an admissible representation $\phi_v \colon W_{L_v'} \to \GL(M)$, 
where $M$ is an $n$-dimensional complex vector space.
Choose $s \in W_{E_v}' - W_{L_v}'$.
We define a representation $\As(\phi_v) \colon
W_{E_v}' \to \GL(M \otimes M)$ by
\[
    \As(\phi_v)(w) = \phi_v(w) \otimes \phi_v(sws^{-1}), \; w \in W_{E_v}', 
\]
\[
    \As(\phi_v)(s)(m_1 \otimes m_2) = m_2 \otimes
    \phi_v(s^2)(m_1), \; m_1 \otimes m_2 \in M \otimes M.
\]
We define the local Asai $L$-factor 
\[
    L(s, \Pi_v, \As_{L_v/E_v}) \coloneqq
    L(s, \As(\phi_v))
\]
where $L(s, \As(\phi_v))$ is the local $L$-factor
associated to $\As(\phi_v)$ as in
\cite[Section~3]{MR546607} and \cite[Section~2.2]{MR2730575}.

Secondly, assume that $L_v = E_v \times E_v$ is split. Then
$\Pi_v = \Pi_{1,v} \boxtimes \Pi_{2,v}$ for some
irreducible representations $\Pi_{1,v}$ and $\Pi_{2,v}$
of $\GL_n(E_v)$.
Let
\[
    L(s, \Pi_v, \As_{L_v/E_v})
    \coloneqq L(s, \Pi_{1,v} \times \Pi_{2,v})
    = L(s, \phi_{1,v} \otimes \phi_{2,v}),
\]
where $\phi_{1,v}$ and $\phi_{2,v}$ are the admissible
representations of $W_{E_v}'$ associated to
$\Pi_{1,v}$ and $\Pi_{2,v}$ by the local Langlands
correspondence.

Let $\mu$ be a Hecke character of $\AA_E^\times/E^\times$.
We view $\mu_v$ as a character of $W_{E_v}'$ via local
class field theory. We also define the
twisted local Asai $L$-factor
\[
    L(s, \Pi_v, \As_{L_v/E_v} \times \mu_v)
    \coloneqq L(s, \As(\phi_v) \otimes \mu_v)
\]
if $L_v$ is a field, and
\[
    L(s, \Pi_v, \As_{L_v/E_v} \times \mu_v)
    \coloneqq L(s, \Pi_{1,v} \times \Pi_{2,v}
    \otimes \mu_v)
\]
if $L_v = E_v \times E_v$.

\begin{definition}
Define the global twisted Asai L-function $(s \in \mathbb C)$
\[
    L(s, \Pi, \As_{L/E} \times \mu) = \prod_{v}
    L(s, \Pi_v, \As_{L_v/E_v} \times \mu_v),
\]
where $v$ runs over all places $v$ of $E$.
This product is absolutely convergent for
$\Re(s)$ sufficiently large.
\end{definition}

Recall that we have also defined the global
zeta integral $\lambda(s, \varphi, \mu, \phi')$
in Definition \ref{def:globalzetaint}.

\begin{theorem}[{\cite[Proposition, page 305; Theorem, page 309]{Flicker}, \cite{BP21}}]
\label{thm:asai}
There exists $\varphi \in \Pi$ and
$\phi' \in \cS(\AA_{E,n})$ such that
$\lambda(\frac12, \varphi, \mu, \phi') \neq 0$
if and only if $L(\frac12, \Pi, \As_{L/E}\times
\mu^{-1}) \neq 0$.
\end{theorem}

\subsection{Proof of Main Theorem}

Recall

\mainthm*

\begin{proof}[Proof that (2) implies (1)]
Assume that the period integral $\P$ is not 
identically zero for some skew-Hermitian space $V$
and a cuspidal automorphic representation
$\pi$ of $\U(V)(\AA_K)$ such that $\operatorname{BC}
(\pi) = \Pi$.

By Theorem \ref{thm:asai}, in order to show
that $L(\frac{1}{2}, \Pi, \As_{L/E} \times \mu^{-1})$ is 
nonzero, it suffices to show there exists a test
function $(f', \phi')$ such that
\[ I_\Pi(f', \phi') \neq 0.\]
To show that such a test function exists,
we first find a test function 
$(f, \phi) \in C_c^\infty(G(\AA)) \times 
\cS(\LL(\AA))^{\otimes 2}$ which is good
in the sense of Definition \ref{def:ugood}
such that 
\[ J_\pi(f, \phi) \neq 0.\] 
Suppose $\P(\varphi, \wt\phi_0) \neq 0$ for some factorizable
$\varphi \in \pi$ and $\widetilde \phi_0 \in \cS(\LL(\AA))$.
By \eqref{eq:postype}, this
implies there exists a factorizable test 
function $(\widetilde f, \widetilde\phi)$ of 
positive type, with $\wt\phi = \wt\phi_0 \otimes
\wt\phi_0$,
such that $J_\pi(\widetilde f, \widetilde\phi) \neq 
0$.

We define a good test function 
\[
    (f, \phi) = (\otimes_v f_v, \otimes_v \phi_v)
\]
by modifying $(\wt f, \wt\phi)$ at certain
places as follows. In the process, we also
define a constant $k$ which will be used for choosing 
$k$-Kottwitz representatives.

Recall that $v_1$ and $v_2$ are two nonarchimedean
places of $F$ which are completely split in $L$,
such that $\pi_{v_1}$ and $\pi_{v_2}$ are
supercuspidal.

Let $\Sigma_{\inert}$ be the set of places $v$
of $F$ which are inert in $E$ and unramified
in $K$, such that
$\mu_v, \psi_v'$ are unramified,
$j \in \O_{E_v}^\times$, and $\Pi_v$ (and thus 
$\pi_v$) is unramified.

Let $\{v_3, \dots, v_m\}$ be the set of places of $v$
of $F$ such that $v \not\in \Sigma_\inert$
and $v$ is not split in $E$.
By assumptions \eqref{item:fieldsunr},
\eqref{item:charsunr}, and \eqref{item:piunr} of
the theorem, this set
is finite and all elements are split in $K$.

As $\Hom_{\U(V)(F_v)}(\pi_v \otimes \ol{\omega_{\psi_v',
\mu_v}}, \CC) \neq 0$ 
for all $v$, the local twisted GGP 
conjecture \cite{le2025local} tells us that $V$ is the
split skew-Hermitian space at all $v \in \Sigma_{\inert}$.
\begin{enumerate}[(1)]
\item \label{item:testfnunr}
For all $v \in \Sigma_{\inert}$, let
$(f_v, \phi_v)$ be the unramified
test function specified in Section \ref{sec:FL}. 
Since $\pi_v$ is unramified,
we have $J_{\pi_v}(f_v, \phi_v) \neq 0$.

\item 
Let $(f_{v_1}, \phi_{v_1})$
be a test function such that
$J_{\pi_{v_1}}(f_{v_1}, \phi_{v_1}) \neq 0$ and 
$f_{v_1}$ is a truncated matrix coefficient of a 
supercuspidal representation. 
Such a test function exists
by Lemma \ref{lem:matrixcoeff}, 
since $\Pi_{v_1}$ is supercuspidal, so
$\pi_{v_1}$ is also supercuspidal.

\item \label{item:kv2} Let $(f_{v_2}, \phi_{v_2})$ be a test function
and let $k \ge 0$ be an integer such that 
$J_{\pi_{v_2}}(f_{v_2}, \phi_{v_2}) \neq 0$, and
the regular semisimple support condition of Definition
\ref{def:ugood} \eqref{ursssupp}
is satisfied for all $k$-Kottwitz 
sets of representatives $R^V$.
These exist by Lemma \ref{lem:rsssupp}.

By Proposition \ref{prop:Ksplittransfer},
we can also assume $k$ is large enough so that
for $3 \le i \le m$, there exists a test function
$(f_{v_i}', \phi_{v_i}')$ on the general linear
side, which is a smooth
transfer of $(\wt f_{v_i}, \wt \phi_{v_i})$ for all
$k$-Kottwitz sets of representatives $S$
and $R^V$.

\item For all $v \not\in \{v_1, v_2\} \cup
\Sigma_\inert$, let $(f_v, \phi_{v}) = 
(\wt f_v, \wt\phi_v)$. Note that all such $v$
are either split in $E$ or an element of
$\{v_3, \dots, v_m\}$.
\end{enumerate}

Let $S$ and $R^V$ be compatible sets of global
representatives such that each element of $S$ 
or $R^V$ is Kottwitz at all places of $F$ which
are inert in $E$, 
and $k$-Kottwitz at $v_2, \dots, v_m$.
Such a choice of representatives exists by
Proposition \ref{prop:globalkottwitz}.
Let $(f', \phi')$ be a smooth transfer, as
defined in Definition \ref{defn: global smooth transfer},
of the collection of test functions
\[
    \{(f^{V'}, \phi^{V'} )\}_{V' \in [\SHerm_n^\times(F)]}
\]
where $(f^{V'}, \phi^{V'}) = 0$ 
for $V' \not\simeq V$ and $(f^{V}, \phi^{V}) = (f, \phi)$.
Because $(f_v, \phi_{v})$ is the unramified test 
function at all $v \in \Sigma_{\inert}$, such 
transfer $(f',\phi')$ exist by
Propositions \ref{prop:FL}, \ref{prop:splittransfer}, 
and \ref{prop:Ksplittransfer}.

Then $(f', \phi')$ is also good, so by
Proposition \ref{prop:identity}, 
\begin{align*}
    I_\Pi(f', \phi') 
    &= \sum_{\operatorname{BC}(\pi') = \Pi} J_{\pi'}(f, \phi) 
    \\&= 
    J_{\pi_{v_1}}(f_{v_1},\phi_{v_1})\cdot
    J_{\pi_{v_2}}(f_{v_2}, \phi_{v_2})
    \sum_{\operatorname{BC}(\pi') = \Pi}
    \prod_{v \ne v_1, v_2}
    J_{\pi_v'}(f_v, \phi_v).
\end{align*}
Note that each of the terms in the
sum on the second line is nonnegative because 
$(f_v, \phi_v)$ is of
positive type for $v \neq v_1, v_2$.
Furthermore,
\[
    \prod_{v \neq v_1, v_2} J_{\pi_v}(f_v, \phi_v)
    = \prod_{v \neq v_1, v_2}
    J_{\pi_v}(\wt f_v, \wt \phi_v) \cdot 
    \prod_{v \text{ inert in } E}
    \frac{J_{\pi_v}(f_v, \phi_v)}
    {J_{\pi_v}(\wt f_v, \wt\phi_v)} \neq 0
\]
since $J_{\pi}(\wt f, \wt \phi) \neq 0$ by assumption,
and the product over inert places is a finite
product of nonzero terms.
Thus, the term corresponding to $\pi$ in the sum is
nonzero, so $I_{\Pi}(f', \phi') \neq 0$.
\end{proof}

\begin{proof}[Proof that (1) implies (2)]
\leavevmode
By Theorem \ref{thm:asai},
$L(\frac{1}{2}, \Pi, \As_{L/E} \times \mu^{-1}) \neq 0$ 
implies that there exists 
$\varphi \in \Pi$, $\widetilde \phi' \in \cS(\AA_{E,n})$ such that
\[
    \lambda(\tfrac12, \varphi, \mu, \widetilde \phi') \neq 0.
\]
Furthermore, for each archimedean place $v$ of $F$,
by Lemma \ref{lem:archgood} applied
to the extension $K/F$, we can assume that
the function $\wt\phi_v'$ is 
of the form
$\phi_{1,v}' \otimes \phi_{2, v}'$ where
$\phi_{1,v}', \phi_{2,v}' \in \cS(F_{v,n})$.
Since $\Pi$ is Hermitian, by \cite{MR1344660}, 
there also exists 
$\varphi' \in \Pi$ such that $\beta(\varphi') \neq 0$.
Thus, there exists factorizable $\widetilde f'$
such that 
\[
    I_\Pi(\widetilde f', \widetilde \phi') \neq 0.
\]

As before, we define a test function
\[
    (f', \phi') = (\otimes_v f'_v, \otimes_v \phi_v'),
\]
which is good in the sense of 
Definition \ref{def:glgood}, by modifying 
$(\wt f', \wt\phi')$.
Let $\Sigma_{\inert}$ and $\{v_3, \dots, v_m\}$ 
be as in the proof of (2) implies (1).
As before, we will also define a constant $k$ used
for choosing $k$-Kottwitz representatives.
\begin{enumerate}[(1)]
\item
For all $v \in \Sigma_{\inert}$,
let $(f_v', \phi_v')$ be the unramified test function.
Since $\Pi_v$ is unramified, we have
$I_{\Pi_v}(f_v', \phi_v') \neq 0$.

\item By Lemma \ref{lem:splitlocaldistrid}, $I_{\Pi_v}(f_v', \phi_v') \neq 0$ implies that
$J_{\pi_v}$ is nonzero for every place
of $F$ which is split in $E$.
Thus by Lemma \ref{lem:matrixcoeff}, there exists
a test function $(f_{v_1}, \phi_{v_1})$
on the unitary side such that
$J_{\pi_{v_1}}(f_{v_1}, \phi_{v_1}) \neq 0$ and 
$f_{v_1}$ is a
truncated matrix coefficient of a supercuspidal
representation.
Let $(f_{v_1}', \phi_{v_1}')$ be the smooth transfer
of $(f_{v_1}, \phi_{v_1})$.

\item By Lemma \ref{lem:rsssupp}, 
there exists a test function $(f_{v_2}, \phi_{v_2})$
on the unitary side and an integer $k \ge 0$ such
that $J_{\pi_{v_2}}(f_{v_2}, \phi_{v_2}) \neq 0$
and condition \eqref{ursssupp} of Definition
\ref{def:ugood} is satisfied for all $k$-Kottwitz
sets of representatives $R^V$.
Let $(f_{v_2}', \phi_{v_2}')$ be the smooth transfer
of $(f_{v_2}, \phi_{v_2})$.

By Proposition \ref{prop:Ksplittransfer},
we can suppose $k$ is large enough so that for $3 \le i \le m$,
there exists a test function $(f_{v_i}, \phi_{v_i})$
on the unitary side which is a smooth transfer
of $(\wt f_{v_i}', \wt \phi_{v_i}')$ for all $k$-Kottwitz
sets of representatives $S$ and $R^V$.

\item For all $v \not\in \{v_1, v_2\}
\cup \Sigma_\inert$, let $(f_v', \phi_v') = (\wt f_v', \wt \phi_v')$.
\end{enumerate}

Let $S$ and $R^V$ be compatible sets
of representatives which are Kottwitz at all inert
places and $k$-Kottwitz at $v_2, \dots, v_m$.
Let 
\[
    \{(f^V, \phi^V)\}_{V \in [\SHerm_n^\times(F)]}
\] 
be a smooth transfer of $(f', \phi')$ as
in Definition \ref{defn: global smooth transfer}. 
Such a transfer exists locally at each place
by Propositions \ref{prop:FL}, 
\ref{prop:splittransfer}, and 
\ref{prop:Ksplittransfer}.

Then Proposition \ref{prop:identity} tells us that
\[
    I_\Pi(f', \phi') 
    = \sum_{V}
    \sum_{\operatorname{BC}(\pi) = \Pi} J_{\pi}(f^{V}, \phi^{V}) \neq 0.
\]
This implies that the period integral
$\P$ is nonzero for some skew-Hermitian space $V$ over $E$ and
some $\pi$ such that $\operatorname{BC}(\pi) = \Pi$.
\end{proof}

% BIBLIOGRAPHY

\bibliographystyle{alpha}
\bibliography{reference}

\end{document}